\documentclass[12pt,a4paper,reqno]{amsart}
\usepackage[utf8]{inputenc}
\usepackage[T1]{fontenc}
\usepackage{indentfirst}
\usepackage[english]{babel}
\usepackage{enumitem}
\usepackage{csquotes}
\usepackage[style=alphabetic,sorting=nyt,maxcitenames=4,maxbibnames=5]{biblatex}
\usepackage{amsmath,amsfonts,amssymb,amsopn,amscd,amsthm}
\usepackage{mathrsfs,bbm}
\usepackage{graphicx}
\usepackage{mathpazo}
\usepackage[dvipsnames]{xcolor}
\usepackage[colorlinks=true,citecolor=DarkOrchid,linkcolor=NavyBlue]{hyperref}
\usepackage{pgfplots}
\usepackage{tikz}
\usetikzlibrary{patterns}
\pgfplotsset{compat=1.17}
\usepackage[a4paper,twoside,margin=0.8in]{geometry}
\newtheorem{definition}{Definition}[section]
\newtheorem{lemma}[definition]{Lemma}
\newtheorem{theorem}[definition]{Theorem}
\newtheorem{proposition}[definition]{Proposition}
\newtheorem{corollary}[definition]{Corollary}
\theoremstyle{remark}
\newtheorem{example}[definition]{Example}
\newtheorem{remark}[definition]{Remark}
	
\newcommand{\N}{\mathbb{N}}
\newcommand{\Z}{\mathbb{Z}}
\newcommand{\R}{\mathbb{R}}
\newcommand{\C}{\mathbb{C}}
\newcommand{\E}{\mathrm{e}}
\newcommand{\I}{\mathrm{i}}
\newcommand{\sph}{\mathbb{S}}
\newcommand{\gauss}[3]{\mathcal{N}_{\R^{#1}}(#2,#3)}
\newcommand{\esper}{\mathbb{E}}
\newcommand{\proba}{\mathbb{P}}

\newcommand{\leb}{\mathrm{L}}
\newcommand{\dkol}{d_\mathrm{Kol}}
\newcommand{\dconv}{d_\mathrm{convex}}

\newcommand{\sinc}{\mathrm{sinc}}
\newcommand{\card}{\mathrm{card}}
\newcommand{\vol}{\mathrm{vol}}
\newcommand{\cov}{\mathrm{cov}}
\newcommand{\eps}{\varepsilon}
\newcommand{\VEC}{\mathbf{V}}
\newcommand{\SEC}{\mathbf{S}}
\newcommand{\XEC}{\mathbf{X}}
\newcommand{\xec}{\mathbf{x}}
\newcommand{\YEC}{\mathbf{Y}}
\newcommand{\yec}{\mathbf{y}}

\newcommand{\zec}{\mathbf{z}}
\newcommand{\XIEC}{\boldsymbol\xi}
\newcommand{\ZIEC}{\boldsymbol\zeta}
\newcommand{\EX}[1]{^{(#1)}}
\newcommand{\SO}{\mathrm{SO}}

\renewcommand{\Re}{\mathrm{Re}}
\renewcommand{\Im}{\mathrm{Im}}
\newcommand{\DD}[1]{\,d\hspace*{-0.3mm}{#1}}
\newcommand{\lle}{\left[\!\left[} 
\newcommand{\rre}{\right]\!\right]} 
\newcommand{\scal}[2]{\left\langle #1\vphantom{#2}\,\right |\left.#2 \vphantom{#1}\right\rangle}
\newcommand{\comment}[1]{}

\setlist[enumerate]{itemsep=10pt,topsep=10pt}
\setlist[itemize]{itemsep=5pt,topsep=5pt}

\title{Gaussian approximations for random vectors}
\author{P.-L. Méliot \and A. Nikeghbali}
\date{\today}

\begin{document}

\begin{abstract}
We present several refinements on the fluctuations of sequences of random vectors (with values in the Euclidean space $\mathbb{R}^d$) which converge after normalization to a multidimensional Gaussian distribution. More precisely we refine such results in two directions: first we give conditions under which one can obtain bounds on the speed of convergence to the multidimensional Gaussian distribution, and then we provide a setting in which one can obtain precise moderate or large deviations (in particular we see at which scale the Gaussian approximation for the tails ceases to hold and how the symmetry of the Gaussian tails is then broken). These results extend some of our earlier works obtained for real valued random variables, but they are not simple extensions, as some new phenomena are observed that could not be visible in one dimension. Even for very simple objects such as the symmetric random walk in $\mathbb{Z}^d$, we observe a loss of symmetry that we can quantify for walks conditioned to be far away from the origin. Also, unlike the one dimensional case where the Kolmogorov distance is natural, in the multidimensional case there is no more such a canonical distance. We choose to work with the so-called convex distance, and as a consequence the geometry of the Borel measurable sets that we consider shall play an important role (also making the proofs more complicated). We illustrate our results with some examples such as correlated random walks, the characteristic polynomials of random unitary matrices, or pattern countings in random graphs.
\end{abstract}
\maketitle
\tableofcontents
\clearpage

\section{Introduction}
Let $d\geq 1$ be a positive integer, and $(\XEC_n)_{n \in \N}$ be an arbitrary sequence of random vectors with values in the finite-dimensional vector space $\R^d$. For many probabilistic models, there exists a renormalisation $\YEC_n = \alpha_n\,\XEC_n$ that admits a limit in law: $\YEC_n \rightharpoonup_{n \to \infty} \YEC$. For instance, if $\XEC_n = \sum_{i=1}^n \mathbf{A}_i$ is a sum of independent and identically distributed random vectors with a second moment, then the multi-dimensional central limit theorem ensures the convergence in law of 
$$\YEC_n = \frac{1}{\sqrt{n}} \left(\XEC_n - n\,\esper[\mathbf{A}_1]\right)$$
to a Gaussian distribution with covariance matrix $K = (\cov(\mathbf{A}_1\EX{i},\mathbf{A}_1\EX{j}))_{1\leq i,j\leq d}$; see \cite[Section VIII.4, Theorem 2]{Fel71} or \cite[Example 2.18]{Vaart00}. Assuming that $K$ is non-degen\-erate, this means that for any Borel set $A \subset \R^d$ whose topological boundary has zero Lebesgue measure,
\begin{equation}
\lim_{n \to \infty} \proba[\YEC_n \in A] = \frac{1}{\sqrt{(2\pi)^{d}\det K}}\int_{\R^d} 1_{\xec \in A}\,\E^{-\frac{\xec^t K^{-1}\xec}{2}} \DD{\xec}.\label{eq:clt}
\end{equation}
There are various ways to go beyond this convergence in law. A first direction consists in describing what happens at the edge of the distribution of $\YEC_n$, that is with Borel subsets $A=A_n$ that vary with $n$ and that grow in size. For instance, one can try to estimate the probability of $A$ being in a spherical sector
$$\proba\!\left[\YEC_n \in [R_n,+\infty)\,B\right],$$
where $B$ is a Borel subset of the Euclidean sphere $\sph^{d-1}$, and $R_n \to +\infty$. If $(R_n)_{n \in \N}$ does not grow too fast, one can guess that the normal approximation will still be valid, whereas if $R_n$ is too large, it will need to be corrected. These results belong to the theory of \emph{large deviations}, and for sums of i.i.d.~random variables, the large deviations are described by Cramer's theorem \cite{Cramer38}; see \cite[Section 2.2.2]{DZ98} for an exposition of the multi-dimensional version of this result.\medskip

Another way to make precise Equation \eqref{eq:clt} is by computing bounds on the difference between $\proba[\YEC_n \in A]$ and its limit. In the one-dimensional case, if $A=[a,b]$ is an interval and $(A_i)_{i\geq 1}$ is as before a sequence of i.i.d.~random variables, then the \emph{Berry--Esseen estimates} \cite{Berry41,Ess45} yield
$$\left|\proba[Y_n \in [a,b]] - \frac{1}{\sqrt{2\pi\, \mathrm{var}(A_1)}}\int_a^b \E^{-\frac{x^2}{2\,\mathrm{var}(A_1)}} \DD{x}\right| \leq \frac{C}{\sqrt{n}},$$
assuming that $A_1$ has a moment of order $3$; see \cite[Section XVI.5]{Fel71}. The multi-dimensional analogue of this result is much less straightforward: one way to state it is by replacing intervals by convex sets, and this approach is followed in \cite{BR10}.\bigskip

In two previous works \cite{FMN16,FMN19}, we proposed a general framework that enabled us to prove large deviation results and estimates on the speed of convergence for sequences $(X_n)_{n\in \N}$ of one-dimensional random variables: the framework of \emph{mod-$\phi$ convergence}. The interest of this theory is that it allowed us to deal with examples that were much more general than sums of i.i.d.~random variables. In this setting we obtained asymptotic results for statistics of random graphs, functionals of Markov chains, arithmetic functions of random integers, characteristic polynomials of random matrices, \emph{etc.}; see the aforementioned works and \cite{BMN19,CDMN20,FMN20,MN22}. In this article, we shall extend this theory to a multi-dimensional setting. This extension is not trivial, as new phenomenons that cannot occur in dimension $1$ take place when $d \geq 2$. In the next paragraphs of this introduction, we fix our notation and we present the hypotheses of multi-dimensional mod-Gaussian convergence. We then give an outline of the main results that we shall establish under these hypotheses (Subsection \ref{subsec:outline}).
\medskip

\subsection{Notational conventions}
We start by fixing some notation that will be used throughout the article. Since we shall work with sequences of random vectors, we need to pay special care to the exponents and indices of the quantities that we manipulate. We shall use  the following conventions:
\begin{itemize}
	\item The numbers in $\R$ are written with regular letters $x,y,X,Y,\ldots$; vectors in $\R^d$ are written with bold letters $\xec,\yec,\XEC,\YEC,\ldots$
	\item Random numbers $X,Y,\ldots$ and random vectors $\XEC,\YEC,\ldots$ are indicated by capital letters. We shall also use capital letters for matrices.
	\item The coordinates of a  vector $\xec$ are written as exponents in parentheses: $$\xec = (x\EX{1},x\EX{2},\ldots,x\EX{d}).$$
	\item The sequences of (random) numbers or vectors are labelled by an index $n\in \N$; thus, for instance, $(\VEC_n)_{n \in \N}$ is a sequence of random vectors in $\R^d$. The previous convention allows one to avoid any ambiguity between the label of a coordinate (exponent) and the label of the sequence (index).
	\item Without loss of generality, we assume that our (random) vectors  take their values in $\R^d$ with $d \geq 2$. Though all our results hold also when $d=1$, one can give shorter, and sometimes more precise proofs in this particular case, see our previous works \cite{FMN16,FMN19}.
\end{itemize}
The vector space $\R^d$ is endowed with the three equivalent norms
\begin{align*}
\|\xec\|_1 &= \sum_{i=1}^d |x\EX{i}|;\\
\|\xec\|_2 &= \sqrt{\sum_{i=1}^d |x\EX{i}|^2};\\
\|\xec\|_\infty &=\max_{i \in \lle 1,d \rre} |x\EX{i}|,
\end{align*}
and we denote: 
\begin{itemize}
	\item $B_{(\xec,\eps)}^d = \{\yec \in \R^d\,|\, \|\xec-\yec\|_2\leq \eps\}$ the Euclidean ball of radius $\eps$ and center $\xec$;
	\item and $C_{(\xec,\eps)}^d = \{ \yec \in \R^d\,|\, \|\xec-\yec\|_\infty \leq \eps\} = \prod_{i=1}^d [x\EX{i}-\eps,x\EX{i}+\eps]$ the hypercube of edge $2\eps$ and centered at $\xec$.
\end{itemize}
As we shall mostly use the Euclidean norm, we simply note $\|\xec\|_2=\|\xec\|$. The corresponding Euclidean distance is denoted by $\|\xec-\yec\|=d(\xec,\yec)$. Finally, $\sph^{d-1}$ is the unit sphere $\{\xec \in \R^d\,|\,\|x\|=1\}$.\bigskip

If $\XEC$ is a random vector in $\R^d$, we denote its complex Laplace transform
$$ \varphi_{\XEC}(\zec) = \esper[\E^{\scal{\zec}{\XEC}}]=\esper\!\left[\E^{\sum_{i=1}^d z\EX{i}\,X\EX{i}}\right]\quad \text{for }\zec= (z\EX{1},\ldots,z\EX{d})\in \C^d.$$
This quantity might  not be well defined for certain values of $\zec$, for instance if the coordinates of $\zec$ have too large real parts; on the other hand the Fourier transform
$$\phi_{\XEC}(\ZIEC) = \varphi_{\XEC}(\I\ZIEC) = \esper\!\left[\E^{\I\sum_{i=1}^d \zeta\EX{i}\,X\EX{i}}\right]$$
is well defined for any vector $\ZIEC$ in $\R^d$. In this article, we shall compare the distribution of random vectors $\VEC_n$ with a reference Gaussian distribution $\gauss{d}{\mathbf{m}}{K}$; let us recall briefly what needs to be known about them. We denote $\mathrm{S}_+(d,\R)$ the set of positive-definite symmetric matrices of size $d \times d$, and we fix $\mathbf{m}\in \R^d$ and $K \in \mathrm{S}_+(d,\R)$. The Gaussian distribution with mean $\mathbf{m}$ and covariance matrix $K$ on $\R^d$ is the probability distribution $\gauss{d}{\mathbf{m}}{K}$ whose density with respect to Lebesgue measure is
$$ \frac{1}{\sqrt{(2\pi)^{d}\,\det K}}\,\E^{-\frac{(\xec-\mathbf{m})^tK^{-1}(\xec-\mathbf{m})}{2}}\DD{\xec}.$$ 
The Laplace transform of a random vector $\XEC$ with Gaussian law $\gauss{d}{\mathbf{m}}{K}$ is given over the complex domain $\C^d$ by the formula
$$\varphi_\XEC(\zec) = \E^{\scal{\mathbf{m}}{\zec}+\frac{\zec^t K \zec}{2}}.$$
In particular, the Fourier transform of $\XEC$ is $\phi_\XEC(\ZIEC) = \E^{\I \scal{\mathbf{m}}{\ZIEC}-\frac{\ZIEC^tK\ZIEC}{2}}$ for any vector $\ZIEC$ in $\R^d$. In our computations, the positive eigenvalues $k\EX{1} \leq k\EX{2} \leq \cdots \leq k\EX{d}$ of $K \in \mathrm{S}_+(d,\R)$ will play an important role. For any vector $\xec \in \R^d$,
$$k\EX{1}\|\xec\|^2 \leq \xec^tK\xec \leq k\EX{d} \|\xec\|^2.$$
The largest eigenvalue $k\EX{d}$ is the spectral norm $\rho(K^{1/2})$ of $K^{1/2}$, and the smallest eigenvalue $k\EX{1}$ is related to the spectral norm of $K^{-1/2}$ by the relation $k\EX{1} = (\rho(K^{-1/2}))^{-1}$.
\medskip

\subsection{Mod-Gaussian convergence}\label{subsec:modgauss}
We now present our main hypothesis on the sequences of random vectors $(\XEC_n)_{n \in \N}$.

\begin{definition}[Mod-Gaussian convergence]\label{def:modgauss}
Let $(\XEC_n)_{n \in \N}$ be a sequence of random vectors in $\R^d$, and $K \in \mathrm{S}_+(d,\R)$. We say that the sequence is mod-Gaussian convergent in the Laplace sense on $\C^d$ with parameters $t_n K$, $t_n \to +\infty$, and limiting function $\psi$ if, locally uniformly on compact sets of $\C^d$,
$$ \esper[\E^{\scal{\zec}{\XEC_n}}]\,\,\exp\!\left(-t_n\,\frac{\zec^t K \zec}{2}\right) = \psi_n(\zec) \to \psi(\zec).$$
Here, $\psi(\zec)$ is a continuous function of $\zec \in \C^d$, with $\psi(\mathbf{0})=1$. We say that the sequence is mod-Gaussian convergent in the Fourier sense on $\R^d$ with the same parameters if the convergence takes place on $\R^d$, that is to say, if there exists a continuous function $\theta(\ZIEC)$ of $\ZIEC \in \R^d$ such that
$$ \esper[\E^{\I\scal{\ZIEC}{\XEC_n}}]\,\,\exp\!\left(t_n\,\frac{\ZIEC^t K \ZIEC}{2}\right) = \theta_n(\ZIEC) \to \theta(\ZIEC)$$
locally uniformly on compact sets of $\R^d$.
\end{definition}
\medskip

\begin{remark}
For the mod-Gaussian convergence in the Laplace sense, it might happen that the convergence only holds on a domain $D \subset \C^d$ of the form $D=\mathcal{S}_{\mathbf{a},\mathbf{b}}$, where $\mathcal{S}_{\mathbf{a},\mathbf{b}}$ is the multi-strip $\mathcal{S}_{(a\EX{1},b\EX{1})} \times \cdots \times \mathcal{S}_{(a\EX{d},b\EX{d})}$, with
$$ \mathcal{S}_{(a,b)} = \{z \in \C\,|\,a<\Re(z) <b\},\quad a,b \in \R \sqcup \{\pm \infty\}. $$
If so, we shall still speak of mod-Gaussian convergence, but making precise this domain of convergence.
\end{remark}
\medskip

In the one-dimensional case, the notion of mod-Gaussian convergence was introduced in \cite{JKN11}; the multi-dimensional definition first appeared in \cite{KN12}. Let us give two important examples:

\begin{example}[Sums of i.i.d.~random vectors]
Consider again a sequence of i.i.d.~random vectors $(\mathbf{A}_i)_{i\geq 1}$, which we assume to be centered, with non-degenerate covariance matrix and with a third moment:
$$\esper[\mathbf{A}_1]=0\qquad;\qquad \left(\cov(\mathbf{A}_1\EX{i},\mathbf{A}_1\EX{j})\right)_{1\leq i,j\leq d} \in \mathrm{S}_+(d,\R) \qquad;\qquad \esper[\|\mathbf{A}_1\|^3]<+\infty.$$
A Taylor expansion of the Fourier transform of $\mathbf{A}_1$ shows that, if 
$$ \XEC_n = \frac{1}{n^{1/3}} \sum_{i=1}^n \mathbf{A}_i,$$
then $(\XEC_n)_{n \in \N}$ is mod-Gaussian convergent in the Fourier sense, with parameters $n^{1/3}\,K$ with $K=\cov(\mathbf{A}_1)$, and limit
$$\theta(\ZIEC) = \exp\!\left(-\frac{\I}{6}\sum_{i,j,k=1}^d \esper[\mathbf{A}_1\EX{i}\mathbf{A}_1\EX{j}\mathbf{A}_1\EX{k}]\,\ZIEC\EX{i}\ZIEC\EX{j}\ZIEC\EX{k}\right).$$
If $\mathbf{A}_1$ has a convergent Laplace transform, then one has in fact a mod-Gaussian convergence in the Laplace sense on $\C^d$.
\end{example}

\begin{example}[Characteristic polynomials of random unitary matrices]
Let $U_n$ be a random matrix in the unitary group $\mathrm{U}(n)$, taken according to the Haar measure of this Lie group. We set 
$$\XEC_n = \log \det(I_n-U_n) = \sum_{i=1}^n \log (1-\E^{\I\theta_i}),$$
where $\E^{\I\theta_1},\ldots,\E^{\I\theta_n}$ are the eigenvalues of $U_n$, all belonging to the unit circle. Notice that on the open unit disk $\{z\,|\,|z|<1\}$, the complex logarithm is given by its Taylor series:
$$\log (1-z) = -\sum_{k=1}^\infty \frac{z^k}{k}.$$
We extend the definition to the unit circle by setting $\log(1-\E^{\I\theta}) = \lim_{r \to 1} \log (1-r\E^{\I \theta});$ the limit exists and is finite if $\E^{\I \theta} \neq 1$. Since the eigenvalues are all different from $1$ with probability $1$ on $\mathrm{U}(n)$, $\XEC_n$ is correctly defined. The random variables $\XEC_n$ take their values in $\C$, which we identify with $\R^2$: $\XEC_n=\XEC_n\EX{1} + \I \XEC_n\EX{2}$. The Laplace transform of $\XEC_n$ is given by the Selberg integral
$$\esper\!\left[\E^{z\EX{1} \XEC_n\EX{1}+z\EX{2}\XEC_n\EX{2}}\right] = \prod_{j=1}^n \frac{\Gamma(j)\,\Gamma(j+z\EX{1})}{\Gamma\!\left(j+\frac{z\EX{1}+\I z\EX{2}}{2}\right)\,\Gamma\!\left(j + \frac{z\EX{1}-\I z\EX{2}}{2}\right)}$$
for $\Re(z\EX{1})>-1$ and $|\Im (z\EX{2})|<1$; see \cite[Formula (71)]{KN00}. Since the $\Gamma$ function has no zero on the complex plane, the right-hand side of this formula is a biholomorphic function on the multi-strip $\mathcal{S}_{(-1,+\infty)} \times \C$, and therefore, the formula above for the Laplace transform of $\XEC_n$ actually holds without restriction on the imaginary part of $z\EX{2}$. Let us introduce Barnes' $G$-function 
$$G(z) = (2\pi)^{\frac{z}{2}}\,\exp\!\left(-z+\frac{z^2(1+\gamma)}{2}\right)\,\prod_{k=1}^\infty \left(1+\frac{z}{k}\right)^k\,\exp\!\left(\frac{z^2}{2k}-z\right),$$
$\gamma$ being the Euler constant. The function $G(z)$ is entire, and it is the solution of the functional equation 
$$G(1)=1\quad;\quad G(z+1) = G(z)\,\Gamma(z).$$
Then, one can rewrite the Laplace transform of $\XEC_n$ as
\begin{align*}
&\esper\!\left[\E^{z\EX{1} \XEC_n\EX{1}+z\EX{2} \XEC_n\EX{2}}\right] \\
& = \frac{G\!\left(1+\frac{z\EX{1}+\I z\EX{2}}{2}\right)\,G\!\left(1+\frac{z\EX{1}-\I z\EX{2}}{2}\right)}{G(1+z\EX{1})}\,\,\frac{G(n+1)\,G(z\EX{1}+n+1)}{G\!\left(\frac{z\EX{1}+\I z\EX{2}}{2}+n+1\right)\,G\!\left(\frac{z\EX{1}-\I z\EX{2}}{2}+n+1\right)}\\
& = \frac{G\!\left(1+\frac{z\EX{1}+\I z\EX{2}}{2}\right)\,G\!\left(1+\frac{z\EX{1}-\I z\EX{2}}{2}\right)}{G(1+z\EX{1})}\,\, n^{\frac{(z\EX{1})^2+(z\EX{2})^2}{4}}\,(1+o(1)),
\end{align*}
see the details in \cite[Section 3]{KN12} for the asymptotics of the ratio of Barnes' functions. Therefore, the sequence of random vectors $(\XEC_n)_{n \in \N}$ is mod-Gaussian convergent in the Laplace sense, on the domain $\mathcal{S}_{(-1,+\infty)} \times \C$, with parameters $t_n I_2= \frac{\log n}{2} \,I_2$ and limiting function 
$$ \psi(\zec) = \frac{G\!\left(1+\frac{z\EX{1}+\I z\EX{2}}{2}\right)\,G\!\left(1+\frac{z\EX{1}-\I z\EX{2}}{2}\right)}{G(1+z\EX{1})}.$$
The real part of $\XEC_n$ has been studied extensively in \cite{MN22}; its mod-Gaussian convergence leads to asymptotic formul{\ae} for $\proba[\Re(\XEC_n) \geq x_n]$ for $x_n$ in a large range of values, up to $O(n)$. The techniques that we shall develop hereafter only deal with the fluctuations of the random vectors $\XEC_n$ in plane domains of size $O(\log n)$, but we shall already see in this setting phenomena that are specific to the multi-dimensional mod-Gaussian convergence.
\end{example}
\medskip

Many other examples of mod-Gaussian convergent sequences will be provided in Section \ref{sec:examples}. In the two previous examples, an adequate renormalisation $\YEC_n$ of $\XEC_n$ admits a limit in law which is a Gaussian distribution. This is a consequence of the following general simple statement:

\begin{proposition}[Central limit theorem]\label{prop:clt}
Let $(\XEC_n)_{n \in \N}$ be a sequence of random vectors that is mod-Gaussian convergent in the Fourier sense with parameters $t_n K$. The rescaled random variables
$$\YEC_n = \frac{\XEC_n}{\sqrt{t_n}}$$
converge in law towards $\gauss{d}{\mathbf{0}}{K}$.
\end{proposition}

\begin{proof}
The Fourier transform of $\YEC_n$ is
$$\phi_{\YEC_n}(\ZIEC) = \phi_{\XEC_n}\!\left(\frac{\ZIEC}{\sqrt{t_n}}\right) = \E^{-\frac{\ZIEC^t K\ZIEC}{2}}\,\theta_n\!\left(\frac{\ZIEC}{\sqrt{t_n}}\right),$$
and the local uniform convergence of the residue $\theta_n$ ensures that $\theta_n(\ZIEC/\sqrt{t_n})$ converges towards $\theta(\mathbf{0})=1$ for any $\ZIEC \in \R^d$. The convergence of the characteristic functions imply the convergence in law by Lévy's continuity theorem.
\end{proof}
\medskip

The purpose of our article is to improve on Proposition \ref{prop:clt} in the two aforementioned directions: estimates of the speed of convergence (Section \ref{sec:speed}) and large deviation results (Section \ref{sec:largedeviation}). Let us mention that a third direction, which is opposite to the one of large deviations and concern estimates of probabilities $\proba[\YEC_n \in A_n]$ with Borel subsets $A_n$ of size going to zero, can be pursued in the framework of multi-dimensional mod-Gaussian convergence. One then obtains \emph{local limit theorems}: see \cite{KN12,DKN15,BMN19}.

\begin{remark}
Let $(\XEC_n)_{n \in \N}$ be a sequence of random vectors that converges mod-Gaussian, for instance in the Laplace sense on $\C^n$, with parameters $t_n K$ and limiting function $\psi(\zec)$. Then, if $\YEC_n = K^{-1/2} \XEC_n$, one has
\begin{align*}
\esper\!\left[\E^{\scal{\zec}{\YEC_n} }\right] &= \esper\!\left[\E^{\scal{K^{-1/2}\zec}{ \XEC_n} }\right]\\
& = \E^{t_n\frac{(K^{-1/2}\zec)^t \,K\, (K^{-1/2}\zec) }{2}}\,\psi_n(K^{-1/2}\zec)\\
& = \E^{t_n\frac{\|\zec\|^2}{2} }\, \psi(K^{-1/2}\zec)\,(1+o(1)).
\end{align*}
So, $(\mathbf{Y}_n)_{n \in \N}$ converges in the mod-Gaussian sense with parameters $t_n I_d$ and limiting function $\psi(K^{-1/2}\zec)$. Thus, every multi-dimensional mod-Gaussian convergence in the sense of Definition \ref{def:modgauss} can be assumed to have as parameters $t_n$ times the identity matrix, $(t_n)_{n \in \N}$ being a sequence increasing to $+\infty$. However, this reduction is not necessarily the most interesting thing to do. Indeed, most of the time, the matrix $t_n K$ is given by the first order asymptotics of the covariance matrix of $\XEC_n$; and it will be more convenient to deal with the joint moments of the coordinates of $\XEC_n$ than with those of $\YEC_n$.
\end{remark}

\begin{remark}
In \cite{FMN16,FMN19} which focused on the one-dimensional case, we worked in a framework which was much more general than mod-\emph{Gaussian} convergence. Hence, for any \emph{infinitely divisible distribution} $\phi$ on $\R$, one can define mod-$\phi$ convergence by replacing in our Definition \ref{def:modgauss} the Lévy exponent $-\zeta^2/2$ of the Gaussian law by an arbitrary Lévy exponent of an infinitely divisible distribution; see \cite[Chapter 2]{Sato99} for an exposition of the Lévy--Khintchine formula. The same thing can be done in dimension $d \geq 2$, since the infinitely divisible distributions have exactly the same classification in dimension $1$ and in higher dimensions \cite[Theorem 8.1]{Sato99}. However, if one wants in this setting large deviation estimates, then the techniques that we shall use require the infinitely divisible reference law to have a convergent Laplace transform on the complex plane; and this happens only for Gaussian distributions. Similarly, if one wants to have estimates on the speed of convergence, then we shall need sufficiently many moments for the reference law, and again we shall be restricted to the Gaussian case. This explains why we only focus on mod-Gaussian convergence when $d \geq 2$. Also, we do not know of many interesting examples of multi-dimensional mod-$\phi$ convergent sequences with $\phi \neq \gauss{d}{\mathbf{0}}{K}$. One interesting example would be related to the convergence in distribution of the marginal law of a process of Ornstein--Uhlenbeck type towards a self-decomposable distribution, see \cite[Chapter 3, Section 17]{Sato99}.
\end{remark}
\medskip

\subsection{The one-dimensional case}
Let us briefly recall  the results of \cite{FMN16,FMN19} that we aim to extend to higher dimensions. We start with the large deviation results, see \cite[Theorem 4.2.1 and Proposition 4.4.1]{FMN16}:
\begin{theorem}[Large deviations in dimension $d=1$]\label{thm:1Dlargedeviation}
Let $(X_n)_{n \in \N}$ be a sequence of real-valued random variables that is mod-Gaussian convergent in the Laplace sense, with domain $\mathcal{S}_{(a,b)}$ with $a<0<b$, parameters $(t_n)_{n \in \N}$ and limit $\psi(z)$. We set $Y_n = \frac{X_n}{\sqrt{t_n}}$; $Y_n \rightharpoonup \gauss{}{0}{1}$. For any $x>0$, 
$$\proba[X_n \geq t_nx]=\proba\!\left[Y_n \geq \sqrt{t_n}x\right] = \frac{\E^{-\frac{t_nx^2}{2}}}{x\sqrt{2\pi t_n}}\,\psi(x)\,(1+o(1)).$$
The estimate remains true as soon as $y=\sqrt{t_n}x$ is a $O(\sqrt{t_n})$ and grows to infinity. If $y =o(\sqrt{t_n})$, then the correction $\psi(x)$ to the Gaussian tail is asymptotic to $1$ and does not appear.
\end{theorem}
\noindent Thus, up to the scale $y=o(\sqrt{t_n})$, the probability $\proba[Y_n \geq y]$ is asymptotically the same as the Gaussian tail; one says that $o(\sqrt{t_n})$ is the \emph{normality zone} of the sequence $(Y_n)_{n \in \N}$. When $y=O(\sqrt{t_n})$, this is not true anymore and the limiting residue $\psi$ of mod-Gaussian convergence appears as a correction of the Gaussian tail.\bigskip

To get estimates of Berry--Esseen type, one only needs a mod-Gaussian convergence in the Fourier sense, but sometimes with precisions on the speed and zone of convergence of 
\begin{equation}
\theta_n(\zeta) = \esper[\E^{\I \zeta X_n}]\,\E^{\frac{t_n\zeta^2}{2}} \label{eq:residue}
\end{equation}
towards its limit $\theta(\zeta)$. Let us start with a result without additional hypotheses \cite[Proposition 4.1.1]{FMN16}:
\begin{theorem}[General Berry--Esseen estimates in dimension $d=1$]\label{thm:1Dgeneralberryesseen}
Let $(X_n)_{n \in \N}$ be a sequence of real-valued random variables that is mod-Gaussian convergent in the Fourier sense, with parameters $(t_n)_{n \in \N}$ and limit $\theta(\zeta)$. We set $Y_n = \frac{X_n}{\sqrt{t_n}}$. Then,
$$\dkol(Y_n,\gauss{}{0}{1})=\sup_{y\in \R}\left|\proba[Y_n \leq y] - \frac{1}{\sqrt{2\pi}}\int_{-\infty}^y \E^{-\frac{x^2}{2}}\DD{x}\right| = O\!\left(\frac{1}{\sqrt{t_n}}\right).$$
\end{theorem}

This result is usually not optimal, and one can get better estimates by controlling the size of $\theta_n(\zeta)$ on a zone that grows with $n$. The precise hypotheses are the following \cite[Theorem 15]{FMN19}:

\begin{theorem}[Berry--Esseen estimates with a zone of control]\label{thm:1Dzoneofcontrol}
Let $(X_n)_{n \in \N}$ be a sequence of real-valued random variables. We suppose that there exists a zone $[-K(t_n)^\gamma,K(t_n)^\gamma]$, such that the residue given by Equation \eqref{eq:residue} satisfies for any $\zeta$ in this zone:
$$|\theta_n(\zeta)-1|\leq K_1|\zeta|^v\,\exp(K_2|\zeta|^w),$$
with $w \geq 2$, $v>0$ and $-\frac{1}{2}<\gamma \leq \min(\frac{v-1}{2},\frac{1}{w-2})$. Then, there is a constant $C(\gamma,K,K_1,K_2)$ such that
$$\dkol(Y_n,\gauss{}{0}{1}) \leq C(\gamma,K,K_1,K_2)\,\frac{1}{(t_n)^{\frac{1}{2}+\gamma}}.	$$
\end{theorem}

\noindent In many situations including sums of i.i.d.~random variables, one has a zone of control with $v=w=3$ and $\gamma=1$, leading to a Berry--Esseen estimate of size $O(\frac{1}{(t_n)^{3/2}})$.
\medskip

Our goal will be to give analogues of Theorems \ref{thm:1Dlargedeviation}, \ref{thm:1Dgeneralberryesseen} and \ref{thm:1Dzoneofcontrol} in higher dimensions. Our hypotheses when $d\geq 2$ will be extremely similar to the previous ones, but the results that we shall obtain are somewhat more subtle, and their proofs will be much more complicated. \medskip

\subsection{Main results and outline of the article}\label{subsec:outline}
In Section \ref{sec:speed}, we focus on the speed of convergence, and we introduce a distance between probability measures on $\R^d$ that is akin to the Kolmogorov distance, and that can be controlled under the assumption of mod-Gaussian convergence. This distance allows one to compare the distribution of $\YEC_n$ and the Gaussian limit $\gauss{d}{\mathbf{0}}{K}$ on any Borel measurable \emph{convex} subset of $\R^d$. In general, we obtain a bound of order $O((t_n)^{-1/2})$, see Theorem \ref{thm:generalberryesseen}. However, if the limit $\psi$ of mod-Gaussian convergence has a certain form, and if one knows more about the convergence of the residues $\psi_n(\zec)\to \psi(\zec)$, then one can prove bounds of order $o((t_n)^{-1/2})$, or even $O((t_n)^{-3/2})$ (Theorems \ref{thm:modifiedberryesseen} and \ref{thm:superspeed}).\medskip

In Section \ref{sec:largedeviation}, we study the large deviations of a mod-Gaussian convergent sequence of random vectors. In the multi-dimensional setting, the open intervals $t_n[b,+\infty)$ of Theorem \ref{thm:1Dlargedeviation} are replaced by spherical or ellipsoidal sectors $t_n ( S \times [b,+\infty))$, where $S$ is a part of the $K$-sphere $\{\xec \in \R^d\,|\, \xec^tK^{-1}\xec = 1\}$. If $S$ is a sufficiently regular subset of this sphere (Jordan measurable), then one obtains in Theorem \ref{thm:largedeviation} a large deviation principle, which involves the surface integral of the residue $\psi$ on this subset $S$. Therefore, the residue $\psi$ measures intuitively a loss of symmetry when looking at the fluctuations of a mod-Gaussian convergent sequence $(\XEC_n)_{n\in \N}$ at the scale $t_n$, and when comparing these fluctuations with the Gaussian fluctuations. The two examples presented in Section \ref{subsec:modgauss} will show clearly this phenomenon. \medskip

Finally, in Section \ref{sec:examples}, we study examples and applications of the main theorems of Sections \ref{sec:speed}-\ref{sec:largedeviation}. A large class of examples relies on the method of cumulants developed in \cite[Sections 5 and 9]{FMN16} and \cite[Sections 4-5]{FMN19}, and of which we propose a multi-dimensional generalisation in Section \ref{subsec:cumulant}. The method of cumulants allows one to study the fluctuations of $d$-dimensional random walks with dependent steps (Section \ref{subsec:dependencygraph}), and those of the empirical measures of finite Markov chains (Section \ref{subsec:markov}).
\medskip

\subsection*{Acknowledgement}
We are extremely grateful to Valentin F\'eray for countless discussions on the subject and we thank him for sharing so generously with us his very helpful insights.

\bigskip

\section{Estimates on the speed of convergence}\label{sec:speed}

If $\mu$ and $\nu$ are two probability measures on $\R^d$, then a general way to measure a distance between $\mu$ and $\nu$ consists in fixing a class $\mathscr{F}$ of bounded measurable functions, and considering 
$$d_{\mathscr{F}}(\mu,\nu) = \sup_{f \in \mathscr{F}} |\mu(f)-\nu(f)|=\sup_{f \in \mathscr{F}} \left|\int_{\R^d}f(x)\,\mu(\!\DD{x})-\int_{\R^d}f(x)\,\nu(\!\DD{x})\right|.$$
In particular, given a family $\mathscr{B}$ of Borel subsets of $\R^d$, one can consider the distance associated to the class of indicator functions $\mathscr{F} = \{1_B,\,\,B \in \mathscr{B}\}$:
$$d_{\mathscr{B}}(\mu,\nu) = \sup_{B \in \mathscr{B}} |\mu(B)-\nu(B)|.$$
In dimension $1$, the Kolmogorov distance $\dkol$ is defined by this mean with the family $\mathscr{B} = \{(-\infty,s],\,\,s\in \R\}$. In higher dimension $d \geq 2$, the analogous distance
$$d_{\mathrm{Kol},\R^d}(\mu,\nu) = \sup_{(s\EX{1},\ldots,s\EX{d}) \in \R^d} \left|\int_{-\infty}^{s\EX{1}}\cdots \int_{-\infty}^{s\EX{d}} \mu(\!\DD{\xec}) - \nu(\!\DD{\xec})\right|$$
is not very convenient to deal with: for instance, it says nothing of the differences
$$\mu(B_{(\xec,\eps)}^d)-\nu(B_{(\xec,\eps)}^d)$$
of the measures of Euclidean balls, since these probabilities are not directly accessible from the $d$-dimensional cumulative distribution functions. In the setting of mod-Gaussian convergence, a much better distance is given by the class of \emph{convex Borel sets}. If $\nu$ is a probability measure that is absolutely continuous with respect to the Lebesgue measure and sufficiently isotropic, then the convex sets define a distance $\dconv$ that is compatible with the convergence in law to $\nu$, see \cite[Theorems 2.11 and 3.1]{BR10} and Section \ref{subsec:convex} hereafter. We shall then prove that, given a mod-Gaussian convergent sequence of random vectors $(\XEC_n)_{n \in \N}$, 
the convex distance between $\YEC_n = \XEC_n/\sqrt{t_n}$ and its limiting distribution $\gauss{d}{\mathbf{0}}{K}$ can be bounded by a negative power of $t_n$ (Section \ref{subsec:berryesseen}, Theorems \ref{thm:generalberryesseen} and \ref{thm:modifiedberryesseen}). The proof of these results rely on various inequalities relating Fourier transforms and the convex distance to a Gaussian distribution (Sections \ref{subsec:smooth} and \ref{subsec:fourierdistance}).\medskip

\subsection{Convex distance between probability measures}\label{subsec:convex}

Recall that a subset $C \subset \R^d$ is called convex if, for all $\xec,\yec \in C$, the whole segment $[\xec,\yec]=\{t\xec+(1-t)\yec\,,\,t \in [0,1]\}$ is included into $C$. Let $\mu$ and $\nu$ be probability measures on $\R^d$.

\begin{definition}[Convex distance]
The convex distance between $\mu$ and $\nu$ by
$$\dconv(\mu,\nu) = \sup_{\substack{C\text{ convex Borel} \\ \text{subset of }\R^d}}|\mu(C) - \nu(C)|.$$
\end{definition}

\noindent If $d=1$, then the convex sets of $\R$ are the intervals. Therefore, when $d=1$, the convex distance is a metric on probability measures that is equivalent to the Kolmogorov distance: 
$$\dkol(\mu,\nu) \leq \dconv(\mu,\nu) \leq 2\,\dkol(\mu,\nu).$$ 
For the reasons stated in the introduction of this section, we consider the convex distance to be the correct multi-dimensional generalisation of the Kolmogorov distance. This convex distance controls the convergence in law:

\begin{proposition}[Convex distance and weak convergence, direct implication]\label{prop:convexdistanceweaktopology}
Consider probability measures $(\mu_n)_{n \in \N}$ and $\nu$ on $\R^d$, such that $\dconv(\mu_n,\nu) \to 0$. Then, $\mu_n$ converges in law to $\nu$.
\end{proposition}

\begin{proof}
The set of open hypercubes $\prod_{i=1}^d (a\EX{i},b\EX{i})$ is a $\pi$-system $\mathscr{H}$ in $\R^d$, and every open set of $\R^d$ is a countable union of such hypercubes. If $\dconv(\mu_n,\nu)\to 0$, then for every $H \in \mathscr{H}$, $\mu_n(H) \to \nu(H)$ since $H$ is convex. Then, by \cite[Theorem 2.2]{Bil99}, $\mu_n \rightharpoonup \nu$.
\end{proof}
\medskip

Though the converse statement is false, it becomes true if $\nu$ is a probability measure that is \emph{regular with respect to the class $\mathscr{C}$ of Borel convex sets}. Let us explain this notion of regularity. If $\eps>0$ and $C \subset \R^d$, we set
$$C^\eps = \{\xec \in \R^d\,\,|\,\,d(\xec,C) \leq \eps\}.$$
If $C$ is a convex set, then $C^\eps$ is a closed convex set. We also define 
$$C^{-\eps} = \R^d \setminus (\R^d \setminus C)^\eps = \{\xec \in \R^d \,\,|\,\,d(\xec,\R^d \setminus C) > \eps\};$$
this is a part of $C$, and if $C$ is a convex set, then $C^{-\eps}$ is an open convex set. The \emph{$\eps$-boundary} of a (convex) set is $\partial^\eps C = C^\eps \setminus C^{-\eps}$: it is a closed set, which contains for every $\eps>0$ the topological boundary $\partial C$ of $C$ (see Figure \ref{fig:epsboundary}).

\begin{figure}[ht]
\begin{center}		
\begin{tikzpicture}[scale=2.5]
\fill [rounded corners,Red!50!white] (0,0) -- (1.7,-0.5) -- (2.5,0) -- (3,1.5) -- (0.4,1) -- cycle;
\draw [rounded corners,blue!20!white,line width=7mm] (0,0) -- (1.7,-0.5) -- (2.5,0) -- (3,1.5) -- (0.4,1) -- cycle;
\draw [rounded corners,thick,Red] (0,0) -- (1.7,-0.5) -- (2.5,0) -- (3,1.5) -- (0.4,1) -- cycle;
\draw [<->,thick] (2,1.31) -- (2,1.17);
\draw [<->,thick] (1.9,1.29) -- (1.9,1.43);
\foreach \x in {(2.1,1.25),(1.8,1.35)}
\draw \x node {$\eps$};
\draw [thick,Red] (4,1.3) -- (3.7,1.3);
\draw (4.2,1.3) node {$\partial C$};
\fill [blue!20!white] (3.7,0.8) rectangle (4,1);
\draw (4.2,0.9) node {$\partial^\eps C$};
\draw (4.2,0.5) node {$C^{-\eps}$};
\fill [Red!50!white] (3.7,0.4) rectangle (4,0.6);
\end{tikzpicture}
\caption{The $\eps$-boundary $\partial^\eps C$ of a convex set $C$.\label{fig:epsboundary}}
\end{center}
\end{figure}
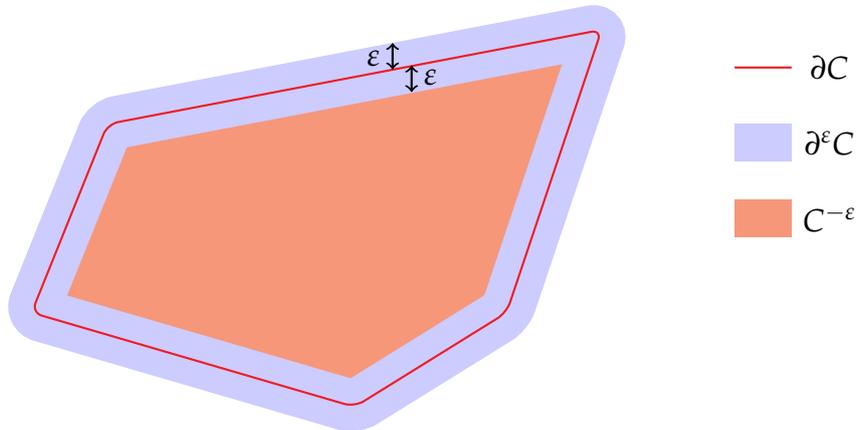

\begin{definition}[Regular probability measures]
A probability measure $\mu$ on $\R^d$ is said regular with respect to the class of convex sets if there exists a constant $R\geq 0$ such that, for every Borel convex set $C$, 
$$\mu(\partial^\eps C) \leq R\eps .$$
\end{definition}

\begin{example}[Regular measures in dimension $d=1$]
In dimension $1$, a measure is regular with respect to the class of convex sets if and only it is absolutely continuous with respect to the Lebesgue measure, and with bounded density. Indeed, if $\mu \in \mathscr{P}(\R)$ and $\frac{d\mu}{dx} \leq \frac{R}{4}$, then for any interval $[a,b]$,
$$\mu(\partial^\eps [a,b]) \leq \mu([a-\eps,a+\eps]) + \mu([b-\eps,b+\eps]) \leq 4\,\frac{R}{4}\,\eps = R\eps,$$
so the measure is regular with respect to the class of convex sets. Conversely, suppose that $\mu$ is regular with respect to the class of convex sets, of constant $R$; and denote $F_\mu$ its cumulative distribution function. Notice then that $F_\mu$ is an absolutely continuous function: if $([a_i,b_i])_{i \in I}$ is a finite family of disjoint intervals, then for any $\eps>0$,
$$\sum_{i \in I} |F_\mu(b_i)-F_\mu(a_i)| \leq \sum_{i\in I} \,\mu\!\left(\partial^{\frac{b_i-a_i+\eps}{2}}\left\{\frac{a_i+b_i}{2}\right\}\right) \leq R\, \sum_{i \in I} \frac{b_i-a_i+\eps}{2}, $$
so making $\eps$ go to $0$,
$$\sum_{i \in I} |F_\mu(b_i)-F_\mu(a_i)| \leq \frac{R}{2} \sum_{i\in I}|b_i-a_i|.$$
As a consequence, $F_\mu$ is almost everywhere derivable, and 
$$\frac{dF_\mu(x)}{dx} = \lim_{\eps \to 0} \,\frac{F_\mu(x+\eps)-F_\mu(x-\eps)}{2\eps} \leq \frac{R}{2},$$
so $\mu$ is absolutely continuous with respect to the Lebesgue measure, with density bounded by $\frac{R}{2}$.
\end{example}

\begin{example}[Non-regular measures in dimension $d \geq 2$]
In dimension $d \geq 2$, one can construct probability measures that have boun\-ded density with respect to the Lebesgue measure, but that are not regular with respect to the class of convex sets. Let $g$ be a continuous non-negative function on $\R_+$, such that:
\begin{itemize}
	\item $\int_{0}^\infty g(r)\,r^{d-1}\DD{r} < + \infty$.
	
	\item $g$ is bounded by some constant $M$ on $\R_+$, and there exists a level $L>0$ such that $\limsup_{r \to \infty} g(r) > L$.
	
\end{itemize}
Hence, we want both $g(r)$ to integrate $r^{d-1}$ and  to reach a fixed level $L>0$ an infinite number of times (on very small intervals). It is easily seen that such functions do exist. Let $\mu$ be the measure on $\R^d$ with density $g(\|\xec\|)\DD{\xec}$; up to a renormalisation of $g$, one can assume $\mu$ to be a probability measure. Then, $\mu$ has density bounded by $M$. On the other hand, there exists a sequence $r_n \to + \infty$ such that $g(r_n) \geq L$ for every $n$. Since $g$ is continuous, one can find for every $r_n$ an interval $[r_n-\eps_n,r_n+\eps_n]$ such that $g(r) \geq \frac{L}{2}$ for every $r$ in this interval. Then,
$$\mu\!\left(\partial^{\eps_n}B_{(\mathbf{0},r_n)}^d\right) = \vol(\sph^{d-1})\, \int_{r_n-\eps_n}^{r_n+\eps_n} g(r)\,r^{d-1}\DD{r} \geq L\,\vol(\sph^{d-1})\,(r_n)^{d-1}\,\eps_n .$$
Since $(r_n)^{d-1}$ can be taken as large as wanted, $\mu$ is not a regular measure with respect to convex sets.
\end{example}
\medskip

The interest of the notion of regularity comes from:
\begin{proposition}[Convex distance and weak convergence, reciprocal implication]
Let $\nu$ be a probability on $\R^d$ that is regular with respect to the class of convex sets. If $(\mu_n)_{n\in \N}$ is a sequence of probability measures with $\mu_n \rightharpoonup \nu$ (convergence in law), then $\dconv(\mu_n,\nu) \to 0$.
\end{proposition}

\begin{proof}
See \cite[Theorem 2.11]{BR10}. Later, we shall see that if $\nu$ is regular, and if sufficiently many derivatives of $\widehat{\mu}_n(\ZIEC)$ converge to those of $\widehat{\nu}(\ZIEC)$ locally uniformly, then we have indeed $\dconv(\mu_n,\nu) \to 0$. We shall even be able to quantify this; see the remark at the end of Section \ref{subsec:fourierdistance}.
\end{proof}

In the sequel, we shall need to know that the Gaussian distributions $\gauss{d}{\mathbf{m}}{K}$ are always regular with respect to the class of convex sets. This is a consequence of the following:
\begin{lemma}\label{lem:regularity}
Consider a probability measure $\nu(\!\DD{\xec}) = g(\|\xec\|)\DD{\xec}$, with $\int_{0}^\infty r^{d-1}\,|g'(r)|\DD{r} = c_1 <+\infty$. Then, for any convex set $C \subset \R^d$, and any $\eps>0$
$$\nu(\partial^\eps C) \leq 2\, c_1\, \vol(\sph^{d-1})\,\eps= \frac{4\pi^{d/2}}{\Gamma(d/2)}\,c_1\,\eps. $$
\end{lemma}

\begin{proof}
See \cite[Theorem 3.1]{BR10}.
\end{proof}

\begin{corollary}[Regularity of the Gaussian distributions]
If $\nu$ is a Gaussian distribution on $\R^d$ with covariance matrix $K$, then $\nu$ is regular with constant $R=2\sqrt{(d+1)\,\rho(K^{-1})}$: for any convex Borel set $C \subset \R^d$ and any $\eps>0$,
$$\nu(\partial^\eps C) \leq 2\sqrt{(d+1)\,\rho(K^{-1})}\,\eps.$$
\end{corollary}

\begin{proof}
If $\nu$ is the standard Gaussian distribution on $\R^d$ ($K=I_d$), then 
$$
c_1 = \frac{1}{(2\pi)^{d/2}}\int_0^\infty r^{d}\,\E^{-\frac{r^2}{2}}\,dr = \frac{2^{(d-1)/2}}{(2\pi)^{d/2}}\int_0^\infty u^{\frac{d+1}{2}-1}\,\E^{-u}\,du  = \frac{1}{\sqrt{2}\,\pi^{d/2}}\,\Gamma\!\left(\frac{d+1}{2}\right).
$$
Therefore, $\nu$ is regular with respect to convex sets, of constant
$$R \leq 2\sqrt{2}\,\frac{\Gamma(\frac{d+1}{2})}{\Gamma(\frac{d}{2})} \leq 2 \sqrt{d+1}. $$
More generally, suppose that $\nu=\mathcal{N}_{\R^d}(\mathbf{0},K)$ is an arbitrary non-degenerated Gaussian distribution on $\R^d$. If $\|\xec-\yec\| \leq \eps$, then 
$$\|K^{-1/2}(\xec-\yec)\| \leq (k\EX{1})^{-1/2} \|\xec-\yec\| \leq (k\EX{1})^{-1/2}\,\eps, $$
where $k\EX{1} \leq k\EX{2} \leq \cdots \leq k\EX{d}$ are the positive eigenvalues of $K$. It follows that $K^{-1/2}(C^\eps) \subset (K^{-1/2}C)^{\eps(k\EX{1})^{-1/2}}$, and similarly, $K^{-1/2}(C^{-\eps}) \supset (K^{-1/2}C)^{-\eps(k\EX{1})^{-1/2}}$. As a consequence, if $\nu=\mathcal{N}_{\R^d}(\mathbf{0},K)$ and $\widetilde\nu=\mathcal{N}_{\R^d}(\mathbf{0},I_d)$, then
\begin{align*}
\nu(\partial^\eps C) &=\frac{1}{\sqrt{(2\pi)^d\,\det K}} \left(\int_{C^\eps} -\int_{C^{-\eps}}\right)\,\left(\E^{-\frac{\xec^t K^{-1}\xec}{2} } \DD{\xec}  \right)\\
& \leq \frac{1}{\sqrt{(2\pi)^d}} \left(\int_{(K^{-1/2}C)^{\eps(k\EX{1})^{-1/2}}} -\int_{(K^{-1/2}C)^{-\eps(k\EX{1})^{-1/2}}}\right)\,\left(\E^{-\frac{\|\yec\|^2}{2} } \DD{\yec}  \right)\\
&\leq \widetilde{\nu}\left(\partial^{\eps(k\EX{1})^{-1/2}}(K^{-1/2}C)\right) \leq 2\sqrt{d+1}\,(k\EX{1})^{-1/2}\, \eps.
\end{align*}
Indeed, if $C$ is a convex set, then $K^{-1/2}C$ is also a convex set. 
\end{proof}
\medskip

\subsection{Smoothing techniques}\label{subsec:smooth}
The regularity of the Gaussian distributions enables us to use smoothing techniques in order to compute $\dconv(\mu,\mathcal{N}_{\R}(\mathbf{0},K))$. We first introduce an adequate smoothing kernel, which is the multi-dimensional analogue of the kernel of \cite[Lemma 17]{FMN19}:

\begin{lemma}\label{lem:kernel}
There exists a kernel $\rho$ on $\R^d$ such that:
\begin{enumerate}
	\item The kernel is non-negative, $\int_{\R^d}\rho(\xec)\DD{\xec} = 1$, and $\int_{B_{(\mathbf{0},(2d+2)^{3/2})}^d}\rho(\xec)\DD{\xec} \geq c_2$ with
	 $$2c_2-1 \geq 1-\frac{4}{9\pi}.$$
	\item The kernel $\rho(\xec)$ integrates $(\|\xec\|_1)^{d+1}$:
	$$\int_{\R^d} \rho(\xec)\,(\|\xec\|_1)^{d+1}\DD{\xec} < +\infty. $$
	\item The Fourier transform $\widehat{\rho}(\ZIEC)=\int_{\R^d} \E^{\I \scal{\ZIEC}{\xec}} \rho(\!\DD{\xec})$ of $\rho$ is compactly supported on $[-1,1]^d$.
	\item The Fourier transform of $\rho$ is of class $\mathscr{C}^{2d}$, and for any multi-index $\boldsymbol\beta$ of weight $|\boldsymbol\beta| = \sum_{i=1}^d \beta\EX{i} \leq 2d$,
$$\left\|\frac{\partial^{|\boldsymbol\beta|} \widehat{\rho}(\ZIEC)}{\partial \ZIEC^{\boldsymbol\beta}}\right\|_{\infty} \leq 2^{1+\frac{d}{2}}\,\pi^{-\frac{d}{2}}\,(2d+2)^{|\boldsymbol\beta|+\frac{d}{2}}.$$
\end{enumerate}
\end{lemma}

\begin{proof}
We set $\sinc(x)=\frac{\sin x}{x}$, $\sinc_{2d+2}(x) = (\sinc(\frac{x}{2d+2}))^{2d+2}$, and finally
$$\rho(\xec) = \frac{\prod_{i=1}^d \sinc_{2d+2} (x\EX{i})}{\int_{\R^d} \prod_{i=1}^d \sinc_{2d+2} (x\EX{i})\DD{\xec}}. $$
\begin{enumerate}
	\item Since $|\tan x| \geq |x|$ on $[-\frac{\pi}{2},\frac{\pi}{2}]$, one has
$$\int_{\R} (\sinc(x))^{2d+2}\DD{x} \geq \int_{-\frac{\pi}{2}}^\frac{\pi}{2} (\cos x)^{2d+2}\DD{x}  = \sqrt{\pi}\,\frac{\Gamma(d+\frac{3}{2})}{\Gamma(d+2)}.$$
Then, since the Euclidean ball $B_{(\mathbf{0},\eps)}^d$ contains the hypercube $C_{(\mathbf{0},\frac{\eps}{\sqrt{d}})}^d$,
\begin{align*}
\frac{\int_{\R^d \setminus B_{(\mathbf{0},\eps)}^d} \prod_{i=1}^d \sinc_{2d+2}(x\EX{i})\DD{\xec}}{\int_{\R^d} \prod_{i=1}^d \sinc_{2d+2}(x\EX{i})\DD{\xec}} &\leq \frac{\int_{\R^d \setminus C_{(\mathbf{0},\frac{\eps}{\sqrt{d}})}^d} \prod_{i=1}^d \sinc_{2d+2}(x\EX{i})\DD{\xec}}{\int_{\R^d} \prod_{i=1}^d \sinc_{2d+2}(x\EX{i})\DD{\xec}} \\
&\leq d\,\, \frac{\int_{\R \setminus [-\frac{\eps}{\sqrt{d}},\frac{\eps}{\sqrt{d}}]} \sinc_{2d+2}(x)\DD{x}}{\int_{\R} \sinc_{2d+2}(x)\DD{x}}\\
&\leq \frac{2d\,\Gamma(d+2)}{\sqrt{\pi}\,\Gamma(d+\frac{3}{2})}\, \int_{\frac{\eps}{\sqrt{d}\,(2d+2)}}^{+\infty} \,\frac{\!\DD{y}}{y^{2d+2}}\\
&\leq \frac{2d\,\Gamma(d+2)}{\sqrt{\pi}\,(2d+1)\,\Gamma(d+\frac{3}{2})}\,\left(\frac{\sqrt{d}\,(2d+2)}{\eps}\right)^{2d+1}\\
&\leq \frac{2d\,\Gamma(d+2)}{\sqrt{\pi}\,(2d+1)\,\Gamma(d+\frac{3}{2})}\,\left(\frac{d}{2d+2}\right)^{d+\frac{1}{2}}
\end{align*}
if $\eps=(2d+2)^{3/2}$.
This last expression takes its maximal value $\frac{2}{9\pi}$ for $d=1$. Therefore,
$$\int_{B_{(\mathbf{0},(2d+2)^{3/2})}^d} \rho(\xec)\DD{\xec} \geq 1-\frac{2}{9\pi} \qquad;\qquad 2c_2-1 \geq 1-\frac{4}{9\pi}.$$

\item Note that 
\begin{align*}
\prod_{i=1}^d \sinc_{2d+2}(x\EX{i}) &\leq \sinc_{2d+2}\left(\max_{1 \leq i \leq d} |x\EX{i}|\right) \\
&\leq \frac{(2d+2)^{2d+2}}{\max_{1 \leq i \leq d} |x\EX{i}|^{2d+2}} \leq \frac{(d(2d+2))^{2d+2}}{(\|\xec\|_1)^{2d+2}}.
\end{align*}
Therefore, the kernel indeed integrates $(\|\xec\|_1)^{d+1}$.

	\item One has 
\begin{align*}
\widehat{\rho}(\ZIEC) &= Z\,\prod_{i=1}^d \left(\int_{\R} \sinc_{2d+2}(x)\,\E^{\I x\zeta\EX{i}}\DD{x} \right) \\
&= Z'\,\prod_{i=1}^d \left(\int_{\R} (\sinc(x))^{2d+2}\,\E^{\I x \frac{\zeta\EX{i}}{2d+2}}\DD{x} \right) \\
&= Z'\,\prod_{i=1}^d \left(\widehat{\sinc}\right)^{\!*(2d+2)} \left(\frac{\zeta\EX{i}}{2d+2}\right),
\end{align*}
where $Z$ and $Z'$ are normalisation constants, and $f^{*(2d+2)}$ is the $(2d+2)$-th convolution power of an integrable function $f$. However, the Fourier transform of the sine cardinal is $\widehat{\sinc}(\zeta)=\pi\,1_{|\zeta|\leq 1}$, which is supported on $[-1,1]$. Therefore, $(\widehat{\sinc})^{*(2d+2)}$ is supported on $[-(2d+2),(2d+2)]$, and $\widehat{\rho}$ is indeed supported on $[-1,1]^d$.

	\item If $f(x) = (\sinc(x))^{2d+2}$, then $\frac{\partial^{k}\widehat{f}}{\partial \zeta^k}(\zeta) = \I^k \widehat{(x^k\,f)}(\zeta)$, so
	\begin{align*}
	\left\|\frac{\partial^{k}\widehat{f}}{\partial \zeta^k}(\zeta)\right\|_\infty &\leq \int_{\R} \frac{\sin^{2d+2}(x)}{|x|^{2d+2-k}}\DD{x} \leq \int_{-1}^1 |x|^{k} \DD{x} + 2 \int_{1}^\infty \frac{1}{x^{2d+2-k}}\DD{x} \\
	&\leq 2\left(\frac{1}{k+1}+\frac{1}{2d+1-k}\right) \leq 2\,\frac{2d+2}{2d+1}
	\end{align*}
	for any $k \in \lle 0,2d\rre$. Since $\widehat{\rho}(\ZIEC) = \prod_{i=1}^d\frac{\widehat{f}((2d+2)\zeta\EX{i})}{\widehat{f}(0)}$, it follows that
	\begin{align*}
	\left\|\frac{\partial^{|\boldsymbol\beta|} \widehat{\rho}(\ZIEC)}{\partial \ZIEC^{\boldsymbol\beta}}\right\|_\infty &\leq (2d+2)^{|\boldsymbol\beta|}\, \left(\frac{2(2d+2)\,\Gamma(d+2)}{\sqrt{\pi}\,(2d+1)\,\Gamma(d+\frac{3}{2})}\right)^d\\
	&\leq 2^{1+\frac{d}{2}}\,\pi^{-\frac{d}{2}}\,(2d+2)^{|\boldsymbol\beta|+\frac{d}{2}}
	\end{align*}
	by using Stirling estimates on the last line.\qedhere
 \end{enumerate}
\end{proof}
\medskip

For $\eps>0$, we note 
$$\rho_{\eps}(\xec) = \left(\frac{(2d+2)^{3/2}}{\eps}\right)^{\!d} \,\,\rho\!\left(\frac{(2d+2)^{3/2}\,\xec}{\eps}\right).$$
For every $d\geq 1$, the two first points in Lemma \ref{lem:kernel} translate into the following properties for the smoothing kernel $\rho_\eps$:
\begin{enumerate}[label=(SK\arabic*)]
	\item\label{hyp:multikernel1} $\rho_{\eps}(\xec)\geq 0$ on $\R^d$, and $\int_{\R^d} \rho_\eps(\xec)\DD{\xec}=1$.
	\item\label{hyp:multikernel2} $\rho_{\eps}$ gives a mass close to $1$ to the Euclidean ball $B_{(\mathbf{0},\eps)}^d$:
	$$\int_{B_{(\mathbf{0},\eps)}^d}\rho_{\eps}(\xec)\DD{\xec}=c_2 > \frac{1}{2}.$$ 
	\item the Fourier transform of $\rho_\eps$ is compactly supported on $\left[-\frac{(2d+2)^{3/2}}{\eps},\frac{(2d+2)^{3/2}}{\eps}\right]^d$.
\end{enumerate}

\noindent The discussion hereafter holds for any kernel with the two properties \ref{hyp:multikernel1} and \ref{hyp:multikernel2}. Let $f$ be a bounded measurable function on $\R^d$. It is well known that, if the \emph{oscillations} of $f$ are not too large with respect to a probability measure $\nu$, then one can use smoothing techniques in order to control $\int_{\R^d} f(\xec)\,(\mu(\!\DD{\xec}) - \nu(\!\DD{\xec}))$. More precisely, consider the following quantities, all related to the oscillations of the function $f$:
\begin{align*}
M(f,\xec,\eps)&=\sup\{f(\yec)\,|\,d(\xec,\yec) \leq\eps\};\\
m(f,\xec,\eps)&=\inf\{f(\yec)\,|\,d(\xec,\yec) \leq\eps\}; \\
\omega(f,\xec,\eps) &= M(f,\xec,\eps)-m(f,\xec,\eps).
\end{align*}
We then denote $f_{\mathbf{t}}(\xec)=f(\xec+\mathbf{t})$ the translate of $f$ by a vector $\mathbf{t} \in \R^d$, and 
\begin{align*}
\omega(f) &= \sup\{|f(\xec)-f(\yec)|,\,\xec,\yec \in \R^d\};\\
\omega(f,\eps)&=\int_{\R^d} \omega(f,\xec,\eps)\,\nu(\!\DD{\xec});\\
\omega^*(f,\eps) &=\sup\{\omega(f_t,\eps)\,|\,t \in \R^d\}.
\end{align*}

\begin{lemma}
For any kernel $\rho_{\eps}$ that satisfies the two properties \ref{hyp:multikernel1} and \ref{hyp:multikernel2}, one has
$$\left|\int_{\R^d} f(\xec)\,(\mu(\!\DD{\xec})-\nu(\!\DD{\xec}))\right| \leq \frac{1}{2c_2-1}\,\left(\frac{\omega(f)}{2}\,\|(\mu-\nu)*\rho_\eps\|+\omega^*(f,2\eps)\right),$$
where $\|(\mu-\nu)*\rho_\eps\|=((\mu-\nu)*\rho_\eps)_+(\R^d) + ((\mu-\nu)*\rho_\eps)_-(\R^d)$ is the norm of the signed measure $(\mu-\nu)*\rho_\eps$.
\end{lemma}

\begin{proof}
\emph{Cf.} \cite[Corollary 11.5]{BR10}.
\end{proof}
\bigskip

Suppose that $f=1_C$ is the indicator function of a Borel convex set $C \subset \R^d$, and that $\nu$ is a probability measure that is regular with constant $R$. Then, $\omega(f)=1$, and $\omega(f,\xec,\eps) = 1_{\xec \in \partial^\eps C}$. As a consequence, 
$$\omega^*(f,2\eps) =\int_{\R^d} 1_{\xec \in \partial^{2\eps}C}\, \nu(\!\DD{\xec})\leq 2R\,\eps .$$
Combining the previous lemma with the estimate of $c_2$ given in Lemma \ref{lem:kernel}, we conclude that:

\begin{proposition}[Convex distance to a regular measure]\label{prop:rao}
Let $\mu$ be an arbitrary probability measure on $\R^d$, and $\nu$ a probability measure that is regular with respect to the class of convex sets, with constant $R$. One has the following inequality:
$$\dconv(\mu,\nu) \leq \frac{1}{1-\frac{4}{9\pi}}\,\left(\frac{\|(\mu-\nu)*\rho_{\eps}\|}{2} + 2R\,\eps\right).$$
\end{proposition}

\noindent In particular, if $\nu$ a Gaussian distribution with covariance matrix $K>0$, then 
$$\dconv(\mu,\nu) \leq \frac{1}{1-\frac{4}{9\pi}}\,\left(\frac{\|(\mu-\nu)*\rho_{\eps}\|}{2} + 4\sqrt{(d+1)\,\rho(K^{-1})}\,\eps\right).$$
\medskip

\subsection{Distance between Fourier transforms}\label{subsec:fourierdistance}
Given two probability measures $\mu$ and $\nu$ on $\R^d$, we now explain how to relate the total variation norm $\|(\mu-\nu)*\rho_\eps\|$ to certain properties of the Fourier transforms $\widehat{\mu}(\ZIEC)$ and $\widehat{\nu}(\ZIEC)$. In the remainder of this paragraph, we assume that $\mu$ and $\nu$ have moments of all order (or at least of all order $\leq d+1$); therefore, $\widehat{\mu}$ and $\widehat{\nu}$ have derivatives of all order. We then denote 
$$\Delta_{\eps}(\widehat{\mu},\widehat{\nu}) = \max_{|\boldsymbol\beta|\in \lle 0, d+1\rre} \int_{\left[-\frac{(2d+2)^{3/2}}{\eps},\frac{(2d+2)^{3/2}}{\eps}\right]^d} \left|\frac{\partial^{|\boldsymbol\beta|}(\widehat{\mu}-\widehat{\nu})}{\partial \ZIEC^{\boldsymbol\beta}}(\ZIEC)\right|\DD{\ZIEC},$$
the maximum being taken over multi-indices $\boldsymbol\beta=(\beta\EX{1},\ldots,\beta\EX{d})$ such that $|\boldsymbol\beta| \in \lle 0,d+1\rre$. We also abbreviate
$$C^d_{(\mathbf{0},(2d+2)^{3/2}/\eps)}= \left[-\frac{(2d+2)^{3/2}}{\eps},\frac{(2d+2)^{3/2}}{\eps}\right]^d = D^d_\eps.$$

\begin{proposition}[Variation norm and distance between Fourier transforms]\label{prop:fouriertovariationnorm}
For any probability measures $\mu$ and $\nu$ with moments of all order, if $\eps \leq \frac{1}{\sqrt{2d+2}}$, then
$$\|(\mu-\nu)*\rho_{\eps}\| \leq 4 \,(d+1)^{\frac{d+1}{2}}\,\Delta_{\eps}(\widehat{\mu},\widehat{\nu}).$$
\end{proposition}
\medskip

\begin{lemma}
Let $g : \R^d \to \R$ be a function such that $\int_{\R^d} |g(\xec)|\,(\|\xec\|_1)^{d+1}\DD{\xec} < +\infty$. Then,
$$\int_{\R^d} |g(\xec)|\DD{\xec} \leq \sqrt{\frac{\pi d}{2}}\,\left(\frac{\E}{\pi}\right)^d\,\max_{|\boldsymbol\beta|\in \{0, d+1\}} \int_{\R^d} \left|\frac{\partial^{|\boldsymbol\beta|} \widehat{g}}{\partial \ZIEC^{\boldsymbol\beta}}(\ZIEC)\right|\DD{\ZIEC}.$$
\end{lemma}

\begin{proof}
We follow the proof of \cite[Lemma 11.6]{BR10}, but we make the constants more explicit. Let $G_+=\{\xec \in \R^d\,|\,g(\xec) \geq 0\}$, and $G_- = \R^d \setminus G_+$. For any sequences of signs $\boldsymbol\alpha \in \{\pm 1\}^d$, we denote $Q_{\boldsymbol\alpha}=\{\xec\in \R^d\,|\,\forall i \in \lle 1,d\rre,\,\,\mathrm{sgn}(x\EX{i})=\alpha\EX{i}\}$ the corresponding quadrant of $\R^d$, with by convention $\mathrm{sgn}(0)=+1$. We then write:
\begin{align*}
\int_{\R^d} &|g(\xec)|\DD{\xec} \\
&= \sum_{\boldsymbol\alpha \in \{\pm 1\}^d}\left(\int_{Q_{\boldsymbol\alpha} \cap G_+} - \int_{Q_{\boldsymbol\alpha} \cap G_-}\right) g(\xec)\,(d^{d+1}+(\|\xec\|_1)^{d+1}) \,\frac{1}{d^{d+1}+(\|\xec\|_1)^{d+1}}\DD{\xec} \\
&=\sum_{\boldsymbol\alpha \in \{\pm 1\}^d}\left(\int_{Q_{\boldsymbol\alpha} \cap G_+} - \int_{Q_{\boldsymbol\alpha} \cap G_-}\right) \left(\frac{1}{(2\pi)^d} \int_{\R^d} \widehat{h_{\boldsymbol\alpha}}(\ZIEC)\,\E^{-\I \scal{\xec}{\ZIEC}}\DD{\ZIEC} \right)\frac{1}{d^{d+1}+(\|\xec\|_1)^{d+1}}\DD{\xec} ,
\end{align*}
where 
$$h_{\boldsymbol\alpha}(\xec) = \left(d^{d+1} + \left(\sum_{i=1}^d\,\alpha\EX{i}\,x\EX{i}\right)^{\!d+1}\right)\,g(\xec).$$
The Fourier transform of $h_{\boldsymbol\alpha}(\xec)$ is:
$$\widehat{h_{\boldsymbol\alpha}}(\ZIEC) = \left(d^{d+1} + (-\I)^{d+1}\sum_{|\boldsymbol\beta|=d+1} \binom{d+1}{\beta\EX{1},\ldots,\beta\EX{d}}\,\boldsymbol\alpha^{\boldsymbol\beta}\,\frac{\partial^{|\boldsymbol\beta|}}{\partial\ZIEC^{\boldsymbol\beta} } \right)\widehat{g}(\ZIEC).$$
Therefore,
$$\left|\int_{\R^d} \widehat{h_{\boldsymbol\alpha}}(\ZIEC)\,\E^{-\I \scal{\xec}{ \ZIEC}}\DD{\ZIEC}\right| \leq 2\,d^{d+1}\,\max_{|\boldsymbol\beta| \in \{0,d+1\}} \left\|\frac{\partial^{|\boldsymbol\beta|}\widehat{g}(\ZIEC)}{\partial \ZIEC^{\boldsymbol\beta}}\right\|_{\leb^1}.$$
It follows that
$$\|g\|_{\leb^1} \leq \frac{2\,d^{d+1}}{(2\pi)^d}\,\left(\int_{\R^d}\frac{1}{d^{d+1}+(\|\xec\|_1)^{d+1}}\DD{\xec}\right)\,\max_{|\boldsymbol\beta| \in \{0,d+1\}} \left\|\frac{\partial^{|\boldsymbol\beta|}\widehat{g}(\ZIEC)}{\partial \ZIEC^{\boldsymbol\beta}}\right\|_{\leb^1}.$$
Last we compute $\int_{\R^d} \frac{1}{d^{d+1}+(\|\xec\|_1)^{d+1}}\DD{\xec}$ as follows:
\begin{align*}
\int_{\R^d} \frac{1}{d^{d+1}+(\|\xec\|_1)^{d+1}}\DD{\xec} &= \frac{2^d}{d^{d+1}} \int_{(\R_+)^d} \frac{1}{1+\left(\frac{x\EX{1}+\cdots+x\EX{d}}{d}\right)^{d+1}}\DD{\xec} \\
&= \frac{2^d}{d} \int_{(\R_+)^d} \frac{1}{1+(y\EX{1}+\cdots+y\EX{d})^{d+1}}\DD{\yec} \\
&= \frac{2^d}{d!} \int_{t=0}^{+\infty} \frac{t^{d-1}}{1+t^{d+1}}\DD{t} = \frac{2^d\,\pi}{(d+1)!\,\sin(\frac{\pi}{d+1})}.
\end{align*}
We conclude by using Stirling's approximation for $(d+1)!$:
\begin{align*}
\|g\|_{\leb^1} 
&\leq \sqrt{\frac{\pi d}{2}}\,\left(\frac{\E}{\pi}\right)^d \,\max_{|\boldsymbol\beta| \in \{0,d+1\}} \left\|\frac{\partial^{|\boldsymbol\beta|}\widehat{g}(\ZIEC)}{\partial \ZIEC^{\boldsymbol\beta}}\right\|_{\leb^1}.\qedhere
\end{align*}
\end{proof}

\begin{proof}[Proof of Proposition \ref{prop:fouriertovariationnorm}]
Since $\rho_{\eps}$ is a smoothing kernel, the signed measure obtained by the convolution $(\mu-\nu)*\rho_{\eps}$ has a (smooth) density $g$ with respect to the Lebesgue measure, and one can apply to it the previous Lemma:
$$\|(\mu-\nu)*\rho_{\eps}\| = \|g\|_{\leb^1} \leq C(d)\,\max_{|\boldsymbol\beta| \in \{0,d+1\} }\left\|\frac{\partial^{|\boldsymbol\beta|}}{\partial \ZIEC^{\boldsymbol\beta}} \left((\widehat{\mu}-\widehat{\nu})\,\widehat{\rho_\eps}\right)\right\|_{\leb^1}.$$
Indeed, 
\begin{align*}
&\int_{\R^d} |g(\xec)|\,(\|\xec\|_1)^{d+1}\DD{\xec} \\
&\leq \int_{\R^d}\int_{\R^d} (\mu+\nu)(\!\DD{\yec}) \,\rho_\eps(\xec-\yec)\,(\|\xec\|_1)^{d+1}\DD{\xec} \\
&\leq 2^{d+1} \int_{\R^d}\int_{\R^d} (\mu+\nu)(\!\DD{\yec}) \,\rho_\eps(\xec-\yec) ((\|\xec-\yec\|_1)^{d+1} +(\|\mathbf{y}\|_1)^{d+1})\DD{\xec}  \\
&\leq 2^{d+2} \left(\int_{\R^d} (\|\xec\|_1)^{d+1}\,\rho_\eps(\xec)\DD{\xec}\right) + 2^{d+1} \left(\int_{\R^d} (\|\xec\|_1)^{d+1}\,(\mu+\nu)(\!\DD{\xec}) \right)
\end{align*}
which is finite since $\mu$ and $\nu$ have moments of all order, and $\rho_{\eps}$ integrates $(\|\mathbf{x}\|_1)^{d+1}$. Now, if $|\boldsymbol\beta|$ is a fixed multi-index of total weight $|\boldsymbol\beta|=d+1$, then
\begin{align*}
&\left\|\frac{\partial^{|\boldsymbol\beta|}}{\partial \ZIEC^{\boldsymbol\beta}} ((\widehat{\mu}-\widehat{\nu})\,\widehat{\rho_\eps})\right\|_{\leb^1} \\
&\leq \sum_{\alpha\EX{1}=0}^{\beta\EX{1}}\cdots \sum_{\alpha\EX{d}=0}^{\beta\EX{d}} \prod_{i=1}^d \binom{\beta\EX{i}}{\alpha\EX{i}} \int_{D^d_{\eps}} \left|\frac{\partial^{|\boldsymbol\beta-\boldsymbol\alpha|} (\widehat{\mu}-\widehat{\nu})(\ZIEC)}{\partial \ZIEC^{\boldsymbol\beta-\boldsymbol\alpha}}\,\frac{\partial^{|\boldsymbol\alpha|} \widehat{\rho_\eps}(\ZIEC)}{\partial \ZIEC^{\boldsymbol\alpha}}\right|\DD{\ZIEC}\\
&\leq \sum_{\alpha\EX{1}=0}^{\beta\EX{1}}\cdots \sum_{\alpha\EX{d}=0}^{\beta\EX{d}} \prod_{i=1}^d \binom{\beta\EX{i}}{\alpha\EX{i}} \left\|\frac{\partial^{|\boldsymbol\alpha|}\widehat{\rho_\eps}}{\partial\ZIEC^{\boldsymbol\alpha}}\right\|_\infty \,\Delta_\eps(\widehat{\mu},\widehat{\nu})\\
&\leq 2 \,\left(\sqrt{\frac{4d+4}{\pi}}\right)^{\!d} \,\left(1+\frac{\eps}{\sqrt{2d+2}}\right)^{d+1}\,\Delta_\eps(\widehat{\mu},\widehat{\nu}) \\
&\leq 2\E^{1/2}\,\left(\sqrt{\frac{4d+4}{\pi}}\right)^{\!d}\,\Delta_\eps(\widehat{\mu},\widehat{\nu})
\end{align*}
by using the bound on the derivatives of $\widehat{\rho_\eps}$ given by Lemma \ref{lem:kernel}. Since the constant before $\Delta_{\eps}(\widehat{\mu},\widehat{\nu})$ is always larger than $1$, the same estimate holds when $|\boldsymbol\beta|=0$, so we get
\begin{align*}
\|(\mu-\nu)*\rho_\eps\| &\leq 2\E^{1/2}\,\sqrt{\frac{\pi d}{2}}\,\left(\frac{\E}{\pi}\,\sqrt{\frac{4d+4}{\pi}}\right)^d \, \Delta_{\eps}(\widehat{\mu},\widehat{\nu}) \\
&\leq \frac{\pi^{2}}{\sqrt{2\E}} \, \left(\frac{2\E}{\pi^{3/2}}\,\sqrt{d+1}\right)^{\!d+1} \Delta_{\eps}(\widehat{\mu},\widehat{\nu}) \\
&\leq 4 \,\left(\sqrt{d+1}\right)^{d+1}\,\Delta_{\eps}(\widehat{\mu},\widehat{\nu}),
\end{align*}
by simplyfing a bit the constants for the last inequality.
\end{proof}

\begin{corollary}[Convex distance and distance between Fourier transforms]\label{cor:fouriertoconvdistance}
Let $\mu$ be an arbitrary probability measure on $\R^d$, and $\nu$ be a probability distribution that is regular with respect to the class of convex sets, with constant $R$. We assume that $\mu$ and $\nu$ have moments of any order smaller than $d+1$. Then, 
$$\dconv(\mu,\nu) \leq \frac{2}{1-\frac{4}{9\pi}}\,\left((d+1)^{\frac{d+1}{2}}\,\Delta_{\eps}(\widehat{\mu},\widehat{\nu}) + R\,\eps\right)$$
for any $\eps < \frac{1}{\sqrt{2d+2}}$.
\end{corollary}
\noindent In particular, 
$$\dconv(\mu,\,\mathcal{N}_{\R^d}(\mathbf{0},K)) \leq \frac{2}{1-\frac{4}{9\pi}}\,\left((d+1)^{\frac{d+1}{2}}\,\Delta_{\eps}\!\left(\widehat{\mu},\E^{-\frac{\ZIEC^tK\ZIEC}{2}}\right) + 2\,\sqrt{(d+1)\,\rho(K^{-1})}\,\eps\right).$$
\begin{remark}
Corollary \ref{cor:fouriertoconvdistance} implies the reciprocal of Proposition \ref{prop:convexdistanceweaktopology} if one assumes that, in addition to the convergence in law $\mu_n \rightharpoonup \nu$, which is equivalent to the local uniform convergence $\widehat{\mu}_n(\ZIEC) \to \widehat{\nu}(\ZIEC)$, one also has  convergence of the partial derivatives of these Fourier transforms up to order $d+1$. Actually, given a measure $\nu$ regular with respect to convex sets, one does not need this additional assumption to have $\dconv(\mu_n,\nu)\to 0$, but this is required if one wants quantitative estimates.
\end{remark}
\medskip

\subsection{Berry--Esseen type estimates}\label{subsec:berryesseen}
We now apply the preliminary results of the previous paragraphs to a mod-Gaussian convergent sequence of random variables $(\XEC_n)_{n \in \N}$.

\subsubsection{Normal approximation of a mod-Gaussian convergent sequence} We start with the following general hypotheses:
\begin{enumerate}[label=(BE\arabic*)]
	\item\label{hyp:berryesseen1} The sequence of random vectors $(\XEC_n)_{n\in \N}$ is mod-Gaussian convergent in the Laplace sense, on $\C^d$ or on a multi-strip $\mathcal{S}_{(\mathbf{a},\mathbf{b})}$ such that $\mathbf{0} \in \prod_{i=1}^d (a\EX{i},b\EX{i})$, with parameters $t_nK$ and limit $\psi(\zec)$. We denote as before $\theta_n(\ZIEC)=\psi_n(\I \ZIEC)$ and $\theta(\ZIEC) = \psi(\I \ZIEC)$ for $\ZIEC \in \R^d$. 
	\item\label{hyp:berryesseen2} Or, the sequence $(\XEC_n)_{n\in \N}$ is mod-Gaussian convergent in a \emph{strong} Fourier sense: the residue $\theta_n$ \emph{and all its partial derivatives up to order $d+1$} converge locally uniformly on $\R$ towards $\theta$ and all its partial derivatives up to order $d+1$. 
\end{enumerate}
The hypothesis \ref{hyp:berryesseen1} is stronger than \ref{hyp:berryesseen2}. Indeed, if the holomorphic functions $\psi_n(\zec)$ converge locally uniformly towards $\psi(\zec)$, then $\psi(\zec)$ is holomorphic on the domain of convergence, and one has automatically a local uniform convergence of the complex derivatives up to any order. In particular, by restriction to the domain $D=(\I\R)^d$, one gets the content of Hypothesis \ref{hyp:berryesseen2}, which is itself just a bit stronger than Fourier mod-Gaussian convergence in the sense of Definition \ref{def:modgauss}.
\bigskip

Until the end of this paragraph, $(\XEC_n)_{n \in \N}$ is a fixed mod-Gaussian convergent sequence that satisfies Hypothesis \ref{hyp:berryesseen2}. We set
$$M(B)=M((\theta_n)_{n \in \N},B) = \sup_{n \in \N}\, \max_{|\boldsymbol\alpha| \in \lle 0,d+1 \rre}\, \max_{\ZIEC \in C^d_{(\mathbf{0},B(2d+2)^{3/2})}} \left|\frac{\partial^{|\boldsymbol\alpha|} \theta_n (\ZIEC)}{\partial \ZIEC^{\boldsymbol \alpha}}\right|.$$
Because $\theta_n \to \theta$ uniformly on $C^d_{(\mathbf{0},B(2d+2)^{3/2})}$, as well as their derivatives up to order $d+1$, $M(B)$ is finite. On the other hand, we denote $\mu_n$ the law of $\XEC_n/\sqrt{t_n}$, and $\nu=\gauss{d}{\mathbf{0}}{K}$. 

\begin{lemma}\label{lem:technicdeltaeps}
Under the hypothesis \ref{hyp:berryesseen2}, assuming also that $d \geq 2$ and that $\det K =1$, one has
$$\Delta_{\frac{1}{B\sqrt{t_n}}}(\widehat{\mu}_n,\widehat{\nu}) \leq \left(\frac{2\pi}{\E}\right)^\frac{d}{2}\, (d+1)^{d + \frac{11}{4}}\, (\rho(K^{-1}))^{\frac{d}{2}+1}\,\,\frac{M(B)}{\sqrt{t_n}}$$
for any fixed $B>0$, and for $n$ large enough.
\end{lemma}
Before proving this result, let us discuss the assumption that $K$ is \emph{normalised}:
$$\det K = 1 .$$
First, this assumption can be made without loss of generality, by replacing $K$ by 
$K_{\text{norm}} = (\det K)^{-1/d}\,K,$
and $t_n$ by $(\det K)^{1/d}\,t_n$. On the other hand, since $\det K = k\EX{1}k\EX{2}\cdots k\EX{d}$ is the product of the eigenvalues of $\det K$, if $\det K = 1$, then $\rho(K) \geq 1$ and $\rho(K^{-1}) \geq 1$. For $K \in \mathrm{S}_+(d,\R)$, set $\tau(K) = \rho(K^{-1})\,\rho(K)$. This quantity is invariant by the transformations $K \mapsto \lambda K$, and it is always larger than $1$, since
$$1 \leq \rho((K_{\mathrm{norm}})^{-1}) = (\det K)^{1/d}\,\rho(K^{-1}) \leq \rho(K)\,\rho(K^{-1}) = \tau(K).$$
Because of this sequence of inequalities, if an upper bound involving $\rho(K^{-1})$ is proven under the additional assumption $\det K =1$, then most of the time, an equivalent upper bound will hold without the assumption $\det K =1$, replacing $\rho(K^{-1})$ by $\tau(K)$.

\begin{proof}[Proof of Lemma \ref{lem:technicdeltaeps}]
Fix a multi-index $\boldsymbol\beta$ of weight $\leq d+1$, and set $\eps=1/(B\sqrt{t_n})$. We introduce the Hermite polynomials
$$ H_{n}(x) = (-1)^n\,\E^{\frac{x^2}{2}}\,\frac{d^n}{dx^n}\left(\E^{-\frac{x^2}{2}}\right),\quad n \geq 0.$$
If $\boldsymbol\alpha$ is a multi-index, then $$\E^{\frac{\|\ZIEC\|^2}{2}}\frac{\partial^{\boldsymbol\alpha}}{\partial\ZIEC^{\boldsymbol\alpha}}\left(\E^{-\frac{\|\ZIEC\|^2}{2}}\right) = (-1)^{|\boldsymbol\alpha|} \prod_{i=1}^d H_{\alpha\EX{i}}(\zeta\EX{i}),$$ which we abbreviate as $(-1)^{|\boldsymbol\alpha|} H_{\boldsymbol\alpha}(\ZIEC)$. In the following, given two multi-indices $\boldsymbol\alpha$ and $\boldsymbol\beta$, we write $\boldsymbol\alpha \leq \boldsymbol\beta$ if $\alpha\EX{i} \leq \beta\EX{i}$ for all $i \in \lle 1,d\rre$. We then set $\binom{\boldsymbol\beta}{\boldsymbol\alpha} = \prod_{i=1}^d \binom{\beta\EX{i}}{\alpha\EX{i}}$. One has
\begin{align*}
&\int_{D^d_{\eps}} \left|\frac{\partial^{|\boldsymbol\beta|}(\widehat{\mu}_n-\widehat{\nu})}{\partial\ZIEC^{\boldsymbol\beta}}(\ZIEC)\right|\DD{\ZIEC} \\
&\leq \int_{D^d_{1/B\sqrt{t_n}}} \left|\frac{\partial^{|\boldsymbol\beta|}\widehat{\nu}(\ZIEC)}{\partial\ZIEC^{\boldsymbol\beta}}\right|\,\left|\theta_n\left(\frac{\ZIEC}{\sqrt{t_n}}\right)-1\right|\DD{\ZIEC}\\
&\quad + \sum_{0<\boldsymbol\alpha \leq \boldsymbol\beta} \,\left(\frac{1}{\sqrt{t_n}}\right)^{\!|\boldsymbol\alpha|}\,\binom{\boldsymbol\beta}{\boldsymbol \alpha}\,\int_{D^d_{1/B\sqrt{t_n}}} \left|\frac{\partial^{|\boldsymbol\beta-\boldsymbol\alpha|}\widehat{\nu}}{\partial\ZIEC^{\boldsymbol\beta - \boldsymbol\alpha}}(\ZIEC) \right|\,\left|\frac{\partial^{|\boldsymbol\alpha|}\theta_n}{\partial\ZIEC^{\boldsymbol\alpha}}\left(\frac{\ZIEC}{\sqrt{t_n}}\right)\right|\DD{\ZIEC} \\
&\leq \frac{M(B)}{\sqrt{t_n}}\,\int_{\R^d} \left|\frac{\partial^{|\boldsymbol\beta|}\widehat{\nu}(\ZIEC)}{\partial\ZIEC^{\boldsymbol\beta}}\right|\,\|\ZIEC\|_1\DD{\ZIEC}\,\,+\,\, M(B)\sum_{0<\boldsymbol\alpha \leq \boldsymbol\beta} \,\left(\frac{1}{\sqrt{t_n}}\right)^{\!|\boldsymbol\alpha|}\,\binom{\boldsymbol\beta}{\boldsymbol \alpha}\,\int_{\R^d} \left|\frac{\partial^{|\boldsymbol\beta-\boldsymbol\alpha|}\widehat{\nu}}{\partial\ZIEC^{\boldsymbol\beta - \boldsymbol\alpha}}(\ZIEC) \right|\DD{\ZIEC}.
\end{align*}
Recall that $\widehat{\nu}(\ZIEC) = \E^{-\frac{\ZIEC^t K \ZIEC}{2}}$. To evaluate the integrals above, we make the change of variables $\XIEC=K^{1/2}\ZIEC$. If $M$ is an invertible matrix in $\mathrm{GL}(\R^d)$, and $\ZIEC=M\,\XIEC$, then
$$\frac{\partial f}{\partial \zeta\EX{i}} =  \sum_{j=1}^d M_{ij}\,\frac{\partial f}{\partial \xi\EX{j}} .$$
Therefore, with $M=K^{-1/2}$,
\begin{align*}
\left|\frac{\partial^{|\boldsymbol\beta|}\widehat{\nu}(\ZIEC)}{\partial\ZIEC^{\boldsymbol\beta}}\right| = \left|\left(\prod_{i=1}^d\frac{\partial^{\beta\EX{i}}}{\partial\zeta^{\beta\EX{i}}}\right)\,\E^{-\frac{\|\XIEC\|^2}{2}}\right| 
&\leq \sum_{\substack{j_{11},\ldots,j_{1\beta\EX{1}} \\ \vdots \\ j_{d1},\ldots,j_{d\beta\EX{d}}}} \left|\left(\prod_{i=1}^d \prod_{k=1}^{\beta\EX{i}} M_{ij_{ik}} \frac{\partial}{\partial \xi\EX{j_{ik}}}\right) \E^{-\frac{\|\XIEC\|^2}{2}} \right| \\
&\leq \sum_{\substack{j_{11},\ldots,j_{1\beta\EX{1}} \\ \vdots \\ j_{d1},\ldots,j_{d\beta\EX{d}}}} \left(\prod_{i=1}^d \prod_{k=1}^{\beta\EX{i}} |M_{ij_{ik}}|\right)\,|H_{\boldsymbol\alpha(\mathbf{j})}(\XIEC)| \,\E^{-\frac{\|\XIEC\|^2}{2}} ,
\end{align*}
where the multi-indices $\boldsymbol\alpha(\mathbf{j})$ depend on the choice of the indices $j_{11},\ldots,j_{d\beta\EX{d}} \in \lle 1,d\rre$, and are all of total weight $|\boldsymbol\beta|$. However, for any multi-indices $\boldsymbol\alpha$ and $\boldsymbol\beta$,
\begin{align*}
  \int_{\R^d} |H_{\boldsymbol\alpha}(\XIEC)| |H_{\boldsymbol\beta}(\XIEC)|\, \E^{-\frac{\|\XIEC\|^2}{2}}\DD{\XIEC} 
  &\leq \sqrt{\int_{\R^d} (H_{\boldsymbol\alpha}(\XIEC))^2\,\E^{-\frac{\|\XIEC\|^2}{2}}\DD{\XIEC}}\,\sqrt{\int_{\R^d} (H_{\boldsymbol\beta}(\XIEC))^2\,\E^{-\frac{\|\XIEC\|^2}{2}}\DD{\XIEC}} \\
  &\leq \sqrt{(2\pi)^d\,\boldsymbol\alpha! \,\boldsymbol\beta!}
 \end{align*}
 where $\boldsymbol\alpha! = \prod_{i=1}^d (\alpha\EX{i})!$ if $\boldsymbol\alpha$ is a multi-index of integers. Notice that if $\boldsymbol\alpha$ is of total weight $|\boldsymbol\alpha|$, then $\boldsymbol\alpha!\leq |\boldsymbol\alpha|!$. So,
 \begin{align*}
 \int_{\R^d} \left|\frac{\partial^{|\boldsymbol\beta|}\widehat{\nu}(\ZIEC)}{\partial\ZIEC^{\boldsymbol\beta}}\right| \|\ZIEC\|_1\DD{\ZIEC} 
 &\leq \sum_{k,l}\sum_{j_{11},\ldots,j_{d\beta\EX{d}}} |M_{kl}|\left(\prod_{i=1}^d \prod_{k=1}^{\beta\EX{i}} |M_{ij_{ik}}|\right)\int_{\R^d}\!|H_{\boldsymbol\alpha(\mathbf{j})}(\XIEC)|\,|\XIEC\EX{l}| \,\E^{-\frac{\|\XIEC\|^2}{2}}\DD{\XIEC}\\
 &\leq d\,\left(\max_{i \in \lle 1,d\rre} \sum_{j=1}^d |M_{ij}|\right)^{\!1+|\boldsymbol\beta|}\,\sqrt{(2\pi)^d\,|\boldsymbol\beta|!}
 \end{align*}
Moreover, 
 \begin{align*}
 \max_{i \in \lle 1,d\rre} \sum_{j=1}^d |M_{ij}| &= \max_{\|\mathbf{v}\|_\infty \leq 1} \|M\mathbf{v}\|_\infty \leq \max_{\|\mathbf{v}\|_2\leq \sqrt{d}} \|M\mathbf{v}\|_2 = \sqrt{d}\,\rho(K^{-1/2}).
 \end{align*}
Thus,
\begin{align*}
  \int_{\R^d} \left|\frac{\partial^{|\boldsymbol\beta|}\widehat{\nu}(\ZIEC)}{\partial\ZIEC^{\boldsymbol\beta}}\right|\,\|\ZIEC\|_1\DD{\ZIEC}  &\leq d^{\frac{3+|\boldsymbol\beta|}{2}}\,\left(\rho(K^{-1})\right)^{\frac{1+|\boldsymbol\beta|}{2}}\,\sqrt{(2\pi)^d\,|\boldsymbol\beta|!},\\
  \int_{\R^d} \left|\frac{\partial^{|\boldsymbol\beta-\boldsymbol\alpha|}\widehat{\nu}}{\partial\ZIEC^{\boldsymbol\beta - \boldsymbol\alpha}}(\ZIEC) \right|\DD{\ZIEC} &\leq \left(d\,\rho(K^{-1})\right)^{\frac{|\boldsymbol\beta-\boldsymbol\alpha|}{2}} \sqrt{(2\pi)^d\, |\boldsymbol\beta-\boldsymbol\alpha|!}\,.
 \end{align*}
We conclude that
\begin{align*}
&\int_{D^d_{\eps}} \left|\frac{\partial^{|\boldsymbol\beta|}(\widehat{\mu}_n-\widehat{\nu})}{\partial\ZIEC^{\boldsymbol\beta}}(\ZIEC)\right|\DD{\ZIEC} \\
&\leq \sqrt{(2\pi)^d\,|\boldsymbol\beta|!}\, \left(d\, \rho(K^{-1})\right)^{\!\frac{|\boldsymbol\beta|}{2}}\,M(B)\,\left(\frac{d^{\frac{3}{2}}\rho(K^{-\frac{1}{2}})}{\sqrt{t_n}} +  \left(1+\frac{1}{\sqrt{t_n\,d\,\rho(K^{-1})}} \right)^{|\boldsymbol\beta|}-1 \right)\\
&\leq_{n \to \infty} \sqrt{(2\pi)^d\,|\boldsymbol\beta|!}\, \left(d\, \rho(K^{-1})\right)^{\!\frac{|\boldsymbol\beta|}{2}}\,M(B)\,\left(\frac{d^{\frac{3}{2}}\rho(K^{-\frac{1}{2}})}{\sqrt{t_n}} +  \frac{|\boldsymbol\beta|+1}{\sqrt{t_n\,d\,\rho(K^{-1})}} \right)\\
&\leq_{n \to \infty} \sqrt{(2\pi)^d\,(d+1)!}\,\left(d\, \rho(K^{-1})\right)^{\!\frac{d+2}{2}}\,M(B)\,\,\frac{d + \frac{d+2}{d}}{\sqrt{t_n}}
\end{align*}
where the symbol $\leq_{n \to \infty}$ means that the inequality holds for $n$ large enough. The inequality follows by using Stirling's estimate $$(d+1)! \leq \sqrt{2\pi}\,\frac{(d+1)^{d+\frac{3}{2}}}{\E^d},$$ 
and by simplifying a bit the constants.
\end{proof}
\medskip

\begin{theorem}[General Berry--Esseen estimates]\label{thm:generalberryesseen}
Let $(\XEC_n)_{n \in \N}$ be a sequence of random variables that satisfies the hypothesis \ref{hyp:berryesseen2}. We denote $\mu_n$ the law of $\frac{\XEC_n}{\sqrt{t_n}}$ and $\nu = \gauss{d}{\mathbf{0}}{K}$. One has
$$\dconv(\mu_n,\nu) = O\!\left(\frac{1}{\sqrt{t_n}}\right).$$
If $d \geq 2$, one can take for constant in the $O(\cdot)$
$$\frac{2\,\rho(K^{-1/2})}{1-\frac{4}{9\pi}}\,\left(\left(\frac{2\pi}{\E}\right)^\frac{d}{2}\,(d+1)^{\frac{3d}{2}+\frac{13}{4}}\,(\tau(K))^{\frac{d}{2}+1}\,M(B)+ 2\,\sqrt{(d+1)\,\tau(K)}\,\frac{1}{B}\right)$$
for any $B>0$.
\end{theorem}

\begin{proof}
Suppose first that $\det K =1$. Then, combining the previous Lemma \ref{lem:technicdeltaeps} with Corollary \ref{cor:fouriertoconvdistance}, and setting $\eps=\frac{1}{B\sqrt{t_n}}$, one obtains 
\begin{align*}
&\dconv(\mu_n,\nu) \\
&\leq \frac{2}{1-\frac{4}{9\pi}}\,\left(\left(\frac{2\pi}{\E}\right)^\frac{d}{2}\,(d+1)^{\frac{3d}{2}+\frac{13}{4}}\,(\rho(K^{-1}))^{\frac{d}{2}+1}\,M(B)+ 2\,\sqrt{(d+1)\,\rho(K^{-1})}\,\frac{1}{B}\right)\,\frac{1}{\sqrt{t_n}}
\end{align*}
for any $B > 0$.\medskip

If $\det K \neq 1$, then 
\begin{align*}
&\sqrt{t_n\,(\det K)^{1/d}}\,\,\dconv\!\left(\frac{\XEC_n}{\sqrt{t_n}},\,\gauss{d}{\mathbf{0}}{K})\right)\\
&=\sqrt{t_n\,(\det K)^{1/d}}\,\,\dconv\!\left(\frac{\XEC_n}{\sqrt{t_n\,(\det K)^{1/d}}},\,\gauss{d}{\mathbf{0}}{K_{\mathrm{norm}}})\right) \\
&\leq \frac{2}{1-\frac{4}{9\pi}}\,\left(\left(\frac{2\pi}{\E}\right)^\frac{d}{2}\,(d+1)^{\frac{3d}{2}+\frac{13}{4}}\,(\tau(K))^{\frac{d}{2}+1}\,M(B)+ 2\,\sqrt{(d+1)\,\tau(K)}\,\frac{1}{B}\right)\\
&\leq \frac{2}{1-\frac{4}{9\pi}}\,\left(\left(\frac{2\pi}{\E}\right)^\frac{d}{2}\,(d+1)^{\frac{3d}{2}+\frac{13}{4}}\,(\tau(K))^{\frac{d}{2}+1}\,M(B)+ 2\,\sqrt{(d+1)\,\tau(K)}\,\frac{1}{B}\right)
\end{align*}
for any $B >0$, by using the inequality $\rho((K_{\mathrm{norm}})^{-1}) \leq \tau(K)$ previously mentioned. Finally, $(\det K)^{1/d} \geq \frac{1}{\rho(K^{-1})}$.
\end{proof}
\medskip

\subsubsection{Modification of the normal approximation}
In \cite{FMN16}, the one-dimensional equivalent of Theorem \ref{thm:generalberryesseen} was an important tool in the proof of the large deviation results (Theorem \ref{thm:1Dlargedeviation}). We shall use in Section \ref{sec:largedeviation} the same ideas, but we shall then need a slightly better Berry--Esseen bound $o((t_n)^{-1/2})$, instead of $O((t_n)^{-1/2})$. Under the general hypothesis \ref{hyp:berryesseen2}, it is possible to write such an upper bound, but then one needs to replace $\nu = \gauss{d}{\mathbf{0}}{K}$ by a \emph{signed} measure $\nu_n$, which is a small modification of the Gaussian distribution $\nu$. This paragraph is devoted to the construction of the measures $\nu_n$, and to the proof of the upper bound $\dconv(\mu_n,\nu_n)=o((t_n)^{-1/2})$. Given a function $f$ of $\ZIEC \in \R^d$, we denote
$$(\nabla f) (\ZIEC) = \left(\frac{\partial f(\ZIEC)}{\partial \zeta\EX{1}},\ldots,\frac{\partial f(\ZIEC)}{\partial \zeta\EX{d}}\right).$$
Given a sequence $(\XEC_n)_{n \in \N}$ that satisfies Hypothesis \ref{hyp:berryesseen2}, notice that 
$$
\frac{\partial\theta_n}{\partial \zeta\EX{i}}(\mathbf{0}) = \frac{\partial}{\partial \zeta\EX{i}} \left(\esper\!\left[\E^{\I \scal{\ZIEC}{\XEC_n}}\right] \E^{\frac{\ZIEC^t K \ZIEC}{2}} \right)_{|\ZIEC=\mathbf{0}} = \I\,\, \esper\!\left[\XEC_n\EX{i}\right] $$
is a purely imaginary number. Therefore, the gradient vector $\nabla \theta (\mathbf{0})$ belongs to $(\I \R)^d$. We set
$$\frac{d\nu_n}{d\xec} = \frac{1}{\sqrt{(2\pi)^d \det K}}\, \E^{-\frac{\xec^t K^{-1}\xec}{2}} \left(1-\I\,\frac{\xec^t\,K^{-1}\nabla\theta(\mathbf{0})}{\sqrt{t_n}}\right).$$
This is the density of a signed measure of total mass $1$ and with Fourier transform
$$\int_{\R^d} \E^{\I \scal{\ZIEC}{ \xec}}\, \nu_n(\!\DD{\xec}) = \E^{-\frac{\ZIEC^t K \ZIEC}{2}}\,\left(1+\frac{\ZIEC^t\,\nabla\theta(\mathbf{0})}{\sqrt{t_n}}\right).$$

\begin{lemma}\label{lem:regularitymodifiedgaussian}
The modified Gaussian distribution $\nu_n$ is regular with respect to the class of Borel convex sets, with constant 
$$R \leq_{n\to \infty} 2\sqrt{(d+1)\,\rho(K^{-1})}.$$
\end{lemma}

\begin{proof}
Suppose first that $K=I_d$. The density of $\nu_n$ is then bounded in absolute value by 
$$(2\pi)^{-\frac{d}{2}}\,\E^{-\frac{\|\xec\|^2}{2}}\,\left(1+\frac{\|\xec\|\,\|\nabla\theta(\mathbf{0})\|}{\sqrt{t_n}}\right),$$
which satisfies the hypotheses of Lemma \ref{lem:regularity}:
\begin{align*}
\int_{r=0}^\infty |g'(r)|\,r^{d-1}\DD{r} &\leq \frac{1}{(2\pi)^{d/2}}\,\int_0^\infty \E^{-\frac{r^2}{2}}\,\left(r+\frac{(r^2+1)\,\|\nabla\theta(\mathbf{0})\|}{\sqrt{t_n}}\right)\,r^{d-1}\DD{r} \\
&= \frac{1}{\sqrt{2}\,\pi^{d/2}}\,\Gamma\left(\frac{d+1}{2}\right) + O\left(\frac{1}{\sqrt{t_n}}\right).
\end{align*}
Hence, $\nu_n$ is regular with constant 
$$R=2\sqrt{2}\,\frac{\Gamma(\frac{d+1}{2})}{\Gamma(\frac{d}{2})} + O\!\left(\frac{1}{\sqrt{t_n}}\right),$$
which is smaller than $2\sqrt{d+1}$ for $t_n$ large enough. The general case is then obtained by the same techniques of change of variables as in the proof of Lemma \ref{lem:regularity}.
\end{proof}

\begin{lemma}\label{lem:technicdeltaeps2}
Under the hypothesis \ref{hyp:berryesseen2}, assuming also $\det K =1$, one has
$$\Delta_{\frac{1}{B\sqrt{t_n}}}(\widehat{\mu}_n,\widehat{\nu}_n) \leq C_1(d,\rho(K^{-1}))\,\frac{\|\nabla\theta_n(\mathbf{0})-\nabla\theta(\mathbf{0})\|_\infty}{\sqrt{t_n}} +C_2(d,\rho(K^{-1}))\,\frac{M(B)}{t_n} $$
for any $B >0$, where $C_1(d,\rho(K^{-1}))$ and $C_2(d,\rho(K^{-1}))$ are some constants that depend only on $d$ and on the spectral radius $\rho(K^{-1})$.
\end{lemma}

\begin{proof}
As in Lemma \ref{lem:technicdeltaeps}, we fix a multi-index $\boldsymbol\beta$ of total weight smaller than $d+1$, and we set $\eps=1/(B\sqrt{t_n})$. We have
$$\widehat{\mu}_n(\ZIEC)-\widehat{\nu}_n(\ZIEC) = \E^{-\frac{\ZIEC^t K \ZIEC}{2}}\,\left(\theta_n\!\left(\frac{\ZIEC}{\sqrt{t_n}}\right) - 1 - \frac{\ZIEC^t\, \nabla\theta(\mathbf{0})}{\sqrt{t_n}}\right) = \widehat{\nu}(\ZIEC)\,\delta_{n}\!\left(\frac{\ZIEC}{\sqrt{t_n}}\right).$$
In the expansion of $$\frac{\partial^{|\boldsymbol\beta|}(\widehat{\mu}_n - \widehat{\nu}_n)}{\partial \ZIEC^{\boldsymbol\beta}} = \sum_{\mathbf{0}\leq \boldsymbol\alpha \leq \boldsymbol\beta} \left(\frac{1}{\sqrt{t_n}}\right)^{|\boldsymbol\alpha|}\binom{\boldsymbol\beta}{\boldsymbol\alpha} \left(\frac{\partial^{|\boldsymbol\beta-\boldsymbol\alpha|} \widehat{\nu}}{\partial\ZIEC^{\boldsymbol\beta-\boldsymbol\alpha}}(\ZIEC)\right) \,\left(\frac{\partial^{|\boldsymbol\alpha|} \delta_n}{\partial\ZIEC^{\boldsymbol\alpha}}\!\left(\frac{\ZIEC}{\sqrt{t_n}}\right)\right),$$
we separate the multi-indices $\boldsymbol\alpha$ in three categories:

\begin{enumerate}
	\setcounter{enumi}{-1}
	\item If $\boldsymbol\alpha = \mathbf{0}$, then 
	\begin{align*}
	&\int_{D^d_{\eps}} \left|\frac{\partial^{|\boldsymbol\beta|} \widehat{\nu}}{\partial\ZIEC^{\boldsymbol\beta}}(\ZIEC)\right|\,\, \left|\delta_n\!\left(\frac{\ZIEC}{\sqrt{t_n}}\right)\right|\DD{\ZIEC} \\
	&\leq \int_{D^d_{1/B\sqrt{t_n}}} \left|\frac{\partial^{|\boldsymbol\beta|} \widehat{\nu}}{\partial\ZIEC^{\boldsymbol\beta}}(\ZIEC)\right|\,\,\left( \left|\frac{\ZIEC^t(\nabla\theta_n(\mathbf{0})-\nabla\theta(\mathbf{0}))}{\sqrt{t_n}}\right|+\frac{M(B)\,(\|\ZIEC\|_1)^2}{2t_n}\right)\DD{\ZIEC} \\
	&\leq \frac{\|\nabla\theta_n(\mathbf{0})-\nabla\theta(\mathbf{0})\|_\infty}{\sqrt{t_n}} \int_{\R^d} \left|\frac{\partial^{|\boldsymbol\beta|} \widehat{\nu}}{\partial\ZIEC^{\boldsymbol\beta}}(\ZIEC)\right| \|\ZIEC\|_1\DD{\ZIEC} + \frac{d\,M(B)}{2t_n} \int_{\R^d} \left|\frac{\partial^{|\boldsymbol\beta|} \widehat{\nu}}{\partial\ZIEC^{\boldsymbol\beta}}(\ZIEC)\right| (\|\ZIEC\|_2)^2\DD{\ZIEC} \
	\end{align*}
	 by using a Taylor--Lagrange expansion of order $2$ on the second line. The first integral has been bounded in the proof of Lemma \ref{lem:technicdeltaeps} by
	 $$d^{\frac{d+4}{2}}\,\left(\rho(K^{-1})\right)^{\frac{d+2}{2}}\,\sqrt{(2\pi)^d\,(d+1)!}\,.$$
	 On the other hand, setting $M=K^{-1/2}$ and $\ZIEC=M\XIEC$, the second integral is smaller than
	 \begin{align*}
	 &\rho(K^{-1})\sum_{l}\sum_{j_{11},\ldots,j_{d\beta\EX{d}}} \left(\prod_{i=1}^d \prod_{k=1}^{\beta\EX{i}} |M_{ij_{ik}}|\right)\,\int_{\R^d}|H_{\boldsymbol\alpha(\mathbf{j})}(\XIEC)|\,|\XIEC\EX{l}|^2 \,\E^{-\frac{\|\XIEC\|^2}{2}}\DD{\XIEC} \\
	 &\leq d\,\rho(K^{-1})\sum_{j_{11},\ldots,j_{d\beta\EX{d}}} \left(\prod_{i=1}^d \prod_{k=1}^{\beta\EX{i}} |M_{ij_{ik}}|\right)\,\sqrt{(2\pi)^d\,(d+1)!}\\
	 &\leq d^{\frac{d+3}{2}}\,\left(\rho(K^{-1})\right)^{\frac{d+3}{2}} \,\sqrt{(2\pi)^d\,(d+1)!}\,.
	 \end{align*}
	 So, the term with multi-index $\boldsymbol\alpha = \mathbf{0}$ gives a contribution smaller than
	 $$ \left(\frac{\|\nabla\theta_n(\mathbf{0})-\nabla\theta(\mathbf{0})\|_\infty}{\sqrt{t_n}}  + \frac{\sqrt{d\,\rho(K^{-1})}\,M(B)}{2t_n}\right) d^{\frac{d+4}{2}}\,\left(\rho(K^{-1})\right)^{\frac{d+2}{2}} \,\sqrt{(2\pi)^d\,(d+1)!} \,.$$

	 \item If $|\boldsymbol\alpha|=1$, then $\boldsymbol\alpha = (0,\ldots,0,1_i,0,\ldots,0)$ for some $i \in \lle 1,d\rre$. The derivative of $\delta_n$ is then 
	 $$
	 \left|\frac{\partial \delta_n(\ZIEC)}{\partial \ZIEC^{\boldsymbol\alpha}} \right| = \left|\frac{\partial\theta_n(\ZIEC)}{\partial \zeta\EX{i}} -\frac{\partial\theta(\mathbf{0})}{\partial \zeta\EX{i}}\right| \leq  \|\nabla\theta_n(\mathbf{0})-\nabla\theta(\mathbf{0})\|_\infty + M(B)\,\|\ZIEC\|_1
	 $$
	 if $\ZIEC \in C^d_{(\mathbf{0},B(2d+2)^{3/2})}$. Therefore, the term with multi-index $\boldsymbol\alpha =  (0,\ldots,1_i,\ldots,0)$ gives a contribution to the sum smaller than
	 \begin{align*}
	 & \frac{\beta\EX{i}}{\sqrt{t_n}}\,\int_{\R^d} \left|\frac{\partial^{|\boldsymbol\beta-\boldsymbol\alpha|}\widehat{\nu}(\ZIEC)}{\partial \ZIEC^{\boldsymbol\beta-\boldsymbol\alpha}}\right|\,\left(\|\nabla\theta_n(\mathbf{0})-\nabla\theta(\mathbf{0})\|_\infty + \frac{M(B)\,\|\ZIEC\|_1}{\sqrt{t_n}}\right) \DD{\ZIEC} \\
	 &\leq \beta\EX{i} \left( \frac{\|\nabla\theta_n(\mathbf{0})-\nabla\theta(\mathbf{0})\|_\infty}{\sqrt{t_n}} + \frac{\sqrt{d^3\,\rho(K^{-1})}\,M(B)}{t_n}\right) \, (d\rho(K^{-1}))^{\frac{d}{2}}\,\sqrt{(2\pi)^d\,d!}
	 \end{align*}
	 
	 \item Finally, if $|\boldsymbol\alpha|\geq 2$, then the derivative with multi-index $\boldsymbol\alpha$ of $\delta_n$ is the same as the derivative of $\theta_n$, so one can use the same bounds as in Lemma \ref{lem:technicdeltaeps}. Thus, the term with multi-index $\boldsymbol\alpha$ gives a contribution smaller than
	 $$M(B)\,\binom{\boldsymbol\beta}{\boldsymbol\alpha}\,\left(\frac{1}{\sqrt{t_n\,d\,\rho(K^{-1})}}\right)^{|\boldsymbol\alpha|}\,(d\,\rho(K^{-1}))^{\frac{d+1}{2}}\,\sqrt{2\pi\,|\boldsymbol\beta-\boldsymbol\alpha|!}. $$ 
\end{enumerate}
The claim follows by summing over the three kind of multi-indices $\boldsymbol\alpha$.
\end{proof}
\medskip

\begin{theorem}[Convex distance to the modified normal approximation]\label{thm:modifiedberryesseen}
Let $(\XEC_n)_{n \in \N}$ be a sequence of random variables that satisfies Hypothesis \ref{hyp:berryesseen2}. We denote $\mu_n$ the law of $\frac{\XEC_n}{\sqrt{t_n}}$, and $\nu_n$ the modified Gaussian distribution whose Fourier transform is
$$ \widehat{\nu}_n(\ZIEC) = \E^{-\frac{\ZIEC^tK\ZIEC}{2}} \left(1+\frac{\ZIEC^t \nabla \theta(\mathbf{0})}{\sqrt{t_n}}\right).$$ 
One has
$$\dconv(\mu_n,\nu_n) = o\!\left(\frac{1}{\sqrt{t_n}}\right).$$
\end{theorem}

\begin{proof}
One can assume without loss of generality that $\det K =1$. Fix $\eta>0$, and take $B= \frac{1}{\eta}$, and $\eps = \frac{1}{B\sqrt{t_n}} = \frac{\eta}{\sqrt{t_n}}$. By Corollary \ref{cor:fouriertoconvdistance} and Lemma \ref{lem:regularitymodifiedgaussian},
$$\dconv(\mu_n,\nu_n) \leq_{n \to \infty} \frac{2}{1-\frac{4}{9\pi}} \left((d+1)^{\frac{d+1}{2}}\,\Delta_{\frac{1}{B\sqrt{t_n}}}(\widehat{\mu}_n,\widehat{\nu}_n) + 2\sqrt{(d+1)\,\rho(K^{-1})}\,\frac{\eta}{\sqrt{t_n}}\right).$$
By hypothesis, $\nabla(\theta_n-\theta)$ converges locally uniformly to $0$. Therefore, in the estimate of $\Delta_{\frac{1}{B\sqrt{t_n}}}(\widehat{\mu}_n,\widehat{\nu}_n)$ given by Lemma \ref{lem:technicdeltaeps2}:
\begin{itemize}
	\item The term that is proportional to $\frac{\|\nabla\theta_n(\mathbf{0}) - \nabla\theta(\mathbf{0})\|_\infty}{\sqrt{t_n}}$ gets smaller than $\frac{\eta}{\sqrt{t_n}}$ for $t_n$ large enough.
	\item This is also trivially true for the term proportional to $\frac{M(B)}{t_n}$, since $B$ is fixed.
\end{itemize}
Thus, there is a constant $C_3(d,\rho(K^{-1}))$ such that, for $n$ large enough,
$$ \dconv(\mu_n,\nu_n) \leq C_3(d,\rho(K^{-1}))\,\frac{\eta}{\sqrt{t_n}}.$$
Since this is true for any $\eta>0$, the result is proven.
\end{proof}
\medskip

\begin{remark}
There is an important precision that one can make to the statement of Theorem \ref{thm:modifiedberryesseen}: if $\eps_n = \sqrt{t_n}\,\dconv(\mu_n,\nu_n)$, then the speed of convergence of $(\eps_n)_{n \in \N}$ to $0$ only depends on $d$, on $\rho(K^{-1})$ and on the speed of convergence to $0$ of $\|\nabla (\theta_n-\theta)(\mathbf{0})\|_\infty$. By this, we mean that if $((\mu_n\EX{i},\nu_n\EX{i})_{n \in \N})_{i \in I}$ are families of sequences of distributions on $\R^d$ that all satisfy the hypotheses above, and with a uniform bound on $\rho((K\EX{i})^{-1})$ and on $\|\nabla (\theta_n\EX{i}-\theta\EX{i})(\mathbf{0})\|_\infty$, then the corresponding sequences $(\eps_n\EX{i})_{i\in I}$ converge uniformly to $0$: 
$$\lim_{n \to \infty} \left(\sup_{i \in I} \eps_n\EX{i}\right) = 0. $$
We shall use this precision during the proof of Theorem \ref{thm:largedeviation}.
\end{remark}
\medskip

If $\nabla \theta(\mathbf{0})=0$, then Theorem \ref{thm:modifiedberryesseen} ensures that the distance computed in Theorem \ref{thm:1Dgeneralberryesseen} is a $o(1/\sqrt{t_n})$ instead of a $O(1/\sqrt{t_n})$. In this situation, one can develop a theory of zones of control, which is analogous to the hypotheses of Theorem \ref{thm:1Dzoneofcontrol} and, and which gives upper bounds 
$$\dconv(\mu_n,\gauss{d}{\mathbf{0}}{K})=O\!\left((t_n)^{-\frac{1}{2}-\gamma}\right).$$
As the main examples that we have in mind for this theory all rely on the method of cumulants, we postpone the exposition of these better estimates to Section \ref{sec:examples}.
\bigskip

\section{Large deviation results}\label{sec:largedeviation}

In this section, we fix a sequence $(\XEC_n)_{n \in \N}$ of random vectors in $\R^d$ that is mod-Gaussian convergent in the Laplace sense on a multi-strip $\mathcal{S}_{\mathbf{a},\mathbf{b}}$, with $\mathbf{0} \in \prod_{i=1}^d (a\EX{i},b\EX{i})$. As before, $(t_n K)_{n \in \N}$ is the sequence of parameters, and the limiting function is denoted by $\psi$. We also suppose that $\psi$ does not vanish on the real part of the multi-strip $\mathcal{S}_{\mathbf{a},\mathbf{b}}$; this hypothesis will be useful in certain arguments of exponential change of measure. We are interested in the asymptotics of the probabilities
$$\proba[\XEC_n \in t_n B],$$
where $B = [b,+\infty) \times S$, and $S$ is some measurable part of the $(d-1)$-dimensional $K$-ellipsoid 
$\sph^{d-1}(K) = K^{-1/2}(\sph^{d-1}) = \{\xec \in \R^d\,|\,\xec^t K \xec \leq 1\}$.
For instance, when $d=2$, assuming to simplify that $K=I_2$, we want to be able to deal with angular sectors
$$B = \{z = r\E^{\I\theta} \in \C\,\,|\,\,r \geq b,\,\,\theta \in (\theta_1,\theta_2)\}.$$
This problem is the multi-dimensional generalisation of the discussion of \cite[Section 4]{FMN16} and of Theorem \ref{thm:1Dlargedeviation}. In Subsection \ref{subsec:toymodel}, we study a toy model in order to make a correct conjecture for the asymptotics of $\proba[\XEC_n \in t_n B]$. In Subsection \ref{subsec:cone}, we use Theorem \ref{thm:modifiedberryesseen} and a method of tilting of measures in order to obtain the precise asymptotics of the probabilities of certain cones. 
Finally, in Subsection \ref{subsec:approximationcone}, we detail the approximation of spherical sectors $B = [b,+\infty) \times S$ by unions of cones, thereby getting the asymptotic formula conjectured in Subsection \ref{subsec:toymodel} (Theorem \ref{thm:largedeviation}). Examples will be examined in Section \ref{sec:examples}.\medskip

\subsection{Sum of a Gaussian noise and an independent random variable}\label{subsec:toymodel}
In order to make a correct conjecture, we consider the trivial example of mod-Gaussian convergence 
$$\XEC_n = \sqrt{t_n}\,\mathbf{G}+\YEC,$$
where $\mathbf{G}$ is a random vector with law $\mathcal{N}_{\R^d}(0,I_d)$, and $\YEC$ is a bounded random vector independent from $\mathbf{G}$. If $\esper[\E^{\scal{\zec}{\YEC}}] = \psi(\zec)$, then $\psi(\zec)$ is well-defined over $\C^d$ since $\YEC$ is bounded, and
$$\psi_n(\zec) = \esper\!\left[\E^{\scal{\zec}{\XEC_n}}\right]\,\E^{-\frac{t_n\|\zec\|^2}{2}} = \psi(\zec). $$
Therefore, $(\XEC_n)_{n \in \N}$ is mod-Gaussian convergent on $\C^d$ with parameters $t_nI_d$ and limit $\psi(\zec)$. To simplify a bit, we also assume that $\YEC$ has a density $\nu(\yec)\DD{\yec}$ with respect to the Lebesgue measure. Then, a Borel set $B$ being fixed, we have
\begin{align*}\proba[\XEC_{n} \in t_nB]&=\frac{1}{(2\pi)^{\frac{d}{2}}}\int_{\xec \in \R^{d}}\int_{\yec \in \R^{d}} \E^{-\frac{\|\xec\|^2}{2}} \,1_{(\sqrt{t_n}\xec+\yec \,\in\, t_nB)}\, \nu(\yec)\DD{\xec}\DD{\yec} \\
&=\left(\frac{t_n}{2\pi}\right)^{\frac{d}{2}}\int_{\mathbf{u} \in \R^{d}} \int_{\mathbf{y} \in \R^{d}} \E^{-\frac{t_n\,\|\mathbf{u}\|^2}{2}} \,1_{( \mathbf{u}+(t_n)^{-1}\yec\,\in\, B)}\, \nu(\yec)\DD{\mathbf{u}}\DD{\yec}\\
&=\left(\frac{t_n}{2\pi}\right)^{\frac{d}{2}}\int_{\mathbf{b} \in B} \E^{-\frac{t_n\,\|\mathbf{b}\|^2}{2}} \DD{\mathbf{b}} \left(\int_{\mathbf{y} \in \R^{d}}  \E^{\scal{\mathbf{b}}{\yec}}\, \E^{-\frac{\|\yec\|^2}{2t_n}}\,\nu(\yec)\DD{\yec} \right). 
\end{align*}
By Lebesgue's dominated convergence theorem, the term in parentheses converges to the integral $\int_{\R^{d}} \E^{\scal{\mathbf{b}}{\yec}}\,\nu(\yec)\DD{\yec}=\psi(\mathbf{b})$. This convergence is locally uniform in the parameter $\mathbf{b}$. Consequently, on any bounded subset of $\R^d$, one can write
$$\int_{\mathbf{y} \in \R^{d}}  \E^{\scal{\mathbf{b}}{\yec}}\, \E^{-\frac{\|\yec\|^2}{2t_n}}\,\nu(\yec)\DD{\yec} = \psi(\mathbf{b})\,(1+o(1)) $$
with a uniform $o(1)$. We now suppose that $B=[b,+\infty)\times S$, where $b>0 $ and $S$ is a measurable part of the sphere $\sph^{d-1}$. We split $B$ in two parts
$$B_{<b+\eps} = [b,b+\eps) \times S \qquad;\qquad B_{\geq b+\eps} = [b+\eps,+\infty) \times S,$$
where $\eps$ is a small parameter to be chosen later. We denote $I_{<b+\eps}$ and $I_{\geq b+\eps}$ the corresponding parts of the integral form of $\proba[\XEC_{n} \in t_{n}B]$ computed above. 
For the integral $I_{< b+\eps}$, recall that for any non-negative measurable function $f$ over $\R^{d}$, one can make the polar change of coordinates
$$\int_{\R^{d}} f(\xec)\DD{\xec} = \int_{r=0}^{\infty} \left(\int_{\sph^{d-1}} \!f(r\mathbf{s})\,\, \mu_{\sph^{d-1}}(\!\DD{\mathbf{s}})\right)r^{d-1}\DD{r},$$
where $\mu_{\sph^{d-1}}$ is the unique $\SO(d)$-invariant measure on the sphere with total mass $\frac{2\,\pi^{d/2}}{\Gamma(d/2)}$. Therefore,
\begin{align*}
(1+o(1))\,I_{< b+\eps}&= \left(\frac{t_{n}}{2\pi}\right)^{\!\frac{d}{2}} \int_{\mathbf{b} \in B_{<b+\eps}}\E^{-\frac{t_{n}\,\|\mathbf{b}\|^{2}}{2}}\,\psi(\mathbf{b})\DD{\mathbf{b}} \\
&=\left(\frac{t_{n}}{2\pi}\right)^{\!\frac{d}{2}} \int_{r=b}^{b+\eps} \left(\int_{S} \psi(r\mathbf{s})\,\mu_{\sph^{d-1}}(\!\DD{\mathbf{s}})\right) r^{d-1}\,\E^{-\frac{t_nr^2}{2}}\DD{r}.
\end{align*}
In the formula above, the multiplicative factor $(1+o(1))$ corresponds to the replacement of the integral $\int_{\mathbf{y} \in \R^{d}}  \E^{\scal{\mathbf{b}}{\yec}}\, \E^{-\|\yec\|^2/2t_n}\,\nu(\yec)\DD{\yec}$ by $\psi(\mathbf{b})$. Since $\psi$ is a continuous function, one can write uniformly on the sphere $\psi(r\mathbf{s})=\psi(b\mathbf{s})\,(1+o(1))$ for any $r \in [b,b+\eps)$ if $\eps=o(1)$. So,
\begin{align*}
  (1+o(1))\,I_{< b+\eps} &= \left(\frac{t_{n}}{2\pi}\right)^{\!\frac{d}{2}} \left(\int_{r=b}^{b+\eps} r^{d-1}\,\E^{-\frac{t_nr^2}{2}}\DD{r}\right)\left(\int_{S} \psi(b\mathbf{s})\,\mu_{\sph^{d-1}}(\!\DD{\mathbf{s}}) \right) \\
  &= \left(\frac{t_n}{2\pi}\right)^{\!\frac{d}{2}}\, \frac{b^{d-2}}{t_n}\, \left(\E^{-\frac{t_nb^2}{2}} - \E^{-\frac{t_n(b+\eps)^2}{2}}\right) \left(\int_{S} \psi(b\mathbf{s})\,\mu_{\sph^{d-1}}(\!\DD{\mathbf{s}}) \right) \\
  &= \left(\frac{t_n}{2\pi}\right)^{\!\frac{d}{2}}\, \frac{b^{d-2}}{t_n}\, \E^{-\frac{t_nb^2}{2}} \left(1-\E^{-\frac{t_n\eps^2}{2} -t_nb\eps}\right)\left(\int_{S} \psi(b\mathbf{s})\,\mu_{\sph^{d-1}}(\!\DD{\mathbf{s}}) \right).
 \end{align*}
Consider now the integral $I_{\geq b+\eps}$. Since $B_{\geq b+\eps} \subset [b+\eps,+\infty) \times \sph^{d-1}$ and $\|\YEC\|$ is bounded almost surely by some constant $C$, we have
\begin{align*}I_{\geq b+\eps}&\leq \left(\frac{t_{n}}{2\pi}\right)^{\!\frac{d}{2}}\int_{\|\mathbf{b}\|\geq b+\eps} \E^{-\frac{t_{n}\|\mathbf{b}\|^{2}}{2}}\,\E^{C\|\mathbf{b}\|}\DD{\mathbf{b}}\\
&\leq \frac{2}{\Gamma(\frac{d}{2})}\left(\frac{t_{n}}{2}\right)^{\!\frac{d}{2}} \int_{b+\eps}^{\infty} w^{d-1}\,\E^{-\frac{t_{n}r^{2}}{2}}\,\E^{Cr}\DD{r} \\
&\leq \frac{2\,\E^{\frac{C^{2}}{2\,t_{n}}}}{\Gamma(\frac{d}{2})}\left(\frac{t_{n}}{2}\right)^{\!\frac{d}{2}} \int_{b+\eps-\frac{C}{t_{n}}}^{\infty} \left(x+\frac{C}{t_{n}}\right)^{\!d-1}\E^{-\frac{t_{n}x^{2}}{2}}\DD{x}\\
&\leq  \frac{2\,\E^{\frac{C^{2}}{2\,t_{n}}}}{\Gamma(\frac{d}{2})}\left(\frac{t_{n}}{2}\right)^{\!\frac{d}{2}} \left(1+\frac{C}{t_{n}b}\right)^{\!d-1} \int_{b+\eps-\frac{C}{t_{n}}}^{\infty} x^{d-1}\,\E^{-\frac{t_{n}x^{2}}{2}}\DD{x}\\
&\leq \frac{1}{\Gamma(\frac{d}{2})}\left(\frac{t_{n}}{2}\right)^{\!\frac{d}{2}-1}  \left(1+\frac{C}{t_{n}b}\right)^{\!d-1} \left(1-\frac{d-2}{t_{n}b^{2}}\right)^{\!-1} \E^{-\frac{t_{n}(b+\eps)^{2}}{2}}\,\E^{C(b+\eps)}\,(b+\eps)^{d-2}
\end{align*}
assuming $\frac{C}{t_{n}} \leq \eps$, and using integration by parts at the end to estimate the Gaussian integral. For $t_n$ large enough, we have $t_{n}b \geq C$ and $t_{n}b^{2}\geq 2(d-2)$, which leads to
$$ I_{\geq b+\eps} \leq M(d)\,(t_{n})^{\frac{d}{2}-1}\,\E^{-\frac{t_{n}(b+\eps)^{2}}{2}}\,\E^{C(b+\eps)}\,(b+\eps)^{d-2},$$
where $M(d)$ is a certain constant that depends only on $d$. 
We take $\eps=(t_n)^{-1/3}$. For $n$ large enough, we have indeed $\frac{C}{t_n} \leq \eps$, and from the previous estimates:
\begin{align*}
I_{\geq b+\eps} &\leq  M(d)\,(t_{n})^{\frac{d}{2}}\,\frac{b^{d-2}}{t_n}\,\E^{Cb}\,\E^{-\frac{t_nb^2}{2}}\,\E^{-\frac{(t_{n})^{1/3}}{2}};\\
(1+o(1))\,I_{< b+\eps} &= \left(\frac{t_n}{2\pi}\right)^{\!\frac{d}{2}}\, \frac{b^{d-2}}{t_n}\, \E^{-\frac{t_nb^2}{2}} \left(\int_{S} \psi(b\mathbf{s})\,\mu_{\sph^{d-1}}(\!\DD{\mathbf{s}})\right).
\end{align*}
Therefore, assuming that $S$ has a non-zero surface measure in $\sph^{d-1}$, $I_{\geq b+\eps}$ is negligible in comparison to the other integral $I_{<b+\eps}$, and we obtain the asymptotic formula
$$
\proba[\XEC_n \in t_n B] \simeq  \left(\frac{t_n}{2\pi}\right)^{\!\frac{d}{2}}\, \frac{b^{d-2}}{t_n}\, \E^{-\frac{t_nb^2}{2}} \left(\int_{S} \psi(b\mathbf{s})\,\mu_{\sph^{d-1}}(\!\DD{\mathbf{s}})\right) .
$$
It is convenient to rewrite the spherical integral as an integral over the sphere of radius $b$, which we denote $\sph^{d-1}(b)$. We also denote $\mu_{\mathrm{surface}}$ the Lebesgue surface measure on it, which has total mass $$ \mu_{\mathrm{surface}}(\sph^{d-1}(b))=\frac{2\,\pi^{\frac{d}{2}}\,b^{d-1}}{\Gamma(\frac{d}{2})}.$$
Then, the previous formula rewrites as 
$$
\proba[\XEC_n \in t_n B] \simeq  \left(\frac{t_n}{2\pi}\right)^{\!\frac{d}{2}}\, \frac{1}{t_n b}\, \E^{-\frac{t_nb^2}{2}} \left(\int_{S_b} \psi(\mathbf{s})\,\mu_{\mathrm{surface}}(\!\DD{\mathbf{s}})\right),
$$
where $S_b = S\times b = \{b\mathbf{s},\,\,\mathbf{s} \in S\}$ is the base of the spherical sector $B = S \times [b,+\infty)$ under consideration. Our goal will then be to show that this asymptotic formula actually holds in the general setting of multi-dimensional mod-Gaussian convergence.
\medskip

\subsection{Asymptotics of the probabilities of cones}\label{subsec:cone}
Until further notice, we assume that $K=I_d$; the general case will be treated by using the arguments of the remark at the end of Subsection \ref{subsec:modgauss}. In this paragraph, we shall compute the asymptotics of the probabilities $\proba[\XEC_n \in t_n C]$, where 
$$C = [1,+\infty) \times D =\{r\mathbf{d},\,\,r\geq 1\text{ and }\mathbf{d} \in D\}$$
and $D$ is a convex domain of an hyperplane $H$ that does not contain the origin $\mathbf{0}$ of $\R^d$; see Figure \ref{fig:cone}.
\begin{center}
\begin{figure}[ht]
\begin{tikzpicture}[scale=0.7]
\draw (-8.3,10.5) node {\textcolor{gray}{$H$}};
\draw (-2.5,9) node {\textcolor{violet}{$D$}};
\draw [thick,gray] (7,10) -- (9,8) -- (-6,8) -- (-8,10);
\fill [NavyBlue!15!white] (-3,16.4) -- (3.6,16) -- (7,14.2) -- (4.5,12.875) -- (-4,13);
\fill [NavyBlue!30!white] (7,14.2) -- (4.5,12.875) -- (-4,13) -- (-2,9) -- (2,8.5) -- (3.5,9.6);
\fill [pattern = north west lines, pattern color = violet] (-2,9) -- (2,8.5) -- (3.5,9.6) -- (1.8,10.5) -- (-1.5,10.7);
\draw [dashed] (-2,9) -- (0,5) -- (0,9.7);
\draw [dashed] (3.5,9.6) -- (0,5) -- (2,8.5);
\draw [dashed] (-1.5,10.7) -- (0,5) -- (1.8,10.5);
\draw [thick,violet] (-2,9) -- (2,8.5) -- (3.5,9.6);
\draw [thick,dashed,violet] (3.5,9.6) -- (1.8,10.5) -- (-1.5,10.7) -- (-2,9);
\draw [thick,NavyBlue] (4.5,12.875) -- (2,8.5);
\draw [thick,NavyBlue] (-2,9) -- (-4,13);
\draw [thick,NavyBlue] (3.5,9.6) -- (7,14.2);
\draw [thick,dashed,NavyBlue] (1.8,10.5) -- (3.6,16);
\draw [thick,dashed,NavyBlue] (-1.5,10.7) -- (-3,16.4);
\fill (0,5) circle (3pt);
\draw [thick] (-0.15,9.55) -- (0.15,9.85);
\draw [thick] (0.15,9.55) -- (-0.15,9.85);
\draw (5.3,12.7) node {\textcolor{NavyBlue}{$C$}};
\draw (-0.5,5) node {$\mathbf{0}$};
\end{tikzpicture}
\caption{Cone $C$ based on a convex domain $D$ of an hyperplane $H$.\label{fig:cone}}
\end{figure}
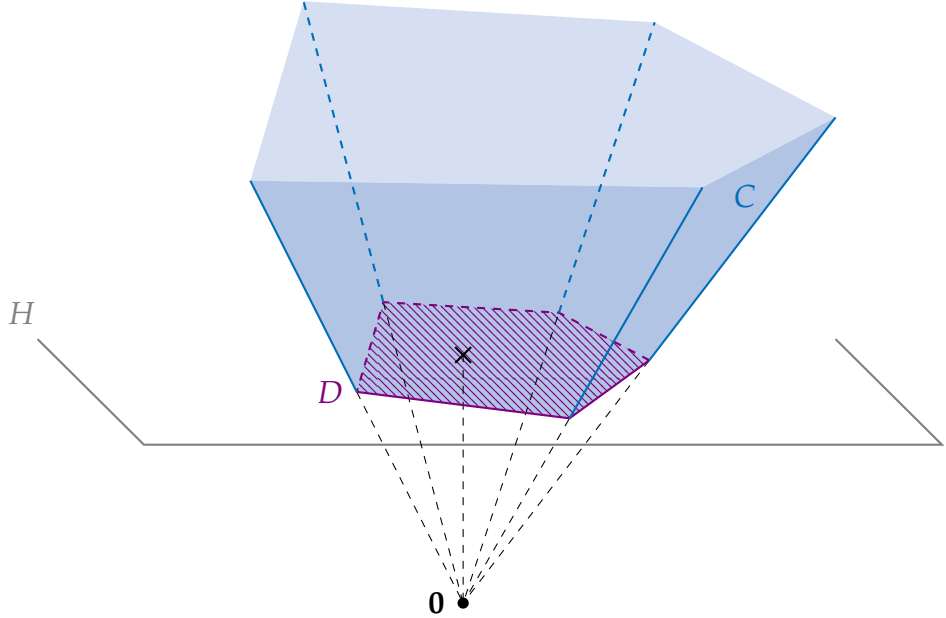
\end{center}
We denote $\mathbf{h}$ the orthogonal projection of the origin $\mathbf{0}$ on the affine hyperplane $H$, and we assume that it is contained in the hypercube $\prod_{i=1}^d (a\EX{i},b\EX{i})$, where $\mathcal{S}_{\mathbf{a},\mathbf{b}}$ is the multi-strip of mod-Gaussian convergence of the sequence $(\XEC_n)_{n \in \N}$. 
If $\rho_n$ is the law of $\XEC_n$, we denote $\widetilde{\rho}_n$ the new probability measure
$$\widetilde{\rho}_n(\!\DD{\xec}) = \frac{\E^{\scal{\mathbf{h}}{\xec}}}{\esper[\E^{\scal{\mathbf{h}}{\XEC_n}}]}\,\rho_n(\!\DD{\xec}).$$
If $\widetilde{\XEC}_n$ follows the law $\widetilde{\rho}_n$, then the new sequence $(\widetilde{\XEC}_n-t_n\mathbf{h})_{n \in \N}$ is again mod-Gaussian convergent in the Laplace sense on the multi-strip $\mathcal{S}_{\mathbf{a}-\mathbf{h},\mathbf{b}-\mathbf{h}}$:
\begin{align*}
\esper\!\left[\E^{\scal{\zec}{\widetilde{\XEC}_n}}\right] &= \frac{\esper[\E^{\scal{\zec + \mathbf{h}}{\XEC_n}} ]}{\esper[\E^{\scal{\mathbf{h}}{\XEC_n}}]} = \E^{\frac{t_n\|\zec\|^2}{2}}\,\E^{t_n\scal{\zec}{\mathbf{h}}}\, \frac{\psi_n(\zec+\mathbf{h})}{\psi_n(\mathbf{h})};\\
\E^{-\frac{t_n\|\zec\|^2}{2}}\,\esper\!\left[\E^{\scal{\zec}{\widetilde{\XEC}_n-t_n\mathbf{h}}}\right] &= \frac{\psi_n(\zec+\mathbf{h})}{\psi_n(\mathbf{h})} \to_{n \to \infty} \frac{\psi(\zec+\mathbf{h})}{\psi(\mathbf{h})}.
\end{align*}
Then,
\begin{align*}
\proba[\XEC_n \in t_nC] &= \int_{\R^d} 1_{(\xec \in t_nC)}\,\rho_n(\!\DD{\xec}) = \esper\!\left[\E^{\scal{\mathbf{h}}{\XEC_n}}\right]\int_{\R^d} 1_{(\xec \in t_nC)}\,\E^{-\scal{\mathbf{h}}{\xec}}\,\widetilde{\rho}_n(\!\DD{\xec}) \\
&= \E^{-\frac{t_n\|\mathbf{h}\|^2}{2}}\,\psi_n(\mathbf{h})\,\int_{\R^d} 1_{(\yec \in \sqrt{t_n}C_0)}\,\E^{-\sqrt{t_n}\scal{\mathbf{h}}{\yec}}\,\mu_n(\!\DD{\yec})
\end{align*}
where $\mu_n$ denotes the law of $(\widetilde{\XEC}_n-t_n\mathbf{h})/\sqrt{t_n}$, and $C_0 = C - \mathbf{h} = \{\yec \in \R^d\,|\, \yec+\mathbf{h} \in C\}$. To compute the integral $I$, notice that if $D_0 = D-\mathbf{h}$, then every element of $D_0$ is orthogonal to $\mathbf{h}$, and on the other hand,
$$C_0 = \{\yec = s\mathbf{h}+(s+1)\mathbf{d}\,|\,s\in \R_+ \text{ and }\mathbf{d} \in D_0 \}.$$
We set $C_{(0,s)} = \{\yec = t\mathbf{h}+(t+1)\mathbf{d}\,|\,t\in [0,s] \text{ and }\mathbf{d} \in D_0 \}$; by definition, $C_0 = C_{(0,+\infty)}$. On the other hand, 
\begin{align*}
I &= \int_{s=0}^\infty \frac{d(\mu_n(\sqrt{t_n}\,C_{(0,s)}))}{ds}\,\E^{-t_ns\|\mathbf{h}\|^2}\DD{s} \\ 
&= t_n\|\mathbf{h}\|^2\,\int_{s=0}^\infty \mu_n(\sqrt{t_n}\,C_{(0,s)})\,\E^{-t_ns\|\mathbf{h}\|^2}\DD{s}
\end{align*}
by using on the first line the identity $\E^{-\sqrt{t_n}\scal{\mathbf{h}}{\yec}} = \E^{-t_ns\|\mathbf{h}\|^2}$ if $\yec = \sqrt{t_n}(s\mathbf{h}+(s+1)\mathbf{d})$, and on the second line an integration by parts. We can now introduce the modified Gaussian measure
$$\nu_n(\!\DD{\xec}) = \frac{1}{(2\pi)^{\frac{d}{2}}}\,\E^{-\frac{\|\xec\|^2}{2}}\,\left(1+\frac{\scal{\nabla\psi(\mathbf{h})}{\xec}}{\psi(\mathbf{h})\sqrt{t_n}}\right)\DD{\xec}.$$
By Theorem \ref{thm:modifiedberryesseen}, for every Borel convex subset of $\R^d$, $\mu_n(C) = \nu_n(C) +  o(\frac{1}{\sqrt{t_n}})$, with a $o(\cdot)$ that is uniform over the class of all Borel convex subsets. Therefore,
$$I = \int_{s=0}^\infty \nu_n(\sqrt{t_n}\,C_{(0,s)})\,t_n\|\mathbf{h}\|^2\,\E^{-t_ns\|\mathbf{h}\|^2}\DD{s} + o\!\left(\frac{1}{\sqrt{t_n}}\right)$$
with a $o(\cdot)$ that does not depend on the cone $C$. Denote $J$ the last integral; we have
\begin{align*}
&\frac{J}{(t_n)^{\frac{d}{2}} \|\mathbf{h}\|}\\
&=\int_{\substack{ t\geq 0\,\,\,\, \\ \mathbf{d} \in D_0}} \frac{(t+1)^{d-1}}{(2\pi)^{\frac{d}{2}}} \left(1+\frac{\scal{\nabla\psi(\mathbf{h})}{t\mathbf{h}+(t+1)\mathbf{d}}}{\psi(\mathbf{h})}\right) \E^{-t_n\left((t+\frac{t^2}{2})(\|\mathbf{h}\|^2+\|\mathbf{d}\|^2) +\frac{\|\mathbf{d}\|^2}{2}\right)}\DD{t}\DD{\mathbf{d}}.
\end{align*}
We apply the Laplace method to 
\begin{align*}
&\int_{t=0}^{\infty}\frac{(t+1)^{d-1}}{(2\pi)^{\frac{d}{2}}}\,\left(1+\frac{\scal{\nabla\psi(\mathbf{h})}{t\mathbf{h}+(t+1)\mathbf{d}}}{\psi(\mathbf{h})}\right)\,\E^{-t_n\left((t+\frac{t^2}{2})(\|\mathbf{h}\|^2+\|\mathbf{d}\|^2) +\frac{\|\mathbf{d}\|^2}{2}\right)}\DD{t} \\
&=\frac{1}{(2\pi)^{\frac{d}{2}}}\,\left(1+\frac{\scal{\nabla\psi(\mathbf{h})}{\mathbf{d}}}{\psi(\mathbf{h})}+O\!\left(\frac{1}{t_n}\right)\right)\,\frac{\E^{-\frac{t_n\|\mathbf{d}\|^2}{2}}}{t_n(\|\mathbf{h}\|^2+\|\mathbf{d}\|^2)},
\end{align*}
see \cite[\S19.2]{Zor04}. In this approximation, it is easily seen that if $D_0$ is a bounded domain and if $\mathbf{h}$ also stays bounded, then the constant in the $O(\cdot)$ of the remainder can be taken uniformly. We conclude:

\begin{proposition}[Asymptotics of the probabilities of cones]\label{prop:asymptoticsprobabilitycone}
Consider a mod-Gaussian convergent sequence $(\XEC_n)_{n\in \N}$ with parameters $t_nI_d$. We fix a bound $M>0$, and we suppose that:
\begin{itemize}
 	\item The vector $\mathbf{h}$ belongs to $\prod_{i=1}^d (a\EX{i},b\EX{i})$, and$\|\mathbf{h}\| \leq M$,
 	\item The domain $D_0$ is a convex part of $\mathbf{h}^\perp$ such that $\sup\{\|\mathbf{d}\| \in D_0\} \leq M$.
 \end{itemize}  If $C$ is the cone based on the domain $D=D_0 + \mathbf{h}$, then
\begin{align*}
&(2\pi)^{\frac{d}{2}}\,\E^{\frac{t_n\|\mathbf{h}\|^2}{2}}\,\frac{1}{\psi_n(\mathbf{h})}\,\proba[\XEC_n \in t_n C] \\
&= (t_n)^{\frac{d}{2}-1}\int_{\mathbf{d}\in D_0}\!\! \left(1+\frac{\scal{\nabla\psi(\mathbf{h})}{\mathbf{d}}}{\psi(\mathbf{h})}+O\!\left(\frac{1}{t_n}\right)\right)\frac{\|\mathbf{h}\|\,\E^{-\frac{t_n\|\mathbf{d}\|^2}{2}}}{\|\mathbf{h}\|^2+\|\mathbf{d}\|^2}\DD{\mathbf{d}}+ o\!\left(\frac{1}{\sqrt{t_n}}\right)
\end{align*}
with constants in the $O(\frac{1}{t_n})$ and the $o(\frac{1}{\sqrt{t_n}})$ that only depend on the bound $M$ (and on the sequence $(\XEC_n)_{n \in \N}$).
\end{proposition}
\begin{proof}
The only non-trivial fact that remains to be proven after the previous discussion is that the $o(1/\sqrt{t_n})$ can be taken uniformly if $\mathbf{h}$ and $D_0$ stay bounded by $M$. This is a consequence of the remark following Theorem \ref{thm:modifiedberryesseen}.
\end{proof}
\medskip

\subsection{Approximation of spherical sectors by unions of cones}\label{subsec:approximationcone}
We still assume until further notice that $K=I_d$. Recall that our final goal is to compute the asymptotics of the probabilities $\proba[\XEC_n \in t_nB]$, where $B$ is a \emph{spherical sector}, that is to say that it can be written as $B=[b,+\infty) \times S$ where $S$ is some measurable part of the sphere $\sph^{d-1}$. After Proposition \ref{prop:asymptoticsprobabilitycone}, a natural method consists in approximating the spherical sector $B$ by a disjoint union of cones with small bases $D_0$ and vectors $\mathbf{h}$ placed on the sphere of radius $b$; see Figure \ref{fig:approximationsectors}. This is only possible if $S$ is a sufficiently regular subset of $\sph^{d-1}$. We develop hereafter an \emph{ad hoc} notion of regularity, which is akin to \emph{Jordan mesurability} in $\R^d$, but with respect to the sphere. We were not able to find an adequate reference for the notion of Jordan mesurability on manifolds; for the Euclidean case, we refer to \cite[Section 1.1.2]{Tao11} and \cite[Chapter 3]{Spi65}.

\begin{center}
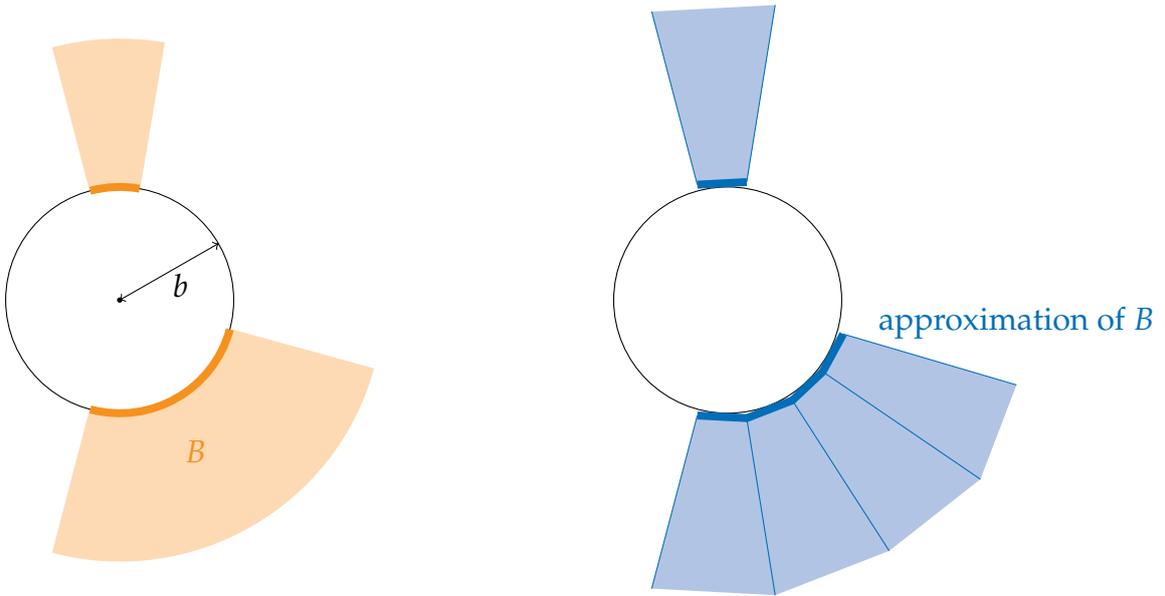
\begin{figure}[ht]
\begin{tikzpicture}[scale=1]
\draw (0,0) circle (1.5cm);
\draw (8,0) circle (1.5cm);
\draw [BurntOrange!30!white, line width=55] ([shift=(80:2.5cm)]0,0) arc (80:105:2.5cm);
\draw [BurntOrange, line width=3] ([shift=(80:1.5cm)]0,0) arc (80:105:1.5cm);
\draw [BurntOrange!30!white, line width=55] ([shift=(255:2.5cm)]0,0) arc (255:345:2.5cm);
\draw [BurntOrange, line width=3] ([shift=(255:1.5cm)]0,0) arc (255:345:1.5cm);
\draw [<->] (0,0) -- (1.30,0.75);
\fill (0,0) circle (1pt);
\draw (0.8,0.2) node {$b$};
\draw (1,-2) node {\textcolor{BurntOrange}{$B$}};
\fill [NavyBlue!30!white] (7.6,1.53) -- (8.25,1.565) -- (8.625,3.9125) -- (7,3.825);
\draw [NavyBlue, line width=3] (7.6,1.53) -- (8.25,1.565);
\fill [NavyBlue!30!white] (7.6,-1.53) -- (8.25,-1.565) -- (8.85,-1.33) -- (9.25,-0.95) -- (9.52,-0.45) -- (11.8,-1.125) -- (11.325,-2.375) -- (10.125,-3.325) -- (8.625,-3.9125) -- (7,-3.825) ;
\draw [NavyBlue, line width=3] (7.6,-1.53) -- (8.25,-1.565) -- (8.85,-1.33) -- (9.25,-0.95) -- (9.52,-0.45);
\draw [NavyBlue] (7.6,1.53) -- (7,3.825);
\draw [NavyBlue] (8.25,1.565) -- (8.625,3.9125);
\draw [NavyBlue] (7.6,-1.53) -- (7,-3.825);
\draw [NavyBlue] (8.25,-1.565) -- (8.625,-3.9125);
\draw [NavyBlue] (8.85,-1.33) -- (10.125,-3.325);
\draw [NavyBlue] (9.25,-0.95) -- (11.325,-2.375);
\draw [NavyBlue] (9.52,-0.45) -- (11.8,-1.125);
\draw (11.8,-0.3) node {\textcolor{NavyBlue}{approximation of $B$}};
\end{tikzpicture}
\caption{Approximation of a spherical sector $B = [b+\infty) \times S$ by a disjoint union of cones.\label{fig:approximationsectors}}
\end{figure}
\end{center}

Let $E$ be the hypercube $[-1,1]^d$. If $A\subset \partial E$ is a rectangular part on one of the face of $\partial E$, we call \emph{hypercubic facet} associated to $A$ on the sphere of radius $b$ the set of vectors 
$$A_b=\{\mathbf{v}\in \R^d\,|\, \|\mathbf{v}\|=b \text{ and }\mathbf{v}\text{ is colinear to a vector in }A\},$$
see Figure \ref{fig:hypercubicfacet} for an example in dimension $d=3$. 

\begin{center}
\begin{figure}[ht]
\begin{tikzpicture}[scale=1.5]
\draw [dashed] (0,0) -- (1,0) -- (1.5,0.5) -- (1.5,1.5) -- (0.5,1.5) -- (0,1) -- (0,0);
\draw [dashed] (0,1) -- (1,1) -- (1,0);
\draw [dashed] (1,1) -- (1.5,1.5);
\fill [Red!30!white] (1.1,0.6) -- (1.25,0.75) -- (1.25,1.0) -- (1.1,0.85); 
\draw [Red] (1.1,0.6) -- (1.25,0.75) -- (1.25,1.0) -- (1.1,0.85) -- (1.1,0.6); 
\draw (0.75,0.75) circle (2cm);
\draw ([shift=(251.3:6.25cm)]0.75,6.7) arc (251.3:288.7:6.25cm);
\draw [Red] (1.1,0.6) -- (1.73,0.53);
\draw [Red] (1.1,0.85) -- (1.7,1.13);
\draw [Red] (1.25,1) -- (2.55,1.62);
\draw [Red] (1.25,0.75) -- (2.75,0.77);
\filldraw [Purple!30!white] ([shift=(8.1:3cm)]-1.25,0.75) arc (8.1:-4.5:3cm) -- ([shift=(288.7:6.25cm)]0.75,6.7) arc (288.7:279:6.25cm) -- ([shift=(0:2cm)]0.75,0.75) arc (0:25:2cm) -- ([shift=(302.2:6.25cm)]-0.75,6.9) arc (302.2:293.2:6.25cm);
\draw [Purple, line width = 2] ([shift=(8.1:3cm)]-1.25,0.75) arc (8.1:-4.5:3cm)  ([shift=(288.7:6.25cm)]0.75,6.7) arc (288.7:279:6.25cm) ([shift=(0:2cm)]0.75,0.75) arc (0:25:2cm) ([shift=(302.2:6.25cm)]-0.75,6.9) arc (302.2:293.2:6.25cm);
\draw (2.85,2.1) node {$\sph^{d-1}(b)$};
\draw (2.3,1) node {$\textcolor{Purple}{A_b}$};
\end{tikzpicture}
\caption{An hypercubic facet on a $2$-dimensional sphere.\label{fig:hypercubicfacet}}
\end{figure}
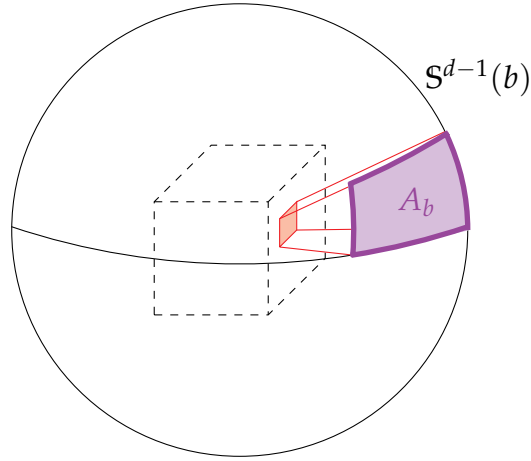
\end{center}

The notion of hypercubic facet allows one to transform the rectangular partitions of the bound\-ary $\partial E$ of the hypercube into partitions of the sphere of radius $b$. Since rectangles can be used to define Jordan mesurability on $\R^d$ and thereby on $\partial E$, by projection, we are led to the analogous notion for subsets of the sphere:
\begin{definition}[Measurability by hypercubic facets]\label{def:measurabilityhypercubic}
Let $S$ be a measurable part of $\sph^{d-1}(b)$. One says that $S$ is measurable by hypercubic facets if, for every $\eps>0$, there exists two subsets $T_1$ and $T_2$ of the sphere such that:
\begin{enumerate}
	\item $T_1$ and $T_2$ are finite unions of hypercubic facets on $\sph^{d-1}(b)$;
	\item $T_1 \subset S \subset T_2$;
	\item $\mu_{\mathrm{surface}}(T_2\setminus T_1) \leq \eps$.
\end{enumerate}
\end{definition}
\begin{proposition}[Topological criterion for measurability]
A subset $S$ of the sphere of radius $b>0$ is mesurable by hypercubic facets if and only if its topological boundary $\partial S$ has zero Lebesgue measure: $\mu_{\mathrm{surface}}(\partial S)=0$.
\end{proposition}
\begin{proof}
The map $\psi$ which projects the hypercube $\partial E$ onto the sphere $\sph^{d-1}(b)$ is an homeomorphism which sends sets of Lebesgue measure zero to sets of Lebesgue surface measure zero, and conversely. On the other hand, $S \subset \sph^{d-1}(b)$ is measurable by hypercubic facets if and only if $\psi^{-1}(S)$ is a subset of $\partial E$ measurable by rectangles, that is to say Jordan measurable in the usual sense. However, it is well known that a bounded set of $\R^d$ is Jordan measurable if and only if its boundary has Lebesgue measure equal to zero; see \cite[p.~56]{Spi65}. So, $S$ is measurable by hypercubic facets if and only if $\partial(\psi^{-1}(S))=\psi^{-1}(\partial S)$ has zero Lebesgue measure, and by the remark at the beginning of the proof, this is equivalent to the fact that $\partial S$ has zero Lebesgue surface measure.
\end{proof}

Let us now relate the notion of mesurability by hypercubic facets to the problem of approximation of spherical sectors by unions of cones. To begin with, notice that if $S_b$ is an hypercubic facet on $\sph^{d-1}(b)$, then the corresponding spherical sector is naturally approximated by two cones with convex bases. Indeed, the boundary of the hypercubic facet $S_b$ is the convex domain in the sense of geodesics of $\sph^{d-1}$ that is delimited by $2^{d-1}$ points $\mathbf{v}_1,\ldots,\mathbf{v}_{2^{d-1}}$, which are the projections of the corners of the rectangle $A$ included in $\partial E$ that corresponds to $S_b=A_b$. Let $\mathbf{h}$ be the barycenter of these $2^{d-1}$ points (in the Euclidean sense); it is easily seen by recurrence on the dimension that $\|\mathbf{h}\| \leq b$, the inequality being strict as soon as the hypercubic facet is non-degenerate (\emph{i.e.} with positive surface measure). For every positive real number $r>0$, we can consider the domain $D(r)$ that is defined as the part of the affine hyperplane
$$H(r)=\left\{r\frac{\mathbf{h}}{\|\mathbf{h}\|}+\mathbf{d},\,\,\mathbf{d} \in \mathbf{h}^\perp\right\}$$
which is the convex hull of the $2^{d-1}$ points of 
$$H(r)\cap \left(\bigcup_{i=1}^{2^{d-1}} \R_+\mathbf{v}_i\right).$$
For $r$ small enough, the cone $C(r)$ based on the domain $D(r)$ contains $S_b$ and therefore the whole spherical sector $S_b \times[1,+\infty)$. We call \emph{negative approximation} of the spherical sector $B=S_b \times [1,+\infty)$ the cone $C(r)$ based on the domain $D(r)$, with $r \in (0,b]$ maximal such that 
$$ S_b \times [1,+\infty) \subset C(r).$$
We then denote $C(r)=C_{-}(B)$ this negative approximation; it satisfies $B \subset C_{-}(B)$, and in some sense it is the smallest possible cone with this property. On the other hand, note that the domain $D(b)$ lies on the outside of the sphere $\sph^{d-1}(b)$, because it is a part of an hyperplane tangent to the sphere. Therefore, the spherical sector $B=S_b \times [1,+\infty)$ contains the cone $C(b)$, which we call the \emph{positive approximation} of the spherical sector, and denote $C(b)=C_{+}(B)$. Thus, we have
$$C_+(B) \subset B \subset C_-(B).$$ 
We refer to Figure \ref{fig:negativepositive} for drawings in dimension $d=2$ that clarify the discussion above (in dimension $d=2$, hypercubic facets are simply arcs of circles). The terminology negative/positive recalls that the cone $C_-(B)$ is based on a domain $D(r)$ with $r\leq b$ (beware that the negative approximation is the largest cone).
\begin{center}
\begin{figure}[ht]
\begin{tikzpicture}[scale=0.85]
\draw (0,0) circle (1.5cm);
\draw [dashed,BurntOrange] (10:1.5cm) -- (55:1.5cm);
\draw [<->] (0,0) -- (-20:1.5cm);
\draw (0.7,-0.5) node {$b$};
\draw [BurntOrange!30!white, line width=50] ([shift=(10:2.5cm)]0,0) arc (10:55:2.5cm);
\draw [BurntOrange, line width=2] ([shift=(10:1.5cm)]0,0) arc (10:55:1.5cm);
\draw (6,0) circle (1.5cm);
\draw (12,0) circle (1.5cm);
\fill (10:1.5cm) circle (1.5pt);
\fill (55:1.5cm) circle (1.5pt);
\fill (32.5:1.385cm) circle (1.5pt);
\draw (32.5:1.15cm) node {$\mathbf{h}$};
\draw (1.8,0.05) node {$\mathbf{v}_1$};
\draw (0.57,1.17) node {$\mathbf{v}_2$};
\draw (2.1,1.4) node {\textcolor{BurntOrange}{$B$}};
\begin{scope}[shift={(6cm,0)}]
\fill [Red!30!white] (10:3.5cm) -- (10:1.5cm) -- (55:1.5cm) -- (55:3.5cm);
\draw [Red, line width = 2] (10:1.5cm) -- (55:1.5cm);
\draw (2.1,1.4) node {\textcolor{Red}{$C_-(B)$}};
\draw [dashed] (0,0) -- (55:3.7cm);
\draw [dashed] (0,0) -- (10:3.7cm);
\draw (1.8,0.05) node {$\mathbf{v}_1$};
\draw (0.57,1.17) node {$\mathbf{v}_2$};
\fill (10:1.5cm) circle (1.5pt);
\fill (55:1.5cm) circle (1.5pt);
\end{scope}
\begin{scope}[shift={(12cm,0)}]
\fill [Purple!30!white] (10:3.5cm) -- (10:1.65cm) -- (55:1.65cm) -- (55:3.5cm);
\draw [Purple, line width = 2] (10:1.65cm) -- (55:1.65cm);
\draw (2.1,1.4) node {\textcolor{Purple}{$C_+(B)$}};
\draw [dashed] (0,0) -- (55:3.7cm);
\draw [dashed] (0,0) -- (10:3.7cm);
\draw (1.8,0.05) node {$\mathbf{v}_1$};
\draw (0.57,1.17) node {$\mathbf{v}_2$};
\fill (10:1.5cm) circle (1.5pt);
\fill (55:1.5cm) circle (1.5pt);
\end{scope}
\end{tikzpicture}
\caption{Negative and positive approximations of a spherical sector based on an hypercubic facet.\label{fig:negativepositive}}
\end{figure}
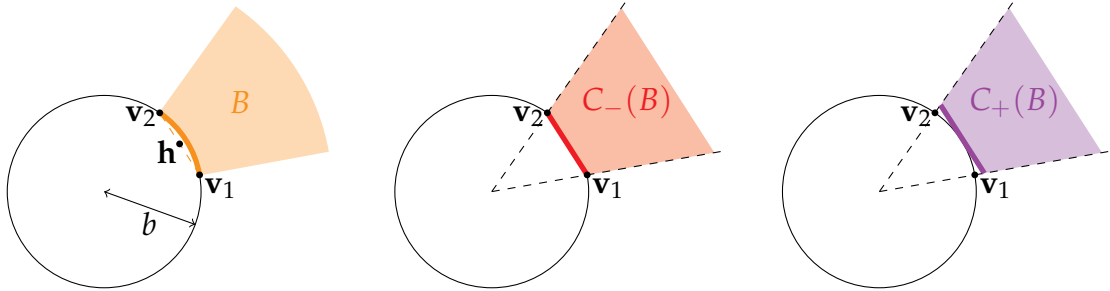
\end{center}
Consider an hypercubic facet $S_b$, with $C_-(B)$ and $C_+(B)$ negative and positive approximations of the spherical sector $B=S_b \times [1,+\infty)$. We say that the pair $(C_-(B),C_+(B))$ is an approximation of \emph{level} $\eps>0$ of $B$ if
\begin{align*}
\min \{\|\mathbf{v}\|,\,\,\mathbf{v} \in D_-(B)\} &\geq b-\eps;\\
\max \{\|\mathbf{v}\|,\,\,\mathbf{v} \in D_+(B)\} &\leq b+\eps,
\end{align*}
where $D_-(B)$ and $D_+(B)$ are the two convex bases on which the cones $C_-(B)$ and $C_+(B)$ are based. The radius $b$ being fixed, if the diameter (for instance in the geodesic sense) of an hypercubic facet $S_b$ is smaller than $\eps$, then the corresponding pair of approximations is of level $O(\eps^2)$, where the constants in the $O(\cdot)$ depend only on $b$ and the dimension $d$.\medskip

We are now ready to deal with approximations of general spherical sectors (not necessarily based on hypercubic facets). Let $S$ be a part of $\sph^{d-1}$ that is measurable by hypercubic facets, and $B= S \times [b,+\infty)$. A negative approximation of the spherical sector $B$ is given by:
\begin{enumerate}
 	\item a finite union $T_2$ of hypercubic facets $A_{b}\EX{1},\ldots,A_{b}\EX{m}$ of $\sph^{d-1}(b)$, such that 
 	$$T_2=\bigcup_{i=1}^m A_{b}\EX{i} \supset S_b; $$
 	\item and a radius $b_- \in (0,b)$, such that for every $i \in \lle 1,m\rre$, 
 	$$b_- \leq r\EX{i}.$$
 	Here, $r\EX{i}$ is the radius such that $C\EX{i}(r\EX{i})=C_-(B\EX{i})$; $B\EX{i}$ (respectively, $(C\EX{i}(r))_{r \in \R_+}$) is the spherical sector (resp., the decreasing family of cones) based on the hypercubic facet $A_b\EX{i}$.
 
 \end{enumerate} 
Combining the two hypotheses, we see that if $((A_b\EX{1},\ldots,A_b\EX{m}),b_-)$ is a negative approximation of the spherical sector $B$, then one has
$$B = S \times [b,+\infty) \subset \bigcup_{i=1}^m \left(A_{b}\EX{i} \times [1,+\infty)\right) \subset \bigcup_{i=1}^m C_-(B\EX{i}) \subset \bigcup_{i=1}^m C\EX{i}(b_-). $$
On the other hand, given a negative approximation, we can assume without loss of generality that the interiors of the hypercubic facets $A_b\EX{i}$ are all disjoint. Similarly, a positive approximation of a general spherical sector $B = S \times [b,\infty)$ with $S$ measurable by hypercubic facets is given by a finite union $T_1$ of hypercubic facets $A_{b}\EX{1},\ldots,A_{b}\EX{l}$ of $\sph^{d-1}(b)$, such that 
$$T_1=\bigcup_{i=1}^l A_{b}\EX{i} \subset S_b.$$
Then, 
$$\bigcup_{i=1}^l C_+(B\EX{i}) = \bigcup_{i=1}^l C\EX{i}(b) \subset \bigcup_{i=1}^l \left(A_b\EX{i} \times [1,+\infty) \right) \subset S \times [b,+\infty) = B.$$
For instance, Figure \ref{fig:approximationsectors} represents a positive approximation of a spherical sector $B$.
\medskip

\begin{proposition}[Existence of negative and positive approximations]\label{prop:goodapprox}
Fix $\eta>0$. For every part $S \subset \sph^{d-1}$ that is measurable by hypercubic facets, and for every $\eps>0$ small enough (this condition depending on $S$ and $\eta$), there exists a positive approximation $T_1$ and a negative approximation $(T_2,b_-)$ of the spherical sector $B=S\times [b,+\infty)$, with:
\begin{enumerate}
	\item $\mu_{\mathrm{surface}}(T_2\setminus T_1) \leq \eta$;
	\item the interiors of the hypercubic facets in $T_1$ (respectively, in $T_2$) are disjoint;
	\item the number of hypercubic facets in $T_1$ and $T_2$ is a $O(\frac{1}{\eps^{d-1}})$;
	\item these hypercubic facets are all of diameter smaller than $O(\eps)$, and $b-b_-=O(\eps^2)$, where the constants in the $O(\cdot)$'s only depend on $b$ and $d$.
\end{enumerate}
\end{proposition}
\begin{proof}
Fix $\eta>0$. In the definition of Jordan mesurability of a part of the hypercube $\partial E$, we can take approximations of sets by unions of cells that have disjoint interiors, and are all taken from a regular grid traced on $\partial E$. However, the notion of measurability by hypercubic facets (Definition \ref{def:measurabilityhypercubic}) derives from the notion of Jordan mesurability by the projection $\psi : \partial E \to \sph^{d-1}(b)$. Consequently, if $S \subset \sph^{d-1}(b)$ is measurable by hypercubic facets, then for every $\eta >0$, there exists a pair $(T_1,T_2)$ of unions of hypercubic facets such that $T_1 \subset S_b \subset T_2$ and $\mu_{\mathrm{surface}}(T_2\setminus T_1) \leq \eta$, and moreover, one can assume that the hypercubic facets in $T_1$ and $T_2$ have disjoint interiors, and that they all come from the projection on $\sph^{d-1}(b)$ of a regular grid on $\partial E$. Now, given a regular grid on $\partial E$, one can split every cell of the grid in $2^{d-1}$ subcells with half the size of the original cell. Thus, given $\eps>0$ smaller than the size of the cells of the original grid, one can assume without loss of generality that the hypercubic facets of $T_1$ and $T_2$ all come from a regular grid on $\partial E$ of mesh smaller than $\eps$. The number of cells in such a grid is $O(\frac{1}{\eps^{d-1}})$, so $T_1$ and $T_2$ have less than $O(\frac{1}{\eps^{d-1}})$ hypercubic facets. Finally, each hypercubic facet $A_b$ of $T_1$ or $T_2$ has diameter smaller than $O(\eps)$, and then, the pair of approximations $(C_-(B),C_+(B))$ of the spherical sector based on such an hypercubic facet has level $O(\eps^2)$ by a previous remark, so in particular one can choose $b_-\geq b-O(\eps^2)$.
\end{proof}
\medskip

We can finally prove our large deviation results. We start with the case when $K=I_d$:
\begin{proposition}[Large deviations in the isotropic case $K=I_d$]\label{prop:largedeviationisotropic}
Let $(\XEC_n)_{n \in \N}$ be a sequence of random vectors in $\R^d$, that converges mod-Gaussian in the Laplace sense on the multi-strip $\mathcal{S}_{\mathbf{a},\mathbf{b}}$, with parameters $t_nI_d$ and limit $\psi$. We assume that $\sph^{d-1}(b) \subset \prod_{i=1}^{d} (a\EX{i},b\EX{i})$, and we consider a Borel subset $S \subset \sph^{d-1}$ that is measurable by hypercubic facets, and of non-zero surface measure. Then, if $B$ is the spherical sector $S \times [b,+\infty)$, we have
 $$\proba[\XEC_n \in t_n B] = \left(\frac{t_n}{2\pi}\right)^{\!\frac{d}{2}}\, \frac{1}{t_n b}\, \E^{-\frac{t_nb^2}{2}} \left(\int_{S_b} \psi(\mathbf{s})\,\mu_{\mathrm{surface}}(\!\DD{\mathbf{s}})\right)\,(1+o(1)),$$
 where $S_b = \{b\mathbf{s},\,\,\mathbf{s} \in S\}$.
\end{proposition}

\begin{proof}
Fix $\eta>0$. For every cone $C$ based on a domain $D$ bounded by $M>0$,  Proposition \ref{prop:asymptoticsprobabilitycone} enables us to write 
\begin{align*}
&(2\pi)^{\frac{d}{2}}\,\E^{\frac{t_n \|\mathbf{h}\|^2}{2}}\,\proba[\XEC_n \in t_n C] \\
&= \psi_n(\mathbf{h}) \left((t_n)^{\frac{d}{2}-1}\!\int_{\mathbf{d}\in D_0} \!\!\left(1+\frac{\scal{\nabla\psi(\mathbf{h})}{\mathbf{d}}}{\psi(\mathbf{h})}+O\!\left(\frac{1}{t_n}\right)\right)\frac{\|\mathbf{h}\|\,\E^{-\frac{t_n\|\mathbf{d}\|^2}{2}}}{\|\mathbf{h}\|^2+\|\mathbf{d}\|^2}\DD{\mathbf{d}}+ o\!\left(\frac{1}{\sqrt{t_n}}\right)\right),
\end{align*}
where $\mathbf{h}$ is the projection of the origin $\mathbf{0}$ on the affine hyperplane containing $D$, and where the $o(\frac{1}{\sqrt{t_n}})$ writes more precisely as 
$$\frac{\delta_n(D)}{\sqrt{t_n}}$$
with $|\delta_n(D)| \leq \delta_n$, $(\delta_n)_{n \in \N}$ being a positive sequence that converges to $0$, and that depends only on the choice of the bound $M$. In the following, we fix a positive sequence $(\eps_n)_{n \in \N}$ such that
$$\lim_{n \to \infty} \eps_n\,\sqrt{t_n} = 0 \qquad;\qquad \lim_{n \to \infty} \frac{\eps_n\,\sqrt{t_n}}{(\delta_n)^{\frac{1}{d-1}}} = +\infty .$$
For $n$ large enough, $\eps=\eps_n$ is small enough and one can choose corresponding positive and negative approximations $T_1$ and $(T_2,b_-)$ of the spherical sector $B$, as in Proposition \ref{prop:goodapprox}. We have obviously
$\proba[\XEC_n \in t_n T_1] \leq \proba[\XEC_n \in t_n B] \leq \proba[\XEC_n \in t_n T_2],$
and on the other hand, we can evaluate the probabilities related to $T_1$ and $T_2$ by using the estimates on cones. We write 
$$T_1 = \bigcup_{i=1}^l A_b\EX{i} \qquad;\qquad T_2 = \bigcup_{j=1}^m \widetilde{A}_b\EX{j},$$
where the $A_b\EX{i}$'s (resp., the $\widetilde{A}_b\EX{j}$'s) are hypercubic facets of $\sph^{d-1}(b)$ with disjoint interiors. We denote $D\EX{i}$ the convex domain on which is based the cone $C\EX{i}(b)$ corresponding to the hypercubic facet $A_b\EX{i}$, and $\widetilde{D}\EX{j}$ the convex domain on which is based the cone $\widetilde{C}\EX{j}(b_-)$ corresponding to the hypercubic facet $\widetilde{A}_b\EX{j}$. We also denote $\mathbf{h}_i$ (resp., $\widetilde{\mathbf{h}}_j$) the orthogonal projection of $\mathbf{0}$ on the affine hyperplane that contains $D\EX{i}$ (respectively, $\widetilde{D}\EX{j}$). Notice that $\|\mathbf{h}_i\|=b$, whereas $\|\widetilde{\mathbf{h}}_j\|=b_-$. Moreover, the vectors in the domains $D\EX{i}_0=D\EX{i} - \mathbf{h}_i$ and $\widetilde{D}\EX{j}_0=\widetilde{D}\EX{j} - \widetilde{\mathbf{h}}_j$ have norm bounded by $O(\eps)$. By the discussion at the beginning of the proof,
\begin{align*}
&(2\pi)^{\frac{d}{2}}\,\E^{\frac{t_n b^2}{2}}\,\proba[\XEC_n \in t_nT_1]\\
&= \sum_{i=1}^l (2\pi)^{\frac{d}{2}}\,\E^{\frac{t_n b^2}{2}}\,\proba[\XEC_n \in t_n C\EX{i}(b)] \\
&= \sum_{i=1}^l \psi_n(\mathbf{h}_i)\left((t_n)^{\frac{d}{2}-1}\!\int_{\mathbf{d}\in D\EX{i}_0} \!\!\left(1+\frac{\scal{\nabla\psi(\mathbf{h}_i)}{\mathbf{d}}}{\psi(\mathbf{h}_i)}+O\!\left(\frac{1}{t_n}\right)\!\right)\frac{b\,\E^{-\frac{t_n\|\mathbf{d}\|^2}{2}}}{b^2+\|\mathbf{d}\|^2}\DD{\mathbf{d}}+ \frac{\delta_n(D\EX{i})}{\sqrt{t_n}}\right) \\
&= \sum_{i=1}^l \frac{\psi_n(\mathbf{h}_i)}{b}\left((t_n)^{\frac{d}{2}-1}\!\int_{\mathbf{d}\in D\EX{i}_0} \!\!\left(1+O(\eps)+O\!\left(\frac{1}{t_n}\right)\right)\frac{1+O(t_n\eps^2)}{1+O(\eps^2)}\DD{\mathbf{d}}\right) \\
&\quad+ \sum_{i=1}^l \psi_n(\mathbf{h}_i)\,\frac{\delta_n(D\EX{i})}{\sqrt{t_n}}
\end{align*}
where the $O(\cdot)$'s on the last line are uniform with respect to the index $l$. The second part on the last line is smaller than a constant times
$$\frac{1}{(\eps_n)^{d-1}}\,\frac{\delta_n}{\sqrt{t_n}} = \frac{(t_n)^{\frac{d}{2}-1}}{\left(\frac{\eps_n\,\sqrt{t_n}}{(\delta_n)^{\frac{1}{d-1}}}\right)^{d-1}}  = o\!\left((t_n)^{\frac{d}{2}-1}\right)$$
by the hypotheses made on $(\eps_n)_{n\in \N}$. On the other hand, the first part is
\begin{align*}
&\sum_{i=1}^l \frac{\psi_n(\mathbf{h}_i)\,(t_n)^{\frac{d}{2}}}{t_nb}\left(\int_{\mathbf{d} \in D\EX{i}_0}(1+o(1))\DD{\mathbf{d}}\right) \\
&= \frac{(t_n)^{\frac{d}{2}}}{t_nb} \left(\int_{T_1} \psi_n(\mathbf{s})\,\mu_{\text{surface}}(\!\DD{\mathbf{s}})\right) + o\!\left((t_n)^{\frac{d}{2}-1}\right).
\end{align*}
Similarly, for the upper approximation $T_2$, we get
$$(2\pi)^{\frac{d}{2}}\,\E^{\frac{t_n (b_-)^2}{2}}\,\proba[\XEC_n \in t_nT_1] = \frac{(t_n)^{\frac{d}{2}}}{t_nb_-} \left(\int_{T_2} \psi_n(\mathbf{s})\,\mu_{\text{surface}}(\!\DD{\mathbf{s}})\right) + o\!\left((t_n)^{\frac{d}{2}-1}\right).$$
Since $b-b_-=O(\eps^2)=o(\frac{1}{t_n})$, and since $\psi_n$ converges locally uniformly towards $\psi$, we conclude that 
\begin{align*}
\int_{T_1}\! \psi(\mathbf{s})\,\mu_{\text{surface}}(\!\DD{\mathbf{s}}) + o(1) &\leq \left(\frac{2\pi}{t_n}\right)^{\!\frac{d}{2}}\,t_nb\,\E^{\frac{t_n b^2}{2}}\,\proba[\XEC_n \in t_nB]\\
&\leq \int_{T_2}\! \psi(\mathbf{s})\,\mu_{\text{surface}}(\!\DD{\mathbf{s}}) + o(1).
\end{align*}
As $T_1 \subset S_b \subset T_2$ and $\mu_{\text{surface}}(T_2\setminus T_1) \leq \eta$ can be taken as small as wanted, we obtain finally
$$\left(\frac{2\pi}{t_n}\right)^{\!\frac{d}{2}}\,t_nb\,\E^{\frac{t_n b^2}{2}}\,\proba[\XEC_n \in t_nB] = \int_{S_b}\! \psi(\mathbf{s})\,\mu_{\text{surface}}(\!\DD{\mathbf{s}}) + o(1),$$
which ends the proof of the estimate of large deviations.
\end{proof}
\medskip

\begin{theorem}[Large deviations]\label{thm:largedeviation}
Let $(\XEC_n)_{n \in \N}$ be a sequence of random vectors that is mod-Gaussian convergent in the Laplace sense, with parameters $t_nK$ and limit $\psi(\zec)$. We consider an ellipsoidal sector $B = S \times [b,+\infty)$, where $S$ is a Borel subset of $K^{1/2}(\sph^{d-1})$ that is the image by $K^{1/2}$ of a subset of $\sph^{d-1}$ that is measurable by hypercubic facets, and with non-zero surface measure. Then,
$$
 \proba[\XEC_n \in t_n B]= \left(\frac{t_n}{2\pi}\right)^{\!\frac{d}{2}}\, \frac{1}{t_n b}\, \E^{-\frac{t_nb^2}{2}} \left(\int_{(K^{-1/2}(S))_b} \psi(K^{-1/2}\mathbf{s})\,\mu_{\mathrm{surface}}(\!\DD{\mathbf{s}})\right)\,(1+o(1)).
$$
\end{theorem}
\begin{proof}
Consider $\YEC_n = K^{-1/2}\XEC_n$; it is mod-Gaussian convergent in the Laplace sense, with parameters $t_nI_d$ and limit $\psi(K^{-1/2}\zec)$. By Proposition \ref{prop:largedeviationisotropic}, if $$K^{-1/2}(B) = B'=S'\times [b,+\infty)$$ with $S' = K^{-1/2}(S) \subset \sph^{d-1}$, then
$$\proba[\YEC_n \in t_nB'] =\left(\frac{t_n}{2\pi}\right)^{\!\frac{d}{2}}\, \frac{1}{t_n b}\, \E^{-\frac{t_nb^2}{2}} \left(\int_{(S')_b} \psi(K^{-1/2}\mathbf{s})\,\mu_{\mathrm{surface}}(\!\DD{\mathbf{s}})\right)\,(1+o(1)).$$
Finally, $\proba[\YEC_n \in t_nB'] = \proba[K^{-1/2}\XEC_n \in t_n K^{-1/2}(B)] = \proba[\XEC_n \in t_nB]$.
\end{proof}
\medskip

\begin{remark}
The case $d=1$ of the asymptotic formula of Theorem \ref{thm:largedeviation} allows one to recover Theorem \ref{thm:1Dlargedeviation} if one agrees that the surface measure on the zero-dimensional sphere $\sph^0 = \{+1,-1\}$ is the counting measure.
\end{remark}
\bigskip

\section{Examples of multi-dimensional convergent sequences}\label{sec:examples}

In this section, we give applications of the theory developed in Sections \ref{sec:speed} and \ref{sec:largedeviation}.\medskip

\subsection{First examples}
We start by looking at the two examples of multi-dimensional mod-Gaussian convergence proposed in the introductory Section \ref{subsec:modgauss}.

\begin{example}[Sums of i.i.d.~random vectors]
Let $\mathbf{A}$ be a random variable in $\R^{d}$ with entire Laplace transform $\esper[\E^{\scal{\zec}{\mathbf{A}}}]$, and $(\mathbf{A}_{n})_{n \in \N}$ be a sequence of independent copies of $\mathbf{A}$. We assume that $\mathbf{A}$ is centered and that $\cov(\mathbf{A})=I_d$; up to a linear change of coordinates and a possible reduction of the dimension, these assumptions do not restrict the generality. The rescaled sum
$$\XEC_n = \frac{\SEC_n}{n^{1/3}} =\frac{1}{n^{1/3}}\,\sum_{k=1}^n \mathbf{A}_k $$
is mod-Gaussian convergent in the Laplace sense on $\C^d$, with parameters $n^{1/3}I_d$ and limit 
$$\psi(\zec) = \exp\!\left(\frac{1}{6}\sum_{i,j,k=1}^d \esper[A\EX{i}A\EX{j}A\EX{k}]\,z\EX{i}z\EX{j}z\EX{k}\right).$$
By Theorem \ref{thm:largedeviation}, for any spherical sector $B=S\times [b,+\infty)$ with $S$ part of $\sph^{d-1}$ measurable by hypercubic facets,
\begin{align*}
&\proba[\SEC_n \in n^{2/3}B] \\
&\simeq\frac{n^{\frac{d-2}{6}}\E^{-\frac{n^{1/3}b^2}{2}}}{(2\pi)^{\frac{d}{2}}b}\left(\int_{S_b} \exp\!\left(\frac{1}{6}\sum_{i,j,k=1}^d \esper[A\EX{i}A\EX{j}A\EX{k}]\,x\EX{i}x\EX{j}x\EX{k}\right)\mu_{\mathrm{surface}}(\!\DD{\xec})\right).
\end{align*}
Hence, up to the scale $n^{2/3}$, the sum $\SEC_n$ of i.i.d.~vectors is described by the Gaussian approximation, and at the scale $n^{2/3}$, this normal approximation is corrected by the exponential $\exp(\frac{1}{6}\sum_{i,j,k=1}^d \esper[A\EX{i}A\EX{j}A\EX{k}]\,x\EX{i}x\EX{j}x\EX{k})$. \medskip

If $\mathbf{A}$ is symmetric in law ($\mathbf{A}$ and $-\mathbf{A}$ have the same law), then all the third moments $\esper[A\EX{i}A\EX{j}A\EX{k}]$ vanish, and one has to look at another renormalisation of the sum $\SEC_n$. Hence, for sums of independent symmetric random variables, the new rescaled sum
$$\YEC_n = \frac{\SEC_n}{n^{1/4}}=\frac{1}{n^{1/4}}\,\sum_{k=1}^n \mathbf{A}_k $$
is mod-Gaussian convergent in the Laplace sense, with parameters $n^{1/2}I_d$ and limit 
$$\psi(\zec) = \exp\!\left(\frac{1}{24}\sum_{i,j,k,l=1}^d \kappa(A\EX{i},A\EX{j},A\EX{k},A\EX{l})\,z\EX{i}z\EX{j}z\EX{k}z\EX{l}\right),$$
where 
\begin{align*}
\kappa(A\EX{i},A\EX{j},A\EX{k},A\EX{l}) &=\esper[A\EX{i}A\EX{j}A\EX{k}A\EX{l}] - \esper[A\EX{i}A\EX{j}]\,\esper[A\EX{k}A\EX{l}] \\ 
&\quad- \esper[A\EX{i}A\EX{k}]\,\esper[A\EX{j}A\EX{l}]- \esper[A\EX{i}A\EX{l}]\,\esper[A\EX{j}A\EX{k}].
\end{align*}
The quantities $\kappa(A\EX{i},A\EX{j},A\EX{k},A\EX{l})$ are the joint cumulants of order $4$ of the coordinates of $\mathbf{A}$; we shall detail this theory in Section \ref{subsec:cumulant}. In this setting, by Theorem \ref{thm:largedeviation}, for any spherical sector $B=S\times [b,+\infty)$,
\begin{align*}
&\proba[\SEC_n \in n^{3/4}B] \\
&\simeq\frac{n^{\frac{d-2}{4}}\E^{-\frac{n^{1/2}b^2}{2}}}{(2\pi)^{\frac{d}{2}}b}\left(\int_{S_b} \exp\!\left(\frac{1}{24}\sum_{i,j,k,l=1}^d \kappa(A\EX{i},A\EX{j},A\EX{k},A\EX{l})\,x\EX{i}x\EX{j}x\EX{k}x\EX{l}\right)\mu_{\mathrm{surface}}(\!\DD{\xec})\right).
\end{align*}
A simple consequence of these multi-dimensional results is the loss of symmetry of the random walks on $\Z^{d}$ conditioned to be far away from the origin; this loss of symmetry has also been brought out in dimension $2$ in \cite{Ben19}. Thus, consider the simple $2$-dimensional random walk $\SEC_{n}=\sum_{k=1}^{n}\mathbf{A}_{k}$, where $\mathbf{A}_{k}=(\pm 1,0)$ or $(0,\pm 1)$ with probability $\frac{1}{4}$ for each direction. The non-zero fourth cumulants of $\mathbf{A}$ are $$\kappa((\Re\,\mathbf{A})^{\otimes 4}) =\kappa((\Im\,\mathbf{A})^{\otimes 4}) =\kappa((\Re\,\mathbf{A})^{\otimes 2},(\Im\,\mathbf{A})^{\otimes 2})= -\frac{1}{4},$$
so one has mod-Gaussian convergence of $n^{-1/4}\,\mathbf{S}_n$ with parameters $\frac{n^{1/2}}{2}\,I_2$ and limiting function 
$$\psi(\mathbf{z})=\exp\!\left(-\frac{(z\EX{1})^4+(z\EX{2})^4+6(z\EX{1}z\EX{2})^2}{96}\right).$$ 
Therefore, for every cylindric sector $C(r,\theta_1,\theta_2)=\{R\E^{\I \theta} \in \C\,\,|\,\, R \geq r,\,\,\theta \in (\theta_1,\theta_2)\}$, Theorem \ref{thm:largedeviation} gives the estimate:
$$\proba\!\left[\mathbf{S}_n \in n^{3/4}\,C(r,\theta_1,\theta_2)\right] = \E^{-n^{1/2}\,r^2}\,\left(\int_{\theta_1}^{\theta_2} \psi(2r\E^{\I\theta})\,\frac{\!\DD{\theta}}{2\pi}\right)(1+o(1)).$$
This leads to the following limiting result: if $\mathbf{S}_{n}=R_{n}\,\E^{\I\theta_{n}}$ with $\theta_{n} \in [0,2\pi)$, then
$$
\lim_{n \to \infty} \proba\!\left[\theta_{n} \in (\theta_{1},\theta_{2})\,\big|\,R_{n}\geq rn^{3/4}\right]= \frac{\int_{\theta_1}^{\theta_2} \psi(2r\E^{\I\theta})\DD{\theta}}{\int_{0}^{2\pi} \psi(2r\E^{\I\theta})\DD{\theta}}= \int_{\theta_{1}}^{\theta_{2}}\!\! F(r,\theta)\DD{\theta} 
$$
with $F(r,\theta)=\frac{ \exp\left(-\frac{r^{4}\,(\sin 2\theta)^{2}}{6}\right)}{\int_{0}^{2\pi} \exp\left(-\frac{r^{4}\,(\sin 2\theta)^{2}}{6}\right)\DD{\theta}}$ drawn in Figure \ref{fig:losssymmetry}.
\begin{center}
\begin{figure}[ht]
\includegraphics[width=10cm]{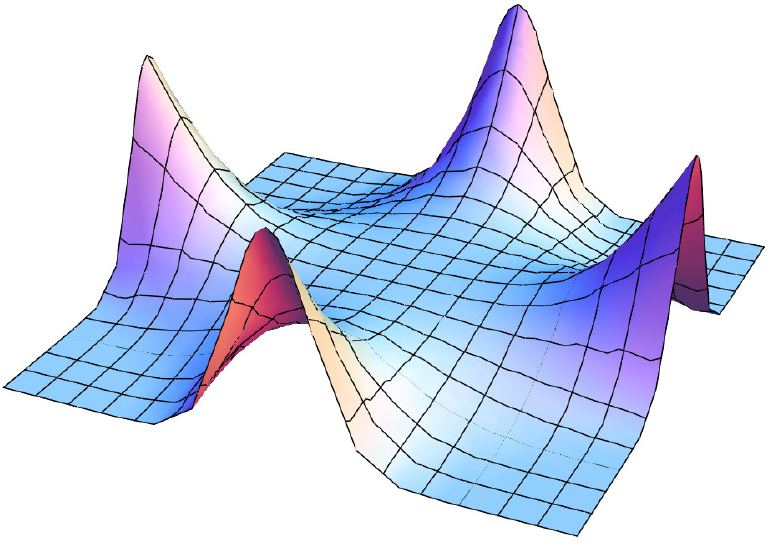}
\vspace{-5mm}
\caption{The function $F(r,\theta)$ measuring the loss of symmetry of the $2$-dimensional random walk $\SEC_{n}$ conditioned to have large radius $\|\SEC_{n}\|\geq rn^{3/4}$ (using \texttt{Mathematica}\textsuperscript{\textregistered}).\label{fig:losssymmetry}}
\end{figure}
\end{center}
\noindent This function gets concentrated around the two axes of $\R^{2}$ when $r \to \infty$, whence a loss of symmetry in comparison to the behavior of the $2$-dimensional Brownian motion (the scaling limit of the random walk). In dimension $d \geq 3$, one obtains the similar result
$$\lim_{n \to \infty} \proba\!\left[\frac{\SEC_{n}}{\|\SEC_{n}\|} \in S\,\bigg|\, \|\SEC_{n}\|\geq rn^{3/4} \right] =K(r)\int_{S} \exp\left(-\frac{r^{4}}{12\,d}\sum_{1 \leq i < j \leq d}(x\EX{i}x\EX{j})^{2}\right) \mu(\!\DD{\mathbf{x}})$$
for any set $S \subset \mathbb{S}^{d-1}$ that is measurable by hypercubic facets. The conditional probability is therefore concentrated around the axes of $\R^{d}$. 
\bigskip

On the other hand, Theorem \ref{thm:generalberryesseen} shows that if $\SEC_n = \sum_{k=1}^n \mathbf{A}_k$ and $\mathbf{A}$ is only assumed to be centered, non-degenerate ($\cov(\mathbf{A}) \in \mathrm{S}_+(d,\R)$) and with a third moment, then
$$\dconv\!\left(\frac{\SEC_n}{n^{1/2}}\,,\,\gauss{d}{\mathbf{0}}{\cov(\mathbf{A})}\right) = O\!\left(\frac{1}{n^{1/6}}\right).$$
This is not optimal, and we shall see in the next section that with the same hypotheses, one can prove a bound $O(\frac{1}{n^{1/2}})$ on the convex distance.
\end{example}

\begin{example}[Characteristic polynomials of random unitary matrices]
Let $U_n$ be a random unitary matrix in $\mathrm{U}(n)$ taken according to the Haar measure. We saw in Section \ref{subsec:modgauss} that the sequence $(\XEC_n = \log \det (I_n-U_n))_{n \in \N}$ is mod-Gaussian convergent on the strip $\mathcal{S}_{(-1,+\infty)}\times \C$, with parameters $\frac{\log n}{2}I_2$ and limit 
$$\psi(\zec)=\frac{G\!\left(1+\frac{z\EX{1}+\I z\EX{2}}{2}\right)\, G\!\left(1+\frac{z\EX{1}-\I z\EX{2}}{2}\right) }{ G(1+z\EX{1})}.$$ 
Therefore, for any $r<1$ and any circular sector $C(r,\theta_1,\theta_2)$,
$$\proba\!\left[\XEC_n \in \frac{\log n}{2}\,C(r,\theta_1,\theta_2)\right] = n^{-\frac{r^2}{4}}\,\left(\int_{\theta=\theta_1}^{\theta_2} \frac{G\!\left(1+\frac{r}{2}\E^{\I\theta}\right)\, G\!\left(1+\frac{r}{2}\E^{-\I\theta}\right)}{G(1+r\cos\theta)}\,\frac{\!\DD{\theta}}{2\pi}\right)(1+o(1)).$$
The function $H(r,\theta) = \frac{G(1+\frac{r}{2}\E^{\I\theta})\, G(1+\frac{r}{2}\E^{-\I\theta})}{G(1+r\cos\theta)}$ takes higher values for $\theta$ close to $\pi$, hence a loss of symmetry of $\XEC_n$ at scale $\log n$.
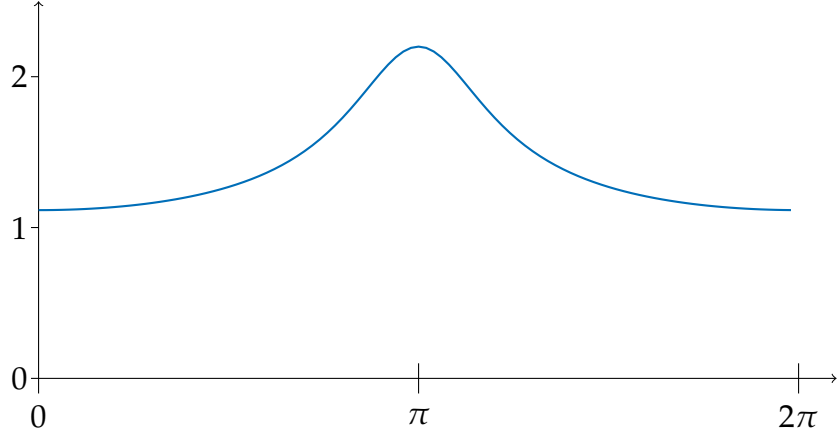
\begin{figure}[ht]
\begin{center}		
\begin{tikzpicture}[xscale=1,yscale=2]
\draw [->] (-0.1,0) -- (10.5,0);
\draw [->] (0,-0.1) -- (0,2.5);
\foreach \x in {5,10}
\draw (\x,-0.1) -- (\x,0.1);
\draw (0,-0.25) node {$0$};
\draw (5,-0.25) node {$\pi$};
\draw (10,-0.25) node {$2\pi$};
\draw (-0.1,1) -- (0,1);
\draw (-0.25,0) node {$0$};
\draw (-0.1,2) -- (0,2);
\draw (-0.25,1) node {$1$};
\draw (-0.25,2) node {$2$};
\draw [smooth,NavyBlue,thick] (.0, 1.1154) -- (.1, 1.1156) -- (.2, 1.1162) -- (.3, 1.1172) -- (.4, 1.1186) -- (.5, 1.1204) -- (.6, 1.1226) -- (.7, 1.1253) -- (.8, 1.1284) -- (.9, 1.1319) -- (1.0, 1.1359) -- (1.1, 1.1404) -- (1.2, 1.1453) -- (1.3, 1.1508) -- (1.4, 1.1568) -- (1.5, 1.1634) -- (1.6, 1.1706) -- (1.7, 1.1784) -- (1.8, 1.1868) -- (1.9, 1.1960) -- (2.0, 1.2059) -- (2.1, 1.2166) -- (2.2, 1.2282) -- (2.3, 1.2408) -- (2.4, 1.2543) -- (2.5, 1.2690) -- (2.6, 1.2849) -- (2.7, 1.3021) -- (2.8, 1.3207) -- (2.9, 1.3409) -- (3.0, 1.3629) -- (3.1, 1.3869) -- (3.2, 1.4130) -- (3.3, 1.4414) -- (3.4, 1.4725) -- (3.5, 1.5065) -- (3.6, 1.5436) -- (3.7, 1.5841) -- (3.8, 1.6283) -- (3.9, 1.6762) -- (4.0, 1.7280) -- (4.1, 1.7834) -- (4.2, 1.8419) -- (4.3, 1.9026) -- (4.4, 1.9641) --(4.5, 2.0242) -- (4.6, 2.0802) -- (4.7, 2.1289) -- (4.8, 2.1669) -- (4.9, 2.1911) -- (5.0, 2.1994) -- (5.1, 2.1911) -- (5.2, 2.1669) -- (5.3, 2.1289) -- (5.4, 2.0802) -- (5.5, 2.0242) -- (5.6, 1.9641) -- (5.7, 1.9026) -- (5.8, 1.8419) -- (5.9, 1.7834) -- (6.0, 1.7280) -- (6.1, 1.6762) -- (6.2, 1.6283) -- (6.3, 1.5841) -- (6.4, 1.5436) -- (6.5, 1.5065) -- (6.6, 1.4725) -- (6.7, 1.4414) -- (6.8, 1.4130) -- (6.9, 1.3869) -- (7.0, 1.3629) -- (7.1, 1.3409) -- (7.2, 1.3207) -- (7.3, 1.3021) -- (7.4, 1.2849) -- (7.5, 1.2690) -- (7.6, 1.2543) -- (7.7, 1.2408) -- (7.8, 1.2282) -- (7.9, 1.2166) -- (8.0, 1.2059) -- (8.1, 1.1960) -- (8.2, 1.1868) -- (8.3, 1.1784) -- (8.4, 1.1706) -- (8.5, 1.1634) -- (8.6, 1.1568) -- (8.7, 1.1508) -- (8.8, 1.1453) -- (8.9, 1.1404) -- (9.0, 1.1359) -- (9.1, 1.1319) -- (9.2, 1.1284) -- (9.3, 1.1253) -- (9.4, 1.1226) -- (9.5, 1.1204) -- (9.6, 1.1186) -- (9.7, 1.1172) -- (9.8, 1.1162) -- (9.9, 1.1156);
\end{tikzpicture}
\caption{The function $H(r,\theta)$ when $r=0.7$.}
\end{center}
\end{figure}

On the other hand, the gradient of $\psi$ at $\mathbf{0}$ is equal to $\mathbf{0}$, so by Theorem \ref{thm:modifiedberryesseen}, we have
$$\dconv\!\left(\sqrt{\frac{2}{\log n}}\,\XEC_n\,,\,\gauss{2}{\mathbf{0}}{I_2}\right) = o\!\left(\frac{1}{\sqrt{\log n}}\right).$$
A more precise analysis shows that this convex distance is actually a $O((\log n)^{-3/2})$; see the arguments of the next paragraph, and \cite[Section 3]{MN22} for the analogous result in dimension $1$ with the real part of the log-characteristic polynomial.
\end{example}
\medskip

\subsection{The method of cumulants}\label{subsec:cumulant}
In dimension $d=1$, assuming that the Laplace transform $\esper[\E^{zX_n}]$ is convergent on a disk around $z=0$, it is possible to reformulate the mod-Gaussian convergence of $(X_n)_{n \in \N}$ in terms of the behavior of the coefficients of the power series 
$$\log \esper[\E^{zX_n}]=\sum_{r=1}^\infty \frac{\kappa^{(r)}(X_n)}{r!}\,z^r,$$
 called the \emph{cumulants} of $(X_n)_{n \in \N}$. We refer to \cite[Chapters 5 and 9]{FMN16} and \cite[Section 4]{FMN19} for developments around this notion. Now, there is a notion of \emph{joint cumulants} (\emph{cf.}~\cite{LS59}) that allows one to generalise this method of cumulants, and to give a numerical criterion of mod-Gaussian convergence in arbitrary dimension $d\geq 1$.

\begin{definition}[Joint cumulant]
Given random variables $Y_1,\ldots,Y_r$ with convergent generating series $\esper[\E^{z_1Y_1+\cdots+z_rY_r}]$, their joint cumulant is
$$\kappa(Y_1,\ldots,Y_r) = \left.\frac{\partial^r}{\partial z_1 \cdots \partial z_r}\right|_{z_1=\cdots=z_r=0} \log\left(\esper[\E^{z_1Y_1+\cdots+z_rY_r}]\right).$$
\end{definition}
    
\noindent The joint cumulant is an homogeneous polynomial of degree $r$ in the joint moments of the variables $Y_1,\ldots,Y_r$. More precisely, if $\mathfrak{Q}_r$ is the set of set partitions of $\lle 1,r\rre$, then
$$\kappa(Y_1,\ldots,Y_r) = \sum_{\pi \in \mathfrak{Q}_r} \mu(\pi)\,\left(\prod_{j=1}^{\ell(\pi)} \,\esper\!\left[\prod_{i \in \pi_j} Y_j\right]\right),$$
where $\mu(\pi) = (-1)^{\ell-1}\,(\ell-1)!$ if $\pi = \pi_1 \sqcup \pi_2 \sqcup \cdots \sqcup \pi_\ell$ is a set partition with $\ell$ parts. For instance, $\kappa(X,Y)=\esper[XY]-\esper[X]\,\esper[Y]$ is the covariance, and
$$\kappa(X,Y,Z) = \esper[XYZ] - \esper[XY]\,\esper[Z] - \esper[XZ]\,\esper[Y] - \esper[YZ]\,\esper[Z] + 2\, \esper[X]\,\esper[Y]\,\esper[Z].$$
Let us recall the main properties of the joint cumulants, which follow readily from their definition:

\begin{proposition}[Properties of joint cumulants]\label{prop:propertiesjointcumulants} Let $Y_1,\ldots,Y_r$ be random variables with a convergent Laplace transform.
\begin{enumerate}
	\item The joint cumulants are multilinear and invariant by permutation of the random variables.
	\item If the variables $Y_1,\ldots,Y_r$ can be separated into two blocks of independent random variables, then $\kappa(Y_1,\ldots,Y_r)=0$.
	\item For any random variable $Y$, the classical $r$-th cumulant $\kappa^{(r)}(Y)$ is equal to the joint cumulant $\kappa(Y,\ldots,Y)$ (with $r$ occurrences of $Y$).
	\item For any random vector $\XEC$ in $\R^d$, assuming that the Laplace transform $\esper[\E^{\scal{\zec}{\XEC}}]$ is convergent, one has the expansion
\begin{align*}
\log \,\esper[\E^{\scal{\zec}{\XEC}}] &= \sum_{r \geq 1} \frac{1}{r!}\, \kappa^{(r)}(z\EX{1}X\EX{1}+\cdots + z\EX{d}X\EX{d}) \\
&=\sum_{r \geq 1} \sum_{(i_1,\ldots,i_r) \in \lle 1,d\rre^r} \frac{\kappa(X\EX{i_1},\ldots,X\EX{i_r})}{r!} \,z\EX{i_1}z\EX{i_2}\cdots z\EX{i_r}.
\end{align*}
\end{enumerate}
\end{proposition}
\medskip

Let $(\SEC_n)_{n \in \N}$ be a sequence of random vectors in $\R^d$, and $K \in \mathrm{S}_{+}(d,\R)$. The following definition is a multi-dimensional analogue of \cite[Definition 28]{FMN19}:

\begin{definition}[Method of cumulants]\label{def:methodmulticum}
One says that $(\SEC_n)_{n \in \N}$ satisfies the hypotheses of the multi-dimensional me\-thod of cumulants with covariance matrix $K$, index $v\geq 3$ and positive parameters $(D_n,N_n,A)$ if, for any choice of coordinates:
\begin{enumerate}[label=(MC\arabic*)]
\item\label{hyp:MC1} The random vectors $\SEC_n$ are centered: $\esper[\SEC_n] = \mathbf{0}$.
\item\label{hyp:MC2} The covariance matrix of $\SEC_n$ is given by
$$\kappa(S_{n}\EX{i},S_{n}\EX{j})= N_n D_n\, K_{ij}\,\left(1+o\!\left(\left(\frac{D_n}{N_n}\right)^{1-\frac{2}{v}}\right)\right)$$
with $\lim_{n \to \infty} \frac{D_n}{N_n}=0$.
\item\label{hyp:MC3} For any $r \geq 3$, 
$$\left|\kappa(S_{n}\EX{i_1},S_{n}\EX{i_2},\ldots,S_{n}\EX{i_r})\right| \leq N_n\,(2D_n)^{r-1}\,A^r\,r^{r-2}.$$
\item\label{hyp:MC4} For any $r \in \lle 3,v-1\rre$, $\kappa(S_{n}\EX{i_1},S_{n}\EX{i_2},\ldots,S_{n}\EX{i_r})=0$.
\item\label{hyp:MC5} There exist limits
$$L_{i_1,i_2,\ldots,i_v} = \lim_{n \to \infty} \frac{\kappa(S_n\EX{i_1},S_n\EX{i_2},\ldots,S_n\EX{i_v})}{N_n\,(D_n)^{v-1}}.$$
\end{enumerate}
\end{definition}

\begin{remark}
The one-dimensional version of Definition \ref{def:methodmulticum} was given in \cite[Section 5.1]{FMN16} and \cite[Section 4.1]{FMN19}. The main difference is that we do not ask the covariance matrix of $\frac{\SEC_n}{N_n D_n}$ to be exactly equal to $K$. Indeed, though the coefficients of $\cov(\SEC_n)$ will have the same order of magnitude $N_n D_n$ in our examples, in general we shall not be able to write $\cov(\SEC_n) = N_nD_n\,K$ with $K$ constant matrix.
\end{remark} \medskip

In the last two paragraphs \ref{subsec:dependencygraph} and \ref{subsec:markov}, we shall see that many random models yield random vectors that satisfy the hypotheses of Definition \ref{def:methodmulticum}, and in particular the bound on cumulants \ref{hyp:MC3}. The purpose of this section is to give the theoretical consequences of the multi-dimensional method of cumulants. We start with the mod-Gaussian convergence and the implied large deviation results:

\begin{theorem}[Method of cumulants and mod-Gaussian convergence]\label{thm:multicumulant}
Let $(\SEC_n)_{n \in \N}$ be a sequence of random vectors that satisfies the hypotheses of the multi-dimensional method of cumulants. We set 
$$\XEC_n = \frac{1}{(N_n)^{1/v}\,(D_n)^{1-1/v}}\,\SEC_n\qquad;\qquad \YEC_n = \frac{1}{(N_nD_n)^{1/2}}\,\SEC_n.$$
\begin{enumerate}
	\item The sequence $(\XEC_n)_{n \in \N}$ is mod-Gaussian convergent in the Laplace sense, with parameters $(\frac{N_n}{D_n})^{1-\frac{2}{v}} K$ and limit
	$$\psi(\zec) = \exp\!\left(\frac{1}{v!}\sum_{i_1,\ldots,i_v=1}^d L_{i_1,\ldots,i_v}\,z\EX{i_1}z\EX{i_2}\cdots z\EX{i_v}\right).$$
	\item Therefore, we have the convergence in law $\YEC_n \rightharpoonup \gauss{d}{\mathbf{0}}{K}$.
	\item Consider an ellipsoidal sector $B = S \times [b,+\infty)$, where $S$ is a subset of $K^{1/2}(\sph^{d-1})$ that is measurable by hypercubic facets, and with non-zero surface measure. Then, 
$$\proba\!\left[\YEC_n  \in  u_n B\right] = \frac{(u_n)^{d-2} }{(2\pi)^{\frac{d}{2}}b}\, \E^{-\frac{(u_n b)^2}{2}} \left(\int_{(K^{-1/2}(S))_b} \psi(K^{-1/2}\mathbf{s})\,\mu_{\mathrm{surface}}(\!\DD{\mathbf{s}})\right)\,(1+o(1))$$
with $u_n = \sqrt{t_n} =  \left(\frac{N_n}{D_n}\right)^{\!\frac{1}{2}-\frac{1}{v}}$.
\end{enumerate}
\end{theorem}

\begin{proof}
The mod-Gaussian convergence follows readily from the hypotheses in Definition \ref{def:methodmulticum}, and from the fourth item of Proposition \ref{prop:propertiesjointcumulants}, which relates the joint cumulants of the coordinates of $\SEC_n$ to the Taylor expansion of the log-Laplace transform of $\SEC_n$. Indeed, 
\begin{align*}
\log \esper[\E^{\scal{\zec}{\XEC_n}}] 
&= \sum_{r \geq 1} \sum_{i_1,\ldots,i_r=1}^d \frac{\kappa(S_n\EX{i_1},\ldots,S_n\EX{i_r})}{r!\,N_n(D_n)^{r-1}}\,\left(\frac{D_n}{N_n}\right)^{\!\frac{r}{v}-1}\,z\EX{i_1}z\EX{i_2}\cdots z\EX{i_r} \\
&=\frac{t_n}{2} \sum_{i,j=1}^d \frac{\cov(S_n\EX{i},S_{n}\EX{j})}{N_nD_n}\,z\EX{i}z\EX{j} + \frac{1}{v!}\sum_{i_1,\ldots,i_v=1}^d \frac{\kappa(S_n\EX{i_1},\ldots,S_n\EX{i_v})}{N_n(D_n)^{v-1}}\,z\EX{i_1}\cdots z\EX{i_v} \\
&\quad+ \text{remainder}
\end{align*}
with a remainder smaller than
 \begin{align*}
 &\sum_{r \geq v+1} \sum_{i_1,\ldots,i_r=1}^d \frac{|\kappa(S_n\EX{i_1},\ldots,S_n\EX{i_r})|}{r!\,N_n(D_n)^{r-1}}\,\left(\frac{D_n}{N_n}\right)^{\!\frac{r}{v}-1}\,|z\EX{i_1}z\EX{i_2}\cdots z\EX{i_r}|\\
 &\leq \frac{N_n}{D_n} \sum_{r \geq v+1} \frac{2^{r-1}\,A^r\,r^{r-2}}{r!} \left(\left(\frac{D_n}{N_n}\right)^{\!\frac{1}{v}} \|\zec\|_1 \right)^r.
 \end{align*}
 The power series is convergent and a $O((D_n/N_n)^{1+\frac{1}{v}} (\|\zec\|_1)^{v+1})$; hence, the remainder goes to $0$. Now, the term with covariances is equivalent to $t_n\,\frac{\zec^t K \zec}{2}$ by Hypothesis \ref{hyp:MC2}, and the term of order $v$ is equivalent to 
 $$\frac{1}{v!}\sum_{i_1,\ldots,i_v=1}^d L_{i_1,\ldots,i_v}\,z\EX{i_1}\cdots z\EX{i_v}$$
 by Hypothesis \ref{hyp:MC5}. This ends the proof of the mod-Gaussian convergence, and the two other points are then immediate consequences of Proposition \ref{prop:clt} and Theorem \ref{thm:largedeviation}.
\end{proof}
\medskip

We now focus on the speed of convergence of $\YEC_n$ to $\gauss{d}{\mathbf{0}}{K}$. By Theorems \ref{thm:generalberryesseen} and \ref{thm:modifiedberryesseen}, under the assumptions \ref{hyp:MC1}-\ref{hyp:MC5}, 
$$\dconv\left(\YEC_n\,,\,\gauss{d}{\mathbf{0}}{K}\right) = o\!\left(\left(\frac{D_n}{N_n}\right)^{\!\frac{1}{2}-\frac{1}{v}}\right),$$
because $\nabla \psi(\mathbf{0})=\mathbf{0}$. Under a slightly stronger hypothesis, one can show a better bound:

\begin{theorem}[Method of cumulants and speed of convergence]\label{thm:superspeed}
Let $(\SEC_n)_{n \in \N}$ be a sequence of random vectors that satisfies \ref{hyp:MC1},
\begin{enumerate}[label=(MC\arabic*{'}),start=2]
	\item \label{hyp:MC2prime} The covariance matrix of $\SEC_n$ is given by
$$\kappa(S_{n}\EX{i},S_{n}\EX{j})= N_n D_n\, K_{ij}\,\left(1+O\!\left(\left(\frac{D_n}{N_n}\right)^{\frac{1}{2}}\right)\right).$$
\end{enumerate}
and \ref{hyp:MC3}. There exists a constant $C=C(d,K,A,B)$ that depends only on $d$, $K$, the constant $A$ of \ref{hyp:MC3} and the constant $B$ in the $O(\cdot)$ of \ref{hyp:MC2prime}, such that for $n$ large enough,
$$ \dconv(\YEC_n\,,\,\gauss{d}{\mathbf{0}}{K}) \leq C\,\sqrt{\frac{D_n}{N_n}}.$$
\end{theorem}

The strategy of proof of Theorem \ref{thm:superspeed} is the following:
\begin{itemize}
	\item We first deal with the fact that $K_n = \cov(\SEC_n)/(N_nD_n)$ is not exactly equal to $K$, by computing the convex distance between two Gaussian distributions (see Lemma \ref{lem:technicaldistancegaussian} below).
	\item Then, we prove a bound on $\Delta_{\eps}(\widehat{\mu}_n,\gauss{d}{\mathbf{0}}{K_n})$ with $\eps=O(\sqrt{D_n/N_n})$ (Lemma \ref{lem:headache}), by using the bounds on cumulants and the fast decay of $\E^{-\ZIEC^t K_n\ZIEC/2}$.
\end{itemize}

\begin{lemma}\label{lem:technicaldistancegaussian}
Let $K_1$ and $K_2$ be two positive-definite symmetric matrices; $\nu_1 = \gauss{d}{\mathbf{0}}{K_1}$ and $\nu_2 = \gauss{d}{\mathbf{0}}{K_2}$. If $K_1$ is fixed, then for any $K_2$ such that $K_2-K_1$ is sufficiently small, one has
$$\Delta_{\eps}(\widehat{\nu}_1,\widehat{\nu}_2) \leq 8\,\sqrt{(2\pi)^d (d+1)!}\,\frac{(\max(1, d\rho(K_1^{-1})))^{\frac{d+3}{2}}}{\sqrt{\det K_1}}\,\rho(K_2-K_1)$$
for any $\eps>0$.
\end{lemma}

\begin{proof}
Since $(K_1)^{-1/2}K_2(K_1)^{-1/2}$ is a positive-definite symmetric matrix, there exists positive eigenvalues $\lambda\EX{1}, \ldots, \lambda\EX{d}$ and an orthogonal matrix $R \in \SO(\R^d)$, such that 
$$R^t (K_1)^{-1/2} K_2 (K_1)^{-1/2} R = D^2=\mathrm{diag}(\lambda\EX{1},\ldots,\lambda\EX{d}).$$
 Set $M=(K_1)^{-1/2}R$. If $\boldsymbol\beta$ is a multi-index of total weight $|\boldsymbol\beta|\leq d+1$, then
\begin{align*}
\int_{D^d_{\eps}}& \left|\frac{\partial^{|\boldsymbol\beta|}(\widehat{\nu}_1-\widehat{\nu}_2)(\ZIEC)}{\partial \ZIEC^{\boldsymbol\beta}}\right|\DD{\ZIEC} \\
&\leq \int_{\R^d} \left|\prod_{i=1}^d \left(\sum_{j=1}^d M_{ij} \frac{\partial}{\partial \xi\EX{j}} \right)^{\!\beta\EX{i}}  \left(\E^{-\frac{\ZIEC{}^t K_1 \ZIEC}{2}} - \E^{-\frac{\ZIEC{}^t K_2 \ZIEC}{2}} \right)\right|\DD{\ZIEC} \\
&\leq \frac{1}{\sqrt{\det K_1}}\,\left(\max_{i \in \lle 1,d\rre} \sum_{j=1}^d |M_{ij}|\right)^{\!|\boldsymbol\beta|} \,\max_{|\boldsymbol\alpha|=|\boldsymbol\beta|} \int_{\R^d} \left|\frac{\partial^{|\boldsymbol\alpha|}}{\partial\XIEC^{\boldsymbol\alpha}} \left(\E^{-\frac{\|\XIEC\|^2}{2}} - \E^{-\frac{\|D\XIEC\|^2}{2}} \right)\right|\DD{\XIEC}.
\end{align*}
We evaluate separately each term. Since $R$ is an orthogonal matrix and $(K_1)^{-1/2}$ is a symmetric matrix, 
\begin{align*}
\max_{i \in \lle 1,d\rre} \sum_{j=1}^d|M_{ij}| &= \max_{\|\mathbf{v} \|_\infty \leq 1} \| M\mathbf{v}\|_{\infty} \\
&\leq \max_{\|\mathbf{v}\|_2 \leq \sqrt{d} } \|M\mathbf{v}\|_2  = \max_{\|\mathbf{v}\|_2 \leq \sqrt{d} } \|(K_1)^{-1/2}\,\mathbf{v}\|_2  \leq \sqrt{d\,\rho(K_1^{-1})}.
\end{align*}
On the other hand,
\begin{align*}
I_{\boldsymbol\alpha}&=\int_{\R^d} \left|\frac{\partial^{|\boldsymbol\alpha|}}{\partial \XIEC^{\boldsymbol\alpha}} \left(\E^{-\frac{\|\XIEC\|^2}{2}} - \E^{-\frac{\|D\XIEC\|^2}{2}} \right)\right|\DD{\XIEC} \\
&= \int_{\R^d} \left|\left(\prod_{i=1}^d H_{\alpha\EX{i}}(\xi\EX{i})\,\E^{-\frac{(\xi\EX{i})^2}{2}}\right) - \left(\prod_{i=1}^d (\lambda\EX{i})^{\frac{\alpha\EX{i}}{2}}\, H_{\alpha\EX{i}}((\lambda\EX{i})^{\frac{1}{2}}\xi\EX{i})\,\E^{-\frac{\lambda\EX{i}(\xi\EX{i})^2}{2}}\right)\right|\DD{\XIEC}.
\end{align*}
 The multi-index $\boldsymbol\alpha=(\alpha\EX{1},\ldots,\alpha\EX{d})$ being fixed, we can assume without loss of generality that the eigenvalues $(\lambda\EX{i})_{i \in \lle 1,d\rre}$ are ordered as follows:
\begin{enumerate}
\setcounter{enumi}{-1}
	\item For $i \in \lle 1,d_0 \rre$, $\alpha\EX{i} =0$. Moreover, $\lambda\EX{1} \leq \lambda\EX{2} \leq \cdots \leq \lambda\EX{d_0}$, and we denote $e_0 \in \lle 0,d_0\rre$ the largest index such that $\lambda\EX{e_0}\leq 1$.
	\item For $i \in \lle d_0+1,d_1\rre$, $\alpha\EX{i}=1$. 
	\item For $i \in \lle d_1+1,d\rre$, $\alpha\EX{i}\geq 2$. Moreover, $\lambda\EX{d_1+1} \leq \lambda\EX{d_1+2} \leq \cdots \leq \lambda\EX{d}$, and we denote $e_2 \in \lle d_1,d\rre$ the largest index such that $\lambda\EX{e_2}\leq 1$.
\end{enumerate}
In the following, we shall use several times the inequality 
$$\int_{\R} |H_{\alpha}(\xi)\,H_{\beta}(\xi)|\,\E^{-\frac{\xi^2}{2}}\DD{\xi} \leq \sqrt{2\pi \,\alpha!\,\beta!},$$
and evaluate the terms of the sum in the right-hand side of
\begin{align*}
 I_{\boldsymbol\alpha} & \leq \sum_{j=1}^d \int_{\R^d} \left| \prod_{i=1}^{j-1} (\lambda\EX{i})^{\frac{\alpha\EX{i}}{2}}\, H_{\alpha\EX{i}}((\lambda\EX{i})^{\frac{1}{2}}\xi\EX{i})\,\E^{-\frac{\lambda\EX{i}(\xi\EX{i})^2}{2}} \prod_{i=j+1}^{d} H_{\alpha\EX{i}}(\xi\EX{i})\,\E^{-\frac{(\xi\EX{i})^2}{2}}\right| \\
 &\qquad\qquad\qquad\qquad \times\Delta(\xi\EX{j},\lambda\EX{j},\alpha\EX{j})\DD{\XIEC} \\
 &\leq \sum_{j=1}^d \left(\prod_{i \neq j}\sqrt{2\pi(\alpha\EX{i})!}\right)\left(\prod_{i<j} (\lambda\EX{i})^{\frac{\alpha\EX{i}-1}{2}}\right)\,\int_{\R} \Delta(\xi\EX{j},\lambda\EX{j},\alpha\EX{j})\DD{\xi\EX{j}},
\end{align*}
where $\Delta(\xi\EX{j},\lambda\EX{j},\alpha\EX{j}) = \left|H_{\alpha\EX{j}}(\xi\EX{j})\,\E^{-\frac{(\xi\EX{j})^2}{2}}-(\lambda\EX{j})^{\frac{\alpha\EX{j}}{2}}\,H_{\alpha\EX{j}}((\lambda\EX{j})^{\frac{1}{2}}\xi\EX{j})\,\E^{-\frac{\lambda\EX{j}(\xi\EX{j})^2}{2}}\right|$. \medskip

We abbreviate $(\lambda\EX{j})^{\frac{\alpha\EX{j}-1}{2}} = \rho\EX{j}$.
\begin{enumerate}
\setcounter{enumi}{-1}
	\item Terms with $j\in \lle 1,d_0\rre$. We then have $\int_{\R} \Delta(\xi\EX{j},\lambda\EX{j},0)\DD{\xi\EX{j}} = \sqrt{2\pi}\,|\rho\EX{j}-1|$, so, if $I_{\boldsymbol\alpha,0}$ is the part of the sum bounding $I_{\alpha}$ with indices $j \in \lle 1,d_0\rre$, then
	\begin{align*}
	I_{\boldsymbol\alpha,0} &= \left(\prod_{i=1}^d \sqrt{2\pi(\alpha\EX{i})!}\right)\left(\sum_{j=1}^{e_0} \left(\rho\EX{j}-1\right) \prod_{i<j} \rho\EX{i}+\sum_{j=e_0+1}^{d_0} \left(1-\rho\EX{j}\right) \prod_{i<j} \rho\EX{i}\right) \\
	&=\sqrt{(2\pi)^d \,\boldsymbol\alpha!}\,\left(2\,\rho\EX{1}\rho\EX{2}\cdots \rho\EX{e_0} - \rho\EX{1}\rho\EX{2}\cdots \rho\EX{d_0} -1\right).
	\end{align*}
	\item Terms with $j \in \lle d_0+1,d_1\rre$. Notice that $\prod_{i<j} \rho\EX{i} = \rho\EX{1}\rho\EX{2}\cdots\rho\EX{d_0} $ for all these indices. On the other hand, 
	$$\int_{\R} \Delta(\xi\EX{j},\lambda\EX{j},1)\DD{\xi\EX{j}} = 4\,\E^{-\frac{\log \lambda\EX{j}}{\lambda\EX{j}-1}}\,\left|(\lambda\EX{j})^{-1}-1\right| \leq 4 \,|(\lambda\EX{j})^{-1}-1|. $$
	Therefore, if $I_{\boldsymbol\alpha,1}$ is the part of the sum bounding $I_{\boldsymbol\alpha}$ with indices $j \in \lle d_0+1,d_1\rre$, then
	$$I_{\boldsymbol\alpha,1} \leq \sqrt{(2\pi)^d \,\boldsymbol\alpha!} \left(\frac{4}{\sqrt{2\pi}}\,\rho\EX{1}\rho\EX{2}\cdots \rho\EX{d_0}\right)\,\sum_{j =d_0+1}^{d_1} \left|(\lambda\EX{j})^{-1}-1\right|.$$
	\item Terms with $j \in \lle d_1+1,d\rre$. Denote $f_{\xi,\alpha}(\lambda) = \lambda^{\frac{\alpha}{2}} H_{\alpha}(\lambda^{\frac{1}{2}}\xi)\,\E^{-\frac{\lambda\xi^2}{2}}$; by using the relations $H_n'(x) = nH_{n-1}(x) = x\,H_{n}(x)-H_{n+1}(x)$, one computes its derivative 
\begin{align*}
f_{\xi,\alpha}'(\lambda) &= \frac{\lambda^{\frac{\alpha-2}{2}}}{2}\left((\alpha-x^2)\,H_{\alpha}(x) + \alpha\,x\,H_{\alpha-1}(x)\right) \E^{-\frac{x^2}{2}} \\
&=\frac{\lambda^{\frac{\alpha-2}{2}}}{2}\left((\alpha-1)\,H_{\alpha}(x) -H_2(x)\,H_{\alpha}(x) + \alpha\,H_1(x)\,H_{\alpha-1}(x)\right)\E^{-\frac{x^2}{2}},
\end{align*}
with $x=\lambda^{\frac{1}{2}}\xi$. As a consequence, assuming for instance $\lambda\EX{j} \geq 1$, one obtains
\begin{align*}
&\int_{\R} |\Delta(\xi\EX{j},\lambda\EX{j},\alpha\EX{j})|\DD{\xi\EX{j}} \\
&\leq \frac{1}{2}\int_{1}^{\lambda\EX{j}} \lambda^{\frac{\alpha\EX{j}-3}{2}}\int_{\R} \left(\substack{(\alpha\EX{j}-1)\,|H_{\alpha\EX{j}}(x)| + |H_2(x)\,H_{\alpha\EX{j}}(x)|\\ + \alpha\EX{j}\,|H_1(x)\,H_{\alpha\EX{j}-1}(x)|}\right)\,\E^{-\frac{x^2}{2}}\DD{x}\DD{\lambda} \\
&\leq \frac{1}{2}\,\sqrt{2\pi(\alpha\EX{j})!}\left((\alpha\EX{j}-1)+\sqrt{2}+\sqrt{\alpha\EX{j}}\right)\int_1^{\lambda\EX{j}} \lambda^{\frac{\alpha\EX{j}-3}{2}}\DD{\lambda} \\
&\leq \left(1+2\sqrt{2}\right)\,\sqrt{2\pi(\alpha\EX{j})!}\,\left|(\lambda\EX{j})^{\frac{\alpha\EX{j}-1}{2}}-1\right|.
\end{align*}
The same bound holds when $\lambda\EX{j} \leq 1$. Therefore, if $I_{\boldsymbol\alpha,2}$ is the part of the sum bounding $I_{\boldsymbol\alpha}$ with indices $j \in \lle d_1+1,d\rre$, then
$$I_{\boldsymbol\alpha,2} \leq \sqrt{(2\pi)^d\,\boldsymbol\alpha!} \, (1+2\sqrt{2})\,\left(\rho\EX{1}\rho\EX{2}\cdots \rho\EX{d_0} + \rho\EX{1}\rho\EX{2}\cdots \rho\EX{d} - 2\,\rho\EX{1}\rho\EX{2}\cdots \rho\EX{e_2}\right).$$
\end{enumerate}
\bigskip

\noindent To conclude, suppose that $\{\lambda\EX{1},\ldots,\lambda\EX{d}\} \subset [\lambda_-=1-\rho,\,\lambda_+=1+\rho]$. Then, 
\begin{align*}
\frac{I_{\boldsymbol\alpha,0}}{\prod_{i=1}^d \sqrt{2\pi(\alpha_i)!}} &\leq (\lambda_-)^{-d/2}\,(2-(\lambda_+)^{-d/2} - (\lambda_-)^{d/2}) \leq_{\rho \to 0} d\rho\\
\frac{I_{\boldsymbol\alpha,1}}{\prod_{i=1}^d \sqrt{2\pi(\alpha_i)!}} &\leq \frac{4d}{\sqrt{2\pi}}\, (\lambda_-)^{-d/2}\,\max(1-\lambda_+^{-1},\lambda_-^{-1}-1) \leq_{\rho \to 0} \frac{4}{\sqrt{2\pi}}\,d\rho\\
\frac{I_{\boldsymbol\alpha,2}}{\prod_{i=1}^d \sqrt{2\pi(\alpha_i)!}} &\leq (1+2\sqrt{2})\,(\lambda_-)^{-d/2} ((\lambda_+)^d-(\lambda_-)^d)\leq_{\rho \to 0} (2+4\sqrt{2})\,d\rho.
\end{align*}
Hence, if the spectral radius $\rho((K_1)^{-1/2}K_2(K_1)^{-1/2}-I_d)=\rho$ is small enough, then 
$$I_{\boldsymbol\alpha} \leq \sqrt{(2\pi)^d\,\boldsymbol\alpha!}\,\,8d \,\rho((K_1)^{-1/2}K_2(K_1)^{-1/2}-I_d).$$
The inequality follows by noticing that $\sqrt{\boldsymbol\alpha!} \leq \sqrt{(d+1)!}$ if $|\boldsymbol\alpha| \leq d+1$.
\end{proof}
\medskip

\begin{lemma}\label{lem:headache}
Fix a sequence of random vectors $(\SEC_n)_{n \in \N}$ as in Theorem \ref{thm:superspeed}, and denote $\mu_n$ the law of $\YEC_n=\SEC_n/(\sqrt{N_nD_n})$, and $\nu = \gauss{d}{\mathbf{0}}{K}$. There exist positive constants $C_1(d,K,A,B)$ and $C_2(d,K,A,B)$ such that, if $n$ is large enough, then
$$\left(\eps \geq C_1(d,K,A,B) \sqrt{\frac{D_n}{N_n}}\right)\quad \Rightarrow \quad \left(\Delta_\eps(\widehat{\mu}_n,\widehat{\nu}) \leq C_2(d,K,A,B)\sqrt{\frac{D_n}{N_n}} \right).$$
\end{lemma}

\begin{proof}
Let $K_n = \cov(\SEC_n)/(N_nD_n)$, and $\nu_n = \gauss{d}{\mathbf{0}}{K_n}$. By Lemma \ref{lem:technicaldistancegaussian}, for $n$ large enough, $$\Delta_\eps(\widehat{\nu}_n,\widehat{\nu}) \leq C_{2,\nu}(d,K,A,B)\,\sqrt{\frac{D_n}{N_n}},$$ because $\rho(K_n-K) = O(\sqrt{D_n/N_n})$ by Hypothesis \ref{hyp:MC2prime}. Therefore, it suffices to find an upper bound for $\Delta_\eps(\widehat{\mu}_n,\widehat{\nu}_n)$. We set $M=(K_n)^{-1/2}$, and $\ZIEC=M\XIEC$. We have
\begin{align*}
&\widehat{\mu}_n(\ZIEC)-\widehat{\nu}_n(\ZIEC) \\
&= \E^{-\frac{1}{2}\sum_{i,j=1}^d (K_n)_{ij}\,\zeta\EX{i}\zeta\EX{j}}\,\left(\exp\!\left(\sum_{r \geq 3} \frac{1}{r!} \sum_{i_1,\ldots,i_r=1}^d \frac{(\I^r)\,\kappa(S_{n}\EX{i_1},\ldots,S_{n}\EX{i_r})}{(N_nD_n)^{\frac{r}{2}}}\,\zeta\EX{i_1}\cdots \zeta\EX{i_r}\right)-1\right)\\
&= \E^{-\frac{\|\XIEC\|^2}{2}}\,\left(\exp\left(\sum_{r \geq 3} \frac{1}{r!}\left(\frac{N_n}{D_n}\right)^{1-\frac{r}{2}} \sum_{j_1,\ldots,j_r=1}^d \delta(S_{n}\EX{j_1},\ldots,S_{n}\EX{j_r})\,\xi\EX{j_1}\cdots \xi\EX{j_r}\right)-1\right)\\
&=f(\XIEC),
\end{align*}
where $\delta(S_{n}\EX{j_1},\ldots,S_{n}\EX{j_r}) = \I^r \,\sum_{i_1,\ldots,i_r=1}^d M_{i_1j_1}\cdots M_{i_rj_r}\,\frac{\kappa(S_{n}\EX{i_1},\ldots,S_{n}\EX{i_r})}{N_n\,(D_n)^{r-1}}$. We can then write
\begin{align*}
\Delta_{\eps}(\widehat{\mu}_n,\widehat{\nu}_n) &=\max_{|\boldsymbol\beta| \in \lle 0,d+1\rre} \int_{D^d_{\eps}} \left|\frac{\partial^{|\boldsymbol\beta|} (\widehat{\mu}_n-\widehat{\nu}_n)(\ZIEC)}{\partial \ZIEC^{\boldsymbol\beta}}\right|\DD{\ZIEC} \\
&\leq  \max_{|\boldsymbol\alpha| \in \lle 0,d+1\rre} \,\left(\max_{i\in \lle 1,d\rre} \sum_{j=1}^d |M_{ij}| \right)^{|\boldsymbol\alpha|} \int_{M^{-1}(D^d_{\eps})} \left|\frac{\partial^{|\boldsymbol\alpha|} f(\XIEC)}{\partial \XIEC^{\boldsymbol\alpha}}\right|\DD{\XIEC} \\
&\leq \max_{|\boldsymbol\alpha| \in \lle 0,d+1\rre} \left(d\rho((K_n)^{-1})\right)^{\frac{|\boldsymbol\alpha|}{2}}\int_{M^{-1}(D^d_{\eps})} \left|\frac{\partial^{|\boldsymbol\alpha|} f(\XIEC)}{\partial \XIEC^{\boldsymbol\alpha}}\right|\DD{\XIEC} .
\end{align*}
Until the end of the proof, the multi-index $\boldsymbol\alpha$ is fixed, and we need to compute bounds on the derivatives of $f(\XIEC)$. We have
$$\left|\frac{\partial^{|\boldsymbol\alpha|} f(\XIEC)}{\partial \XIEC^{\boldsymbol\alpha}}\right| \leq \sum_{\boldsymbol\gamma \leq \boldsymbol\alpha} \binom{\boldsymbol\alpha}{\boldsymbol\gamma} \, |H_{\boldsymbol\alpha-\boldsymbol\gamma}(\XIEC)|\,\E^{-\frac{\|\XIEC\|^2}{2}}\,\left|\frac{\partial^{|\boldsymbol\gamma|} g(\XIEC)}{\partial \XIEC^{\boldsymbol\gamma}}\right|,$$
where 
$$g(\XIEC) = \exp\left(\sum_{r \geq 3} \frac{1}{r!}\left(\frac{N_n}{D_n}\right)^{1-\frac{r}{2}}  \sum_{j_1,\ldots,j_r=1}^d \delta(S_{n}\EX{j_1},\ldots,S_{n}\EX{j_r})\,\xi\EX{j_1}\cdots \xi\EX{j_r}\right)-1. $$
Suppose first that $\boldsymbol\gamma=\mathbf{0}$. Then, $|g(\XIEC)| \leq |h(\XIEC)|\,\exp|h(\XIEC)|$, where 
\begin{align*}
|h(\XIEC)| &\leq  \sum_{r=3}^\infty \frac{2^{r-1}r^{r-2}}{r!}\left(\frac{N_n}{D_n}\right)^{1-\frac{r}{2}} \sum_{i_1,\ldots,i_r=1}^d A^r \,\left|\zeta\EX{i_1}\cdots \zeta\EX{i_r}\right| \\
&\leq \frac{N_n}{4\E^3 D_n}\sum_{r=3}^\infty \left(2A\E\|\ZIEC\|_1\sqrt{\frac{D_n}{N_n}}\right)^r \leq 2A^3 \sqrt{\frac{D_n}{N_n}} \frac{(\|\ZIEC\|_1)^3}{1-2A\E\|\ZIEC\|_1\sqrt{\frac{D_n}{N_n}}}.
\end{align*}
Suppose that
$$\eps \geq \eps_1=2 A\E\,(2d)(2d+2)^{3/2}  \sqrt{\frac{D_n}{N_n}}.$$
Then, for any $\XIEC \in M^{-1}(D_\eps^d)$,
\begin{align*}
2A \E \|\ZIEC\|_1  \sqrt{\frac{D_n}{N_n}} &\leq  2d\,A\E \,\|\ZIEC\|_{\infty} \sqrt{\frac{D_n}{N_n}} \leq \frac{A\E (2d)(2d+2)^{3/2}}{\eps} \sqrt{\frac{D_n}{N_n}} \leq \frac{1}{2};\\
|X| &\leq 4A^3 \sqrt{\frac{D_n}{N_n}}\,(\|\ZIEC\|_1)^3 \leq \frac{(C\|\XIEC\|^3)}{2}\,\sqrt{\frac{D_n}{N_n}},
\end{align*}
where $C = 2A\sqrt{d\rho(K^{-1})}$.
More generally, denote $*$ any quantity with modulus smaller than $\E C$. Notice that 
\begin{align*}
|\delta(S_{n}\EX{j_1},\ldots,S_{n}\EX{j_r})| &= \|\delta\|_{\infty} \leq \|\delta\|_2 \leq \rho(M^{\otimes r})\,\left\| \frac{\kappa(S_{n}\EX{i_1},\ldots,S_{n}\EX{i_r})}{N_n\,(D_n)^{r-1}}\right\|_2\\
& \leq \frac{1}{2r^2}\left(2A\sqrt{d\rho(K^{-1})}\,r\right)^{\!r}\!,
\end{align*}
so each coefficient $\delta(S_{n}\EX{j_1},\ldots,S_{n}\EX{j_r})$ is equal to $\frac{1}{2r^2}\,(\frac{*r}{\E})^r$. Therefore, for $\boldsymbol\gamma>\mathbf{0}$,
$\frac{\partial^{|\boldsymbol\gamma|} g(\XIEC)}{\partial \XIEC^{\boldsymbol\gamma}} = \frac{\partial^{|\boldsymbol\gamma|} (\exp h(\XIEC))}{\partial \XIEC^{\boldsymbol\gamma}}$, with
\begin{align*}
h(\XIEC) &= \frac{N_n}{2D_n} \sum_{r=3}^\infty \frac{r^{r-2}}{\E^r\,r!} \sum_{j_1,\ldots,j_r=1}^d (*)^r\,\frac{\xi\EX{j_1} \cdots \xi\EX{j_r}}{(N_n/D_n)^{r/2}}; \\
|h(\XIEC)| &\leq \frac{(C\|\XIEC\|_1)^3}{4}\sqrt{\frac{D_n}{N_n}}\,\frac{1}{1-\E C \|\XIEC\|_1 \sqrt{\frac{D_n}{N_n}}}.
\end{align*}
Here and in the sequel, we assume that $\E C \|\XIEC\|_1 \sqrt{D_n/N_n}<1$; we shall give in a moment a sufficient condition for this inequality. By the multi-dimensional version of Fa\`a di Bruno's formula, 
$$\frac{\partial^c (\exp h(\XIEC))}{\partial \xi\EX{i_1}\cdots \partial \xi\EX{i_c}} = \left(\sum_{\pi \in \mathfrak{P}(c)} \prod_{j=1}^{\ell(\pi)} \frac{\partial^{|\pi_j|}  h(\XIEC)}{\prod_{k\in \pi_j} \partial \xi\EX{i_k}}\right)\exp(h(\XIEC)),$$
where $\mathfrak{P}(c)$ is the set of set partitions $\pi=\pi_1\sqcup \pi_2 \sqcup \cdots \sqcup \pi_{\ell}$ of $\lle 1,c\rre$. In this formula, writing $r^{\downarrow p} = r(r-1)\cdots(r-p+1)$,
\begin{align*}
\frac{\partial^{p}  h(\XIEC)}{\prod_{k\in \pi_j} \partial \xi\EX{i_k}} &= \frac{N_n}{2D_n} \sum_{r=\max(3,p)}^\infty \frac{r^{r-2}}{\E^r\,r!} \sum_{j_1,\ldots,j_{r-p}=1}^d \frac{(*)^r\,r^{\downarrow p}}{(N_n/D_n)^{r/2}}\,\xi\EX{j_1} \cdots \xi\EX{j_{r-p}};\\
\left|\frac{\partial^{p}  h(\XIEC)}{\prod_{k\in \pi_j} \partial \xi\EX{i_k}}\right| &\leq \frac{N_n}{4\E^3D_n} \sum_{r=\max(3,p)}^\infty  \frac{(\E C)^r\,(\|\XIEC\|_1)^{r-p}\,r^{\downarrow p}}{(N_n/D_n)^{r/2}}\\
&\leq \frac{(\E C)^{\max(3,p)}\,(\max(3,p))!\,(\|\XIEC\|_1)^{\max(3-p,0)} }{4\E^3\,(N_n/D_n)^{\frac{\max(3,p)}{2}-1} \,\left(1-\E C \|\XIEC\|_1 \sqrt{\frac{D_n}{N_n}} \right)^{p+1}},
\end{align*}
where for the last inequality one has to treat separately the cases $p\geq 3$ and $p=1 \text{ or }2$. \medskip

Suppose that $\ZIEC \in C_{(\mathbf{0},(2d+2)^{3/2}/\eps)}^{d}=D_\eps^d$. Then, 
\begin{align*}
&\|\XIEC\|_1 \leq \sqrt{d}\,\|\XIEC\|_2 \leq \sqrt{d\rho(K)}\, \|\ZIEC\|_2 \leq d\sqrt{\rho(K)}\,\|\ZIEC\|_\infty \leq \frac{(2d)(2d+2)^{3/2}\sqrt{\rho(K)}}{2\eps};\\
&\E C \|\XIEC\|_1 \sqrt{D_n/N_n} \leq \frac{\eps_1}{\eps} \sqrt{d\,\rho(K)\rho(K^{-1})}= \frac{\eps_1}{\eps}\, \sqrt{d\,\tau(K)} . 
\end{align*}
Set 
$$\eps_2 = 2\sqrt{d\,\tau(K)}\,\eps_1 = 16 A\E\,\sqrt{2(d(d+1))^3 \,\tau(K)}\,  \sqrt{\frac{D_n}{N_n}}. $$
Then, if $\eps\geq \eps_2$ (this condition is stronger than $\eps \geq \eps_1$) and $\ZIEC \in D_\eps^d$, we have for any $p \geq 1$:
\begin{align*}
\left|\frac{\partial^{p}  h(\XIEC)}{\prod_{k\in \pi_j} \partial \xi\EX{i_k}}\right| &\leq \frac{(\E C)^{\max(3,p)}\,2^{p+1}\,(\max(3,p))!\,(\|\XIEC\|_1)^{\max(3-p,0)} }{4\E^3\,(N_n/D_n)^{\frac{\max(3,p)}{2}-1} }.
\end{align*}
As a consequence, 
 \begin{align*}
 J_{\boldsymbol\alpha} &= \int_{M^{-1}(D^d_{\eps})} \left|\frac{\partial^{|\boldsymbol\alpha|} f(\XIEC)}{\partial \XIEC^{\boldsymbol\alpha}}\right|\DD{\XIEC} \leq \sum_{\boldsymbol\gamma \leq \boldsymbol\alpha} \sum_{\pi \in \mathfrak{P}(|\boldsymbol\gamma|)} \int_{M^{-1}(D^d_{\eps})} \E^{-\frac{\|\XIEC\|^2}{2} + \frac{(C\|\XIEC\|)^3}{2}\sqrt{\frac{D_n}{N_n}} }\, \frac{P_{\boldsymbol\gamma,\boldsymbol\alpha,\pi}(\XIEC)}{(N_n/D_n)^{e(\pi)}}\DD{\XIEC},
 \end{align*}
where each $P_{\boldsymbol\gamma,\boldsymbol\alpha,\pi}(\XIEC)$ is a quantity that is bounded by a polynomial in $\|\XIEC\|_1$, and with
$$e(\pi) = \sum_{i=1}^{\ell(\pi)} \left(\frac{\max(3,|\pi_i|)}{2}-1\right).$$
Note that if $\|\XIEC\|\sqrt{\frac{D_n}{N_n}}$ is small enough, then 
$$-\frac{\|\XIEC\|^2}{2} + \frac{(C\|\XIEC\|)^3}{2}\sqrt{\frac{D_n}{N_n}} \leq -\frac{\|\XIEC\|^2}{4}.$$
Therefore, one can find a positive constant $C_1(d,K,A,B) \geq \eps_2\sqrt{N_n/D_n}$ such that, if $\eps \geq C_1(d,K,A,B)\,\sqrt{D_n/N_n}$, then 
$$J_{\boldsymbol\alpha} \leq \sum_{\boldsymbol\gamma \leq \boldsymbol\alpha} \sum_{\pi \in \mathfrak{P}(|\boldsymbol\gamma|)} \frac{1}{(N_n/D_n)^{e(\pi)}} \int_{\R^d} \E^{-\frac{\|\XIEC\|^2}{4}}\, P_{\boldsymbol\gamma,\boldsymbol\alpha,\pi}(\XIEC)\DD{\XIEC}.$$
All the integrals appearing in the right-hand side of this inequality are convergent, and the main contribution to the sum is provided by set partitions such that $e(\pi)=\frac{1}{2}$ is minimal. This only happens when $\ell(\pi)=1$ and $|\boldsymbol\gamma| \leq 3$. Therefore, keeping only these terms, we obtain
\begin{align*}
J_{\boldsymbol\alpha}&\leq \sum_{\substack{\boldsymbol\gamma \leq \boldsymbol\alpha \\ |\boldsymbol\gamma| \leq 3}} \binom{\boldsymbol\alpha}{\boldsymbol\gamma} \int_{\R^d} \E^{-\frac{\|\XIEC\|^2}{2}}\,|H_{\boldsymbol\alpha-\boldsymbol\gamma}(\XIEC)|\,\frac{\partial^{|\boldsymbol\gamma|}h(\XIEC)}{\partial \XIEC^{\boldsymbol\gamma}}\, g(\XIEC)\DD{\XIEC} + O\!\left(\frac{D_n}{N_n}\right),\\
&\leq \sum_{\substack{\boldsymbol\gamma \leq \boldsymbol\alpha \\ |\boldsymbol\gamma| \leq 3}} \binom{\boldsymbol\alpha}{\boldsymbol\gamma} \int_{\R^d} \E^{-\frac{\|\XIEC\|^2}{4}}\,|H_{\boldsymbol\alpha-\boldsymbol\gamma}(\XIEC)|\,\frac{\partial^{|\boldsymbol\gamma|}h(\XIEC)}{\partial \XIEC^{\boldsymbol\gamma}}\DD{\XIEC} + O\!\left(\frac{D_n}{N_n}\right)
\end{align*}
where the constant in the $O$ only depends on $d,K,A,B$. Using the previous estimates on the derivatives of $h(\XIEC)$, we conclude that
$$J_{\boldsymbol\alpha} = O\!\left(\sqrt{\frac{D_n}{N_n}} \sum_{\substack{\boldsymbol\gamma \leq \boldsymbol\alpha \\ |\boldsymbol\gamma| \leq 3}} \binom{\boldsymbol\alpha}{\boldsymbol\gamma} \int_{\R^d} \E^{-\frac{\|\XIEC\|^2}{4}}\,|H_{\boldsymbol\alpha-\boldsymbol\gamma}(\XIEC)|\,(\|\XIEC\|_1)^{3-|\boldsymbol\gamma|}\DD{\XIEC}\right) + O\!\left(\frac{D_n}{N_n}\right).$$
Therefore, there exists a positive constant $C_{2,\mu}(d,K,A,B)$ such that, if $\eps \geq C_1(d,K,A,B)$, then
$$\Delta_{\eps}(\widehat{\mu}_n,\widehat{\nu}_n)\leq C_{2,\mu}(d,K,A,B)\,\sqrt{\frac{D_n}{N_n}}.$$
This ends the proof of the lemma, taking $C_{2} = C_{2,\mu} + C_{2,\nu}$.
\end{proof}
\medskip

\begin{proof}[Proof of Theorem \ref{thm:superspeed}]
We apply Corollary \ref{cor:fouriertoconvdistance} to the inequality of Lemma \ref{lem:headache}, with $\eps = C_1(d,K,A,B)\sqrt{D_n/N_n}$ and the regularity constant $R= 2\sqrt{(d+1)\rho(K^{-1})}$.
\end{proof}
\medskip

\subsection{Sparse dependency graphs}\label{subsec:dependencygraph}
A large set of applications of Theorems \ref{thm:multicumulant} and \ref{thm:superspeed} is provided by the theory of \emph{dependency graphs}, which has already been used successfully in dimension $d=1$ in \cite{FMN16,FMN19}. 

\begin{definition}[Dependency graph]
Let $(\mathbf{A}_v)_{v \in V_{\R^d}}$ be a family of random vectors in $\R^d$. A simple undirected graph $G_{\R^d}=(V_{\R^d},E_{\R^d})$ is called a dependency graph for this family if, for any disjoint subsets $V_1$ and $V_2$ of $V_{\R^d}$, if there is no edge $e=(v_1,v_2)$ in $E_{\R^d}$ that connects $v_1 \in V_1$ and $v_2 \in V_2$, then $(\mathbf{A}_v)_{v \in V_1}$ and $(\mathbf{A}_v)_{v \in V_2}$ are independent families of random vectors of $\R^d$.
\end{definition}

\noindent The parameters of a dependency graph $G_{\R^d}$ are:
\begin{enumerate}
	\item its \emph{size} $N = \card\,V_{\R^d}$.
	\item its \emph{maximal degree} $D$: every vertex $v \in V_{\R^d}$ has at most $D-1$ neighbors $w$ such that $\{v,w\} \in E_{\R^d}$. 
\end{enumerate}
In the sequel, we always put an index $\R^d$ on the dependency graph of a family of $d$-dimensional random vectors. The reason for this is the following. Fix a family $(\mathbf{A}_v)_{v \in V_{\R^d}}$ of random vectors in $\R^d$, a dependency graph $G_{\R^d}=(V_{\R^d},E_{\R^d})$ for this family, and denote $W = V_{\R^d} \times \lle 1,d\rre$. To any pair $w=(v,i) \in W$, we associate the real random variable 
$$A_w = A_v\EX{i}.$$
We then endow the vertices of $W$ with the following structure of graph $G=(W,E)$:
$$ (v_1,i_1) \sim_G (v_2,i_2) \quad\iff \quad (v_1 = v_2 \text{ and }i_1 \neq i_2) \,\text{ or }\, (v_1 \sim_{G_{\R^d}} v_2).$$
It is then easily seen that $(A_w)_{w\in W}$ is a family of real random variables with dependency graph $G=(W,E)$. Moreover, if $N$ and $D$ are the parameters of $G_{\R^d}$, then $dN$ and $dD$ are the parameters of $G$.
\medskip

\begin{theorem}[Upper bound on cumulants]\label{thm:upperboundcumulant}
Let 
$$\SEC=\sum_{v \in V_{\R^d}} \mathbf{A}_v$$
be a sum of random vectors in $\R^d$, such that $G_{\R^d} = (V_{\R^d},E_{\R^d})$ is a dependency graph for the family $(\mathbf{A}_v)_{v \in V_{\R^d}}$. We also suppose that $\|\mathbf{A}_v\|_{\infty} \leq A$ for any $v \in V_{\R^d}$. For any choice of indices $i_1,\ldots,i_r \in \lle 1,d\rre$, 
$$|\kappa(S\EX{i_1},\ldots,S\EX{i_r})| \leq N\,(2D)^{r-1}\,A^r\,r^{r-2},$$
where $N$ and $D$ are the parameters of the graph $G_{\R^d}$.
\end{theorem}

\begin{proof}
Let $G=(W,E)$ be the dependency graph on one-dimensional random variables $A_v\EX{i}$ that is obtained from $G_{\R^d}$ by the aforementioned construction. By \cite[Equation 9.9]{FMN16}, for any choice of vertices $v_{1},\ldots,v_r$,
$$\left|\kappa(A_{v_1}\EX{i_1},\ldots,A_{v_r}\EX{i_r})\right| \leq 2^{r-1}\,A^r\,\mathrm{ST}_H ,$$
where $H=G[(v_1,i_1),\ldots,(v_r,i_r)]$ is the (multi)graph induced by $G$ on the set of vertices $\{(v_j,i_j)\,|\,j \in \lle 1,r\rre\}$ (see \cite[Section 9.3.2]{FMN16}), and $\mathrm{ST}_H$ is the number of spanning trees of $H$ (hence, $0$ if $H$ is not connected). Therefore,
if we extend by multilinearity the joint cumulant of the coordinates of the sum $S$, then we obtain
\begin{align*}
\left|\kappa(S\EX{i_1},\ldots,S\EX{i_r})\right| &\leq \sum_{v_1,\ldots,v_r \in V_{\R^d}} \left|\kappa(A_{v_1}\EX{i_1},\ldots,A_{v_r}\EX{i_r})\right| \\
&\leq 2^{r-1} A^r\sum_{v_1,\ldots,v_r \in V_{\R^d}} \mathrm{ST}_{G[(v_1,i_1),\ldots,(v_r,i_r)]}.
\end{align*}
We now use an argument similar to the one of \cite[Lemma 9.3.5]{FMN16}. Consider a pair $(T,(v_1,\ldots,v_r))$ such that $v_1,\ldots,v_r \in V_{\R^d}$, and $T$ is a spanning tree on $\lle 1,r\rre$ included in the induced subgraph $G[(v_1,i_1),\ldots,(v_r,i_r)]$. By Cayley's formula for the number of spanning trees on $r$ vertices, there are $r^{r-2}$ possible choices for $T$. Then, to choose $(v_1,\ldots,v_r)$ such that $T \subset G[(v_1,i_1),\ldots,(v_r,i_r)]$, we proceed as follows. There are $N$ possible choices for the vector $v_1$. Then, for each $j$ such that $1 \sim j$ in $T$, there are at most $D$ possible choices for $v_j$, such that either $(v_1,i_1)=(v_j,i_j)$, or $(v_1,i_1) \sim_G (v_j,i_j)$. Indeed, one can only choose $v_j$ among the vertices $v_j$ with $v_1 \sim_{G_{\R^d}} v_j$, plus $v_1$. We pursue this reasoning with the neighbors of the neighbors $j$ of $1$ in $T$, performing a breadth-first search of $T$. We conclude that $T$ being fixed, the number of compatible choices for $v_1,\ldots,v_r$ is smaller than $N\,D^{r-1}$. This ends the proof of the theorem.
\end{proof}
\medskip

By combining Theorems \ref{thm:multicumulant}, \ref{thm:superspeed} and \ref{thm:upperboundcumulant}, we obtain the following general result on sums of random vectors with sparse dependency graphs:

\begin{theorem}[Sums of random vectors with sparse dependency graphs]\label{thm:sumsparsegraph}
Consider a sequence of sums $(\SEC_n = \sum_{i=1}^{N_n} \mathbf{A}_{i,n})_{n \in \N}$ of centered random vectors in $\R^d$. We endow each family $(\mathbf{A}_{i,n})_{i \in \lle 1,N_n\rre}$ with a dependency graph of parameters $N_n$ and $D_n$, and we make the following assumptions:
\begin{enumerate}
 	\item All the random vectors $\mathbf{A}_{i,n}$ are bounded by $A$ in norm $\|\cdot\|_\infty$.
 	\item The sequences of parameters $(N_n)_{n \in \N}$ and $(D_n)_{n \in \N}$ satisfy
	$$\lim_{n \to \infty} N_n = + \infty \qquad;\qquad \lim_{n \to \infty} \frac{D_n}{N_n} = 0.$$
	\item There exists a matrix $K \in \mathrm{S}_+(d,\R)$, and a family of real numbers $(L_{i,j,k})_{1\leq i,j,k\leq d}$, such that
	\begin{align*}
	\frac{\cov(\SEC_n)}{N_nD_n} & = K+J_n \quad\text{with }\rho(J_n) \leq B\,\sqrt{\frac{D_n}{N_n}}.
	\end{align*}
\end{enumerate} 
The sums $\SEC_n$ satisfy:
\begin{itemize}
		\item Central limit theorem: if $\YEC_n = \frac{\SEC_n}{\sqrt{N_nD_n}}$, then $\YEC_n \rightharpoonup \gauss{d}{\mathbf{0}}{K}$.
		\item Speed of convergence: there exists a constant $C(d,K,A,B)$ such that
		$$\dconv(\YEC_n\,,\,\gauss{d}{\mathbf{0}}{K}) \leq C(d,K,A,B)\,\sqrt{\frac{D_n}{N_n}}.$$
\end{itemize}	
If one adds the assumptions \ref{hyp:MC4} and \ref{hyp:MC5} on $\SEC_n$, then one also gets large deviation estimates as in Theorem \ref{thm:multicumulant}.
\end{theorem}
\medskip

We conclude this section by examining three examples:
\begin{example}[Random walks with dependent steps, $d = 2$]
One can construct an example in dimension $d=2$ with almost arbitrary parameters $N_n \to +\infty$ and $D_n=o(N_n)$. If $\lambda$ is a positive parameter, denote $\mathcal{L}(\lambda)$ the law of the Brownian motion on the unit circle $\mathbb{S}^1=\{z \in \C\,|\,|z|=1\}$, taken at time $t=\lambda$. This law is given by the density 
$$\rho_\lambda(\theta) = \frac{1}{\sqrt{2\pi\lambda}} \sum_{k \in \Z} \E^{-\frac{(\theta-2k\pi)^2}{2\lambda}}$$
for $\theta \in (-\pi,\pi)$. Thus, if $\mathbf{X} = \E^{\I\theta}$ follows the law $\mathcal{L}(\lambda)$, then
$$\proba[\theta \in (a,b)] = \frac{1}{\sqrt{2\pi\lambda}} \sum_{k \in \Z}\int_{a}^b \E^{\frac{(\theta-2k\pi)^2}{2\lambda}}\DD{\theta} $$
for any $a,b$ such that $-\pi < a < b <\pi$. The density of the angle $\theta$ can be rewritten in Fourier series:
$$\rho_\lambda(\theta)\DD{\theta} = \left(\sum_{l\in \Z}\E^{-\frac{\lambda l^2}{2}}\,\E^{\I l\theta}\right)\,\frac{\!\DD{\theta}}{2\pi}.$$
The family of laws $(\mathcal{L}(\lambda))_{\lambda \in \R_+}$ is a semigroup of probability measures on the multiplicative semigroup $\mathbb{S}^1$. Thus, if $\mathbf{U}_1,\ldots,\mathbf{U}_D$ follow the same law $\mathcal{L}(\frac{\lambda}{D})$ and are independent, then the product $\mathbf{U}_1\mathbf{U}_2\cdots \mathbf{U}_D$ follows the law $\mathcal{L}(\lambda)$. \bigskip

Consider a family of independent random variables $(\mathbf{U}_i)_{i \in \lle 1,N_n\rre}$, all following the same law $\mathcal{L}(\frac{\lambda}{D_n})$. For $i \in \lle 1,N_n\rre$, we denote $\XEC_i=\mathbf{U}_{i}\mathbf{U}_{i+1}\cdots \mathbf{U}_{i+D_n-1}$, where the $\mathbf{U}_i$'s are labelled cyclically, that is to say that $\mathbf{U}_{N_n+k}=\mathbf{U}_{k}$. Two variables $\XEC_i$ and $\XEC_j$ are independent unless they share some common factor $\mathbf{U}_k$, which happens if and only if the distance between $i$ and $j$ in $\Z/N_n\Z$ is smaller than $D_n$. Thus, $(\XEC_i)_{i \in \lle 1,N_n\rre}$ is a family of dependent random variables in $\R^2 \supset \mathbb{S}^1$, with a dependency graph of parameters $N_n$ and $2D_n-1$. By the previous remark, each $\XEC_i$ follows the law $\mathcal{L}(\lambda)$; in particular,
$$\esper[\XEC_i] = \int_{-\pi}^{\pi} \E^{\I\theta}\,\rho_\lambda(\theta)\DD{\theta} = \sum_{l \in \Z} \E^{-\frac{\lambda l^2}{2}} \int_{-\pi}^{\pi} \E^{\I(l+1)\theta}\,\frac{\!\DD{\theta}}{2\pi} = \E^{-\frac{\lambda}{2}}.$$
On the other hand, if $\XEC_i$ and $\XEC_j$ are two variables with $d(i,j)=D<D_n$ in $\Z/N_n\Z$, then one can compute the covariance matrix of these two vectors of $\R^2$:
\begin{align*}
\esper[\Re(\XEC_i)\,\Re(\XEC_j)]&=\frac{\esper[\XEC_i\XEC_j+\overline{\XEC_i}\XEC_j + \XEC_i\overline{\XEC_j}+\overline{\XEC_i\XEC_j}]}{4} = \frac{\E^{-\frac{D}{D_n}\lambda}+\E^{-(2-\frac{D}{D_n})\lambda}}{2} ; \\
\esper[\Im(\XEC_i)\,\Im(\XEC_j)]&=\frac{\esper[-\XEC_i\XEC_j+\overline{\XEC_i}\XEC_j + \XEC_i\overline{\XEC_j}-\overline{\XEC_i\XEC_j}]}{4}=\frac{\E^{-\frac{D}{D_n}\lambda}-\E^{-(2-\frac{D}{D_n})\lambda}}{2};\\
\esper[\Re(\XEC_i)\,\Im(\XEC_j)]&=\esper[\Im(\XEC_i)\,\Re(\XEC_j)] = 0. 
\end{align*}
So,
\begin{align*}
\cov(\Re(\XEC_i),\Re(\XEC_j)) &= \frac{\E^{-\frac{D}{D_n}\lambda}+\E^{-(2-\frac{D}{D_n})\lambda}-2\E^{-\lambda}}{2} ;\\
\cov(\Im(\XEC_i),\Im(\XEC_j)) &= \frac{\E^{-\frac{D}{D_n}\lambda}-\E^{-(2-\frac{D}{D_n})\lambda}}{2}; \\
\cov(\Re(\XEC_i),\Im(\XEC_j)) &= \cov(\Im(\XEC_i),\Re(\XEC_j)) = 0.
\end{align*}
As a consequence, if $\SEC_n=\sum_{i=1}^{N_n} \XEC_i$, then
\begin{align*}
\frac{\cov(\SEC_n)}{(2D_n-1)\,N_n} &= \frac{1}{2(2D_n - 1)} \sum_{D=-(D_n-1)}^{D_n-1} \left(\begin{smallmatrix} \E^{-\frac{|D|}{D_n}\lambda}+\E^{-(2-\frac{|D|}{D_n})\lambda}-2\E^{-\lambda} & 0 \\ 0 & \E^{-\frac{|D|}{D_n}\lambda}-\E^{-(2-\frac{|D|}{D_n})\lambda}\end{smallmatrix}\right)\\
&=\frac{1}{2D_n-1} \left(\begin{smallmatrix} \frac{1-\E^{-\lambda}}{1-\E^{-\lambda/D_n}} + \E^{-2\lambda}\,\frac{\E^{\lambda}-1}{\E^{\lambda/D_n}-1} - 1 -(2D_n-1)\,\E^{-\lambda} & 0 \\
0& \frac{1-\E^{-\lambda}}{1-\E^{-\lambda/D_n}} - \E^{-2\lambda}\,\frac{\E^{\lambda}-1}{\E^{\lambda/D_n}-1}\end{smallmatrix}\right)\\
&= K+O\!\left(\frac{1}{D_n}\right)\qquad\text{with }K=\frac{1}{2\lambda}\begin{pmatrix} 1-2\lambda\E^{-\lambda}-\E^{-2\lambda} & 0 \\ 0 & (1-\E^{-\lambda})^2\end{pmatrix}.
\end{align*}
Therefore, one can apply Theorem \ref{thm:sumsparsegraph} if 
$ D_n = o(N_n)$ and $(N_n)^{1/3}= O(D_n)$.
In this situation, $\YEC_n=\frac{\SEC_n-N_n\,\E^{-\frac{\lambda}{2}}}{\sqrt{(2D_n-1)\,N_n}}$ converges in law towards $\mathcal{N}_{\R^2}(\mathbf{0},K)$, with
$$\dconv(\YEC_n,\,\gauss{2}{\mathbf{0}}{K}) = O\!\left(\sqrt{\frac{D_n}{N_n}}\right).$$
We have drawn on Figure \ref{fig:randomwalks} a random walk with steps $\XEC_i$, with $\lambda=1$, $N_n=1000$ and $D_n=40$ or $D_n=1$. The deviation of the whole sum $\SEC_n$ from its mean $1000\,\E^{-1/2} \simeq 600$ follows approximatively a (non-standard) Gaussian law of size $O(\sqrt{D_n\,N_n})=O(200)$.
\begin{center}
\begin{figure}[ht]
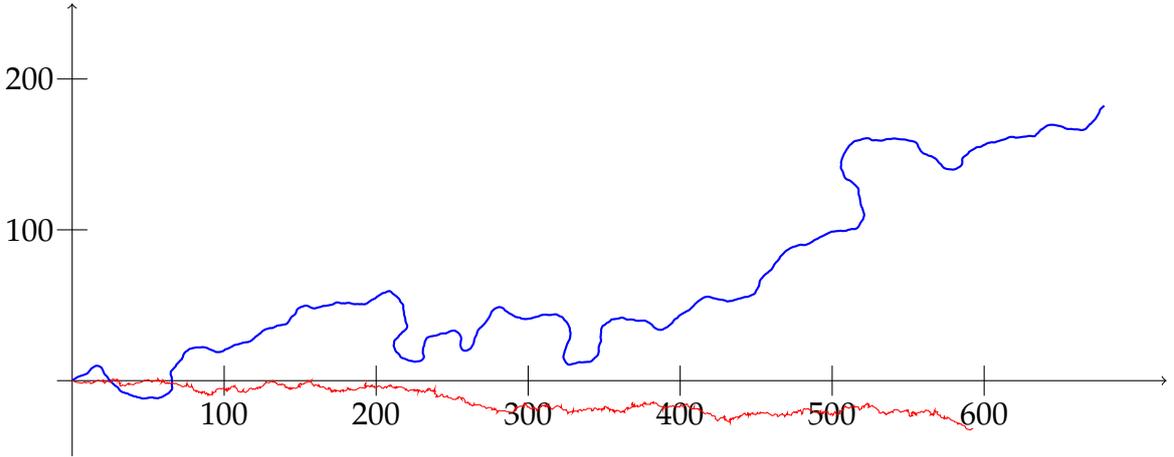


\caption{Random walks associated to a sum of dependent random vectors in $\R^2$ (thick blue line, $N_n=1000$ and $D_n=40$), and to a sum of independent random vectors (thin red line, $N_n=1000$ and $D_n=1$).\label{fig:randomwalks}}
\end{figure}
\end{center}
\end{example}

\begin{example}[Random walks with dependent steps, $d \geq 2$]
Let us describe a generalization of the previous example with an arbitrary dimension $d \geq 2$. Denote $p_{\SO(d),t}(g)$ the density with respect to the Haar measure of a Brownian motion traced on the special orthogonal group $G=\SO(d)$, taken at time $t>0$. It can be shown that, if $\widehat{G}$ is a set of representatives of the isomorphism classes of finite-dimensional irreducible representations of $G$, and if $\chi^\lambda$ and $d^\lambda$ denote respectively the normalised character and the dimension of an irreducible representation $\lambda \in \widehat{G}$, then there exists positive rational coefficients $c_\lambda$ such that
$$p_{\SO(d),t}(g) = \sum_{\lambda \in \widehat{G}} \E^{-\frac{c_\lambda t}{2}}\,(d^\lambda)^2\,\chi^\lambda(g).$$
We refer to \cite[Section 2]{Mel14} for precisions on Brownian motions on compact Lie groups and their symmetric quotients. The coefficients $c_\lambda$ can be computed by associating to each class of isomorphism $\lambda \in \widehat{G}$ the highest weight of the underlying representation, which sits in a part of the lattice of characters $\widehat{T}$ of a maximal connected torus $T$ in $\SO(d)$. Then, $c_{\lambda}=\scal{\lambda}{\lambda+2\rho}$, where $2\rho$ is the sum of all positive roots of $G$, and $\scal{\cdot}{\cdot}$ is an appropriate scalar product on $\widehat{T} \otimes_{\Z} \R$.\bigskip

We denote $\mathcal{L}_{\SO(d)}(t)$ the law on $\SO(d)$ with density $p_{\SO(d),t}(\cdot)$, and we consider a family $(u_i)_{i \in \lle 1,N_n\rre}$ of independent random variables with values in $\SO(d)$, all following the same law $\mathcal{L}_{\SO(d)}(\frac{\lambda}{D_n})$. We then set $g_i=u_iu_{i+1}\cdots u_{i+D_n-1}$, where as in the previous example, the $u_i$'s are labelled cyclically. Each random variable $g_i$ follows the same law $\mathcal{L}_{\SO(d)}(\lambda)$, and $g_i$ and $g_j$ are independent if and only if $i$ and $j$ are at distance larger than $D_n$ in $\Z/N_n\Z$. We finally set $\XEC_i=g_i(\mathbf{e}_d)$, where $\mathbf{e}_d=(0,\ldots,0,1) \in \R^d$. Each $\XEC_i$ follows the same law $\mathcal{L}_{\mathbb{S}^{d-1}}(\lambda)$, which is the law of a Brownian motion on the sphere $\mathbb{S}^{d-1}$ starting from the vector $\mathbf{e}_d$ and taken at time $t=\lambda$. The density of this law with respect to the Haar measure on the sphere $\mathbb{S}^{d-1}$ is given by the following formula:
$$p_{\mathbb{S}^{d-1},t}(\mathbf{x}) = \sum_{\lambda \in \widehat{G}^{K}} \E^{-\frac{c_{\lambda}t}{2}}\,d^\lambda\,\phi^{\lambda}(\mathbf{x}).$$
Here, $\widehat{G}^{K}$ is the subset of $\widehat{G}$ that consists in spherical representations, that is to say, irreducible representations of $G=\SO(d)$ that admits a (normalised) $K=\SO(d-1)$-invariant vector $\mathbf{e}^\lambda$. On the other hand, $\phi^\lambda(\xec) = \scal{\rho^{\lambda}(g)\,(\mathbf{e}^\lambda)}{\mathbf{e}^\lambda}$, where $g \in G$ is any element such that $g(\mathbf{e}_d)=\xec$; $\rho^\lambda$ is the defining morphism of the spherical representation $\lambda$; and $\scal{\cdot}{\cdot}$ is a $G$-invariant scalar product on the representation space, with $\scal{\mathbf{e}^\lambda}{\mathbf{e}^\lambda}=1$. Standard arguments of harmonic analysis on the real spheres (see \cite{AH10,Vol09}) ensure that there is a labeling of the spherical representations in $\widehat{G}^K$ by the set of natural integers $\N$. Indeed, in terms of highest weights, $\widehat{G}^K=\N \omega_0$, where $\omega_0$ is the highest weight associated to the geometric representation of $\SO(d)$ on $\C^d$. Then, $c_{k\omega_0} = k^2+(d-2)k$, $d^{k\omega_0}=\frac{2k+d-2}{k+d-2}\binom{k+d-2}{d-2}$ by Weyl's formula for dimensions of irreducible representations, and 
$$\phi^{k\omega_0}(\xec)=P^{d-1,k}(x\EX{d}),$$
where $(P^{d-1,k}(t))_{k \in \N}$ is the set of Legendre polynomials associated to the sphere $\mathbb{S}^{d-1}$: they are the orthogonal polynomials with respect to the weight 
$$ \frac{\Gamma(\frac{d}{2})}{\Gamma(\frac{1}{2})\,\Gamma(\frac{d-1}{2})}\,(1-t^2)^{\frac{d-3}{2}}\,1_{t\in [-1,1]}\DD{t} ,$$
normalised by the condition $P^{d-1,k}(1)=1$. So,
$$p_{\mathbb{S}^{d-1},t}(\xec) = \sum_{k=0}^\infty \E^{-\frac{(k^2+(d-2)k)t}{2}}\,\frac{2k+d-2}{k+d-2}\binom{k+d-2}{d-2}\,P^{d-1,k}(x\EX{d}).$$
The expectation of a variable $\XEC_i$ following the law $\mathcal{L}_{\mathbb{S}^{d-1}}(\lambda)$ is
$$\sum_{k=0}^\infty \E^{-\frac{(k^2+(d-2)k)t}{2}} \,d^{k\omega_0}\,\int_{\mathbb{S}^{d-1}} \xec\,\phi^{k\omega_0}(\xec)\,\mathrm{Haar}(\!\DD{\xec}).$$
For every coordinate $i \in \lle 1,d\rre$, $x\EX{i}$ is a coefficient of the geometric representation of $\SO(d)$. Since coefficients of distinct irreducible representations of $\SO(d)$ are orthogonal by Schur's lemma, the scalar product of $x\EX{i}$ with the spherical function $\phi^{k\omega_0}$ is thus equal to $0$, unless $k=1$. Then, $\phi^{\omega_0}(\xec)=x\EX{d}$, and
$$d^{\omega_0}\,\int_{\mathbb{S}^{d-1}} \xec\,x\EX{d}\,\mathrm{Haar}(\!\DD{\xec}) = \mathbf{e}_d.$$
Hence, $\esper[\XEC_i] = \E^{-\frac{(d-1)\lambda}{2}}\,\mathbf{e}_d$. Let us now compute the covariance matrix between two variables $\XEC_i$ and $\XEC_j$. We fix two coordinates $k$ and $l$ in $\lle 1,d\rre$, and consider $\esper[X_i\EX{k}X_j\EX{l}]$. Since $\XEC_i$ and $\XEC_j$ are independent if $i$ and $j$ are at distance larger than $D_n$ in $\Z/N_n\Z$, we can exclude this case and assume $d(i,j) =D < D_n$. Then, with $p_{\SO(d),t}=p_t$,
\begin{align*}
&\esper\!\left[X_i\EX{k}X_j\EX{l}\right] \\
&= \int_{G^3} p_{\frac{D}{D_n}\lambda}(f)\, p_{\frac{D_n-D}{D_n}\lambda}(g)\,p_{\frac{D}{D_n}\lambda}(h) \scal{(fg)(\mathbf{e}_d)}{\mathbf{e}_k}\,\scal{(gh)(\mathbf{e}_d)}{\mathbf{e}_l}\DD{f}\DD{g}\DD{h} \\
 &= \int_{G^3} p_{\frac{D}{D_n}\lambda}(f)\, p_{\frac{D_n-D}{D_n}\lambda}(g)\,p_{\frac{D}{D_n}\lambda}(h) \scal{g(\mathbf{e}_d)}{f^{-1}(\mathbf{e}_k)}\,\scal{h(\mathbf{e}_d)}{g^{-1}(\mathbf{e}_l)}\DD{f}\DD{g}\DD{h}\\
 &=\E^{-\frac{(d-1)D}{D_n}\lambda}\int_{G} p_{\frac{D_n-D}{D_n}\lambda}(g) \scal{g(\mathbf{e}_d)}{\mathbf{e}_k}\,\scal{g(\mathbf{e}_d)}{\mathbf{e}_l}\DD{g}
 \end{align*}
 by using on the second line the invariance of the scalar product under $G=\SO(d)$, and on the third line the computation of the expectation of the law $\mathcal{L}_{\mathbb{S}^{d-1}}(t)$. The remaining integral is the expectation of the product of matrix coefficients $g_{dk}g_{dl}$ under the law $\mathcal{L}_{\SO(d)}(\frac{D_n-D}{D_n}\lambda)$ of the Brownian motion on the group. This computation is performed in \cite[Section 4]{Mel14} (beware that the normalisation of time in this article differs by a factor $d$). Hence, at time $t$, one has
 \begin{align*}
 \esper[(g_{dd})^2] &= \frac{1}{d} + \left(1-\frac{1}{d}\right)\,\E^{-dt};\\
 \esper[(g_{dk})^2] &= \frac{1}{d}\left(1-\E^{-dt}\right) \quad\forall k \neq d;\\
 \esper[g_{dk}g_{dl}] &=0\quad \forall k \neq l. 
 \end{align*}
 Therefore,
 \begin{align*}
 \cov(X_i\EX{d},X_j\EX{d}) &= \frac{\E^{-\frac{(d-1)D}{D_n}\lambda} + (d-1)\E^{-(d-\frac{D}{D_n})\lambda} - d\E^{-(d-1)\lambda}}{d}; \\
 \cov(X_i\EX{k},X_j\EX{k}) &= \frac{\E^{-\frac{(d-1)D}{D_n}\lambda}-\E^{-(d-\frac{D}{D_n})\lambda}}{d}\quad\forall k \neq d;\\
 \cov(X_i\EX{k},X_j\EX{l}) &= 0 \quad\forall k \neq l.
 \end{align*}
 These formul{\ae} generalize indeed what has been obtained when $d=2$. Then, setting $\SEC_n=\sum_{i=1}^{N_n} \XEC_i$, we obtain:
 $$\frac{\cov(\SEC_n)}{(2D_n-1)\,N_n} = K + O\!\left(\frac{1}{D_n}\right),$$
 where $K$ is the diagonal matrix with 
 \begin{align*}
 K_{dd} &= \frac{1+d(d-2-(d-1)\lambda)\,\E^{-(d-1)\lambda}-(d-1)^2\,\E^{-d\lambda}}{d(d-1)\lambda};\\
 K_{kk} &=\frac{1-d\E^{-(d-1)\lambda}+(d-1)\E^{-d\lambda}}{d(d-1)\lambda} \quad \forall k \neq d.
 \end{align*}
 One can apply Theorem \ref{thm:sumsparsegraph} to the sum $\SEC_n=\sum_{i=1}^{N_n} \XEC_i$ if $D_n = o(N_n)$ and $(N_n)^{1/3}=O(D_n)$; then, $$\YEC_n=\frac{\SEC_n-N_n\,\E^{-\frac{(d-1)\lambda}{2}}\mathbf{e}_d}{\sqrt{(2D_n-1)\,N_n}}$$ converges in law towards $\mathcal{N}_{\R^d}(\mathbf{0},K)$, with
$$\dconv(\YEC_n,\,\mathcal{N}_{\R^d}(\mathbf{0},K)) = O\!\left(\sqrt{\frac{D_n}{N_n}}\right).$$
We have drawn in Figure \ref{fig:randomwalks3D} the projection of the two first coordinates of the random walk associated to the steps $\XEC_i$ when $d=3$, in the case of dependent and independent random vectors. As in the previous example, these two random walks do not share the same aspect and fluctuations.
\begin{center}
\begin{figure}[ht]
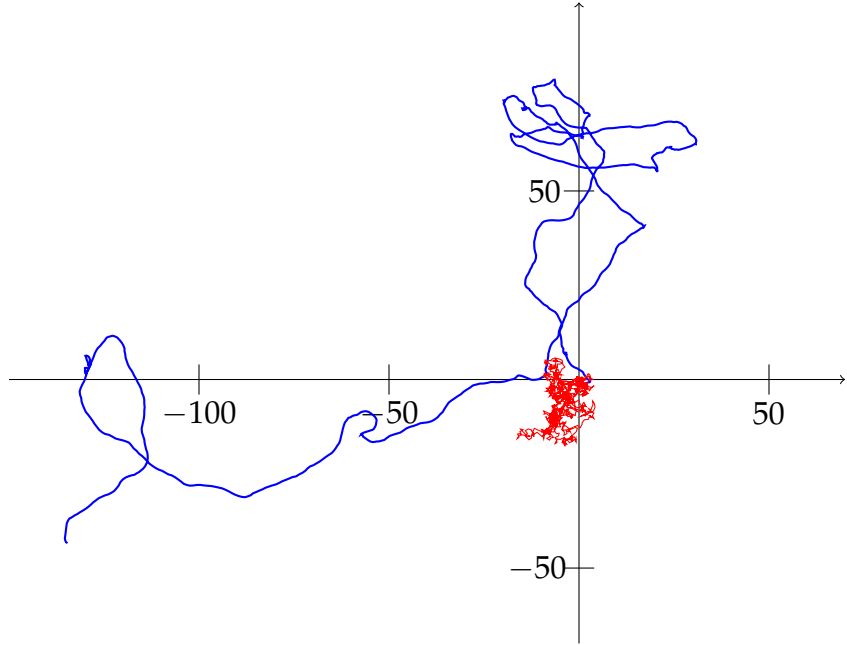


\caption{Projection of the two first coordinates of random walks associated to a sum of dependent random vectors in $\R^3$ (thick blue line, $N_n=1000$ and $D_n=40$), and to a sum of independent random vectors (thin red line, $N_n=1000$ and $D_n=1$).\label{fig:randomwalks3D}}
\end{figure}\end{center}
\end{example}

\begin{example}[Subgraph counts in random Erd\"os--R\'enyi graphs]
Let us examine an example where the theory of sparse dependency graphs applies to a sum of random vectors, but where the limiting law of the rescaled sum is a degenerate Gaussian distribution. We fix a parameter $p \in (0,1)$, and we consider the \emph{random Erd\"os--R\'enyi graph} $G_n=G(n,p)$, which is the random subgraph of the complete graph on $n$ vertices, such that the family of random variables 
$$ \left(1_{\{i,j\} \text{ is an edge of the graph}}\right)_{1 \leq i < j \leq n}$$
is a family of independent Bernoulli random variables with same parameter $p$. We have drawn in Figure \ref{fig:erdosrenyi} such a random graph with parameters $n=30$ and $p=0.15$. In \cite[Section 10]{FMN16}, it has been shown that if $H$ is a fixed graph (motive), then the number of such motives $H$ appearing in $G_n$ converges after renormalisation in the mod-Gaussian sense; see also \cite{FMN20} for a generalisation of this result to graph subcounts in graphon models. Theorem \ref{thm:upperboundcumulant} will enable us to deal with several motives at once, and to understand the fluctuations of a random vector of subgraph counts.
\begin{center}
\begin{figure}[ht]
\includegraphics[scale=0.7]{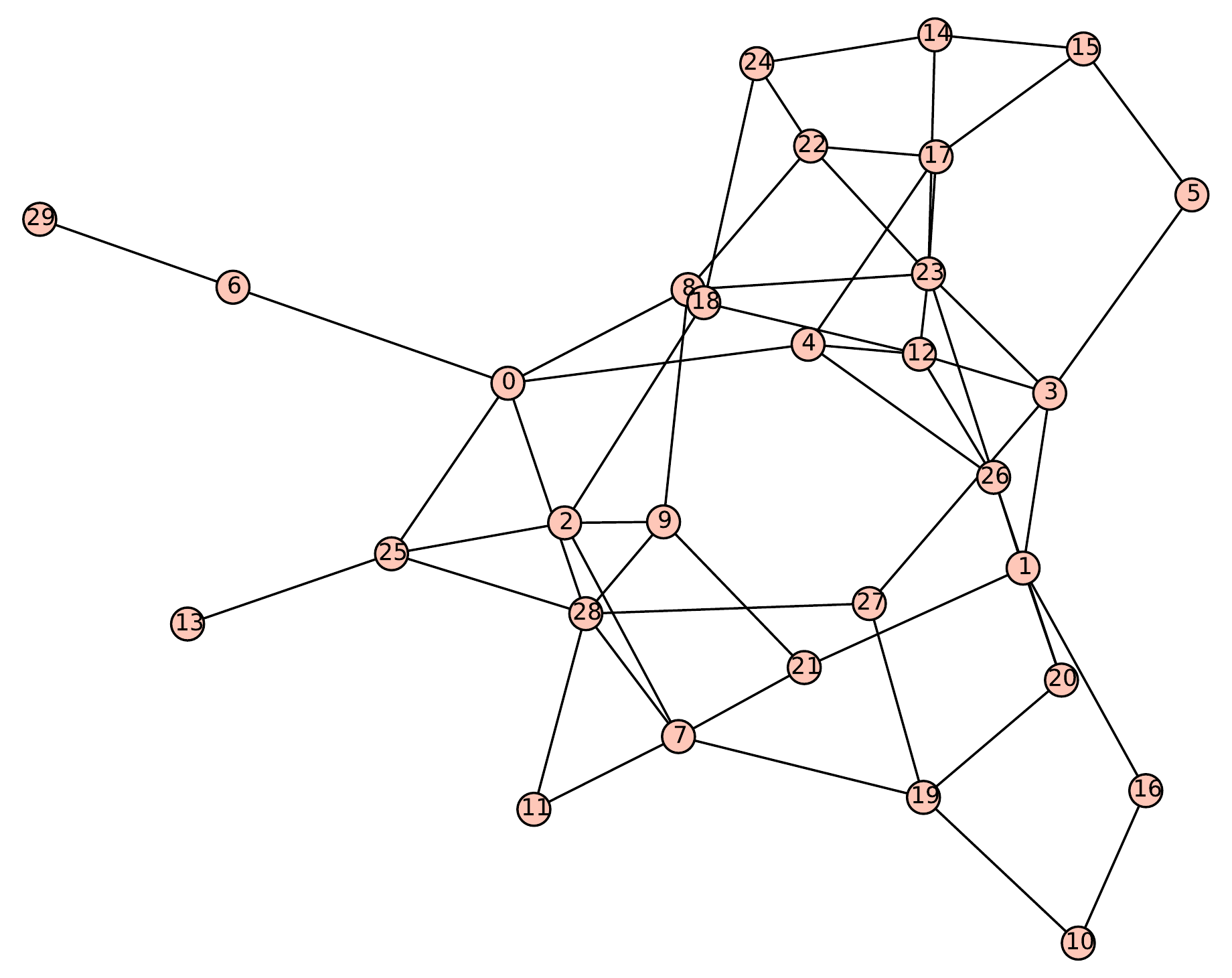}
\caption{Random Erd\"os-R\'enyi graph with $n=30$ vertices and parameter $p=0.15$.\label{fig:erdosrenyi}}
\end{figure}
\end{center}

We start by recalling the theory of subgraph counts in random graphs. If $G$ is a (undirected, simple) graph, we denote $V_G$ and $E_G$ its sets of vertices and of edges. If $H$ and $G$ are two graphs, the \emph{number of embeddings} of $H$ in $G$, denoted $I(H,G)$, is the number of injective maps $i:V_H \to V_G$ such that, if $\{v_1,v_2\} \in E_H$, then $\{i(v_1),i(v_2)\} \in E_G$. Fix a graph $H$ with vertex set $\lle 1,k\rre$. An embedding of $H$ into a graph $G_n$ with vertex set $\lle 1,n\rre$ corresponds a choice of an arrangement $A=(a_1,\ldots,a_k)$ of $k$ distinct elements of $\lle 1,n\rre$. Given such an arrangement $A$, we set
$$X_A(H,G) = \begin{cases}
	1 &\text{if the map $j\mapsto a_j$ yields an embedding of $H$ into $G$},\\
	0 &\text{otherwise}.
\end{cases}$$
With the graph of Figure \ref{fig:erdosrenyi}, $X_{(3,1,23)}(K_3,G)=1$, whereas $X_{(13,25,28)}(K_3,G)=0$. The number of embeddings of $H$ into $G$ is given by the sum 
$$I(H,G)=\sum_{A \in \mathfrak{A}(n,k)} X_A(H,G),$$ 
where $\mathfrak{A}(n,k)$ is the set of all arrangements of size $k$ in $\lle 1,n\rre$. \bigskip

Suppose now that $G=G_n$ is a random Erd\"os-R\'enyi graph of parameters $n$ and $p$. Then, assuming $k\leq n$, for every arrangement $A \in \mathfrak{A}(n,k)$, $X_A(H,G_n)$ is a random Bernoulli variable of expectation
$$\esper[X_A(H,G_n)] = p^h, \quad\text{with }h=\card\, E_H.$$
Indeed, one can write each $X_A(H,G_n)$ as the product
$$X_A(H,G_n) = \prod_{\{i,j\} \in E_H} 1_{\{a_i,a_j\} \in E_{G_n}},$$
and by assumption on the random Erd\"os-R\'enyi graph, the random variables $1_{\{a_i,a_j\} \in E_{G_n}}$ are independent Bernoulli random variables of parameter $p$. The same decomposition shows that, if $A$ and $B$ are two arrangements that do not share at least $2$ different points $a$ and $b$, then $X_A(H,G_n)$ and $X_B(H,G_n)$ can be written as products of variables falling into two independent families; therefore, they are independent. As a consequence, a dependency graph for the sum $I(H,G_n)$ is the graph: 		
\begin{itemize}
	\item whose vertices are the arrangements $A \in \mathfrak{A}(n,k)$; 
	\item with an edge between $A$ and $B$ if and only if $\card (A \cap B)\geq 2$.
\end{itemize}
The parameters of this dependency graph are the following. The number of arrangements in $\mathfrak{A}(n,k)$ is $N_n = n^{\downarrow k} = n(n-1)(n-2)\cdots (n-k+1) =n^k\,(1+O(\frac{1}{n}))$. On the other hand, if an arrangement $A$ is fixed, then there are less than
$$\binom{k}{2}^2\,2\,(n-2)(n-3)\cdots (n-k+1)$$
arrangements $B$ that share at least two points with $A$. So, one can take for parameter $$D_n = 2\,\binom{k}{2}^2\,(n-2)(n-3)\cdots (n-k+1) = \frac{k^2(k-1)^2}{2}\,n^{k-2}\,\left(1+O\!\left(\frac{1}{n}\right)\right).$$
Fix now several graphs $H\EX{1},H\EX{2},\ldots,H\EX{d}$. Notice that when dealing with the random vector $I(\mathbf{H},G_n) = (I(H\EX{1},G_n),I(H\EX{2},G_n),\ldots,I(H\EX{d},G_n))$, one can always assume without loss of generality that 
$ k = \card\, V_{H\EX{1}} = \card\,V_{H\EX{2}} = \cdots = \card\,V_{H\EX{d}}$.
Indeed, suppose that $V_H$ has cardinality $j \leq k$, and denote $H^{+k}$ the graph obtained from $H$ by adding $k-j$ disconnected vertices. Then, for $n\geq k$,
$$I(H^{+k},G_n) = (n-j)\cdots (n-k+1)\,I(H,G_n),$$ 
so the study of the fluctuations of $I(H,G_n)$ is equivalent to the study of the fluctuations of $I(H^{+k},G_n)$.  Setting $ X_A(\mathbf{H},G_n)=(X_A(H\EX{1},G_n),\ldots,X_A(H\EX{d},G_n))$, one can then write the random vector $I(\mathbf{H},G_n)$ as a sum of random vectors:
$$I(\mathbf{H},G_n) = \sum_{A \in \mathfrak{A}(n,k)} X_A(\mathbf{H},G_n).$$
Moreover, the random vectors $X_A(\mathbf{H},G_n)$ take their values in $[0,1]^d$, and they have the same dependency graph as described before. Consequently, the conclusions of Theorem \ref{thm:upperboundcumulant} hold:
$$\left|\kappa(I(H\EX{i_1},G_n),I(H\EX{i_2},G_n),\ldots,I(H\EX{i_r},G_n))\right|\leq r^{r-2}\,k^{4(r-1)}\,n^{k+(r-1)(k-2)}$$
for any choice of indices $i_1,\ldots,i_r$, since $N_n \leq n^k$ and $D_n \leq \frac{k^4}{2}\,n^{k-2}$. Unfortunately, one cannot apply Theorem \ref{thm:sumsparsegraph} to the recentred sum $\SEC_n=I(\mathbf{H},G_n)-\esper[I(\mathbf{H},G_n)]$, as the limiting covariance matrix of $\frac{\SEC_n}{n^{k-1}}$ is a \emph{degenerate} non-negative symmetric matrix. Indeed, given two arrangements $A$ and $B$, one has
\begin{align*}
&\cov(X_A(H\EX{1},G_n),X_B(H\EX{2},G_n)) \\
&= \esper\!\left[\prod_{\{i,j\} \in E_{H\EX{1}}} 1_{\{a_i,a_j\} \in E_{G_n}} \prod_{\{i,j\} \in E_{H\EX{2}}} 1_{\{b_i,b_j\} \in E_{G_n}} \right] - p^{h_1+h_2},
\end{align*}
where $h_1$ and $h_2$ are the numbers of edges of $H\EX{1}$ and $H\EX{2}$. If $A$ and $B$ do not share at least two points, one obtains $0$ as explained before. Suppose now that $A$ and $B$ share exactly two points: there are indices $\{f,g\}$ and $\{i,j\}$ such that $a_f=b_i$ and $a_g=b_j$, and all the other $a$'s are different from the other $b$'s. 
\begin{enumerate}
	\item If moreover $\{f,g\}$ and $\{i,j\}$ are edges respectively of $H\EX{1}$ and $H\EX{2}$, then the expectation of the product of Bernoulli variables is equal to $p^{h_1+h_2-1}$, since $1_{\{a_f,a_g\} \in E_{G_n}}=1_{\{b_i,b_j\} \in E_{G_n}}$. Thus, 
$$\cov(X_A(H\EX{1},G_n),X_B(H\EX{2},G_n)) = p^{h_1+h_2-1}-p^{h_1+h_2}$$
if $A\cap B = \{a_f,a_g\} = \{b_i,b_j\}$ has cardinality equal to two, and if this intersection comes from an edge of $H\EX{1}$ and an edge of $H\EX{2}$.
	\item On the other hand, if the intersection of cardinality two $A \cap B$ does not come from edges of $H\EX{1}$ and $H\EX{2}$, then the covariance is again equal to $0$.
\end{enumerate} 
The number of pair of arrangements $(A,B)$ that fulfills the first hypothesis is equal to $2h_1h_2 n^{\downarrow 2k-2}$. Thus, the pairs of arrangements $(A,B)$ with an intersection of cardinality $2$ yield a contribution
$$2(p^{-1}-1)\,h_1p^{h_1}\,h_2p^{h_2}\,n^{2k-2}\,\left(1+O\!\left(\frac{1}{n}\right)\right) $$
to $\cov(I(H\EX{1},G_n),I(H\EX{2},G_n))$. The other pairs of arrangements have an intersection of cardinality larger than $3$, hence, they yield a contribution which is a $O(n^{2k-3})$. We conclude that
\begin{align*}
\cov(I(H\EX{i},G_n),I(H\EX{j},G_n)) &= 2(p^{-1}-1)\,h_ip^{h_i}\,h_jp^{h_j}\,n^{2k-2}\,\left(1+O\!\left(\frac{1}{n}\right)\right);\\
\cov\!\left(\frac{\SEC_n}{n^{k-1}}\right) &= 2(p^{-1}-1) \mathbf{w}\mathbf{w}^t + O\!\left(\frac{1}{n}\right),
\end{align*}
where $\mathbf{w}=(h_1p^{h_1},\ldots,h_dp^{h_d})$. Thus, the limiting covariance is always a rank one symmetric matrix, and one cannot apply Theorem \ref{thm:sumsparsegraph}. These one-dimensional fluctuations are not very surprising: each fluctuation of a random number of embeddings $I(H\EX{i},G_n)$ is driven at first order by the fluctuations of the number of edges in $G_n$, and truly multi-dimensional fluctuations only occur at higher order.
In this setting, one can obtain an estimate of the speed of the convergence by projecting the random variable $\SEC_n$ to the vector line on which the Gaussian distribution with covariance matrix $\mathbf{w}\mathbf{w}^t$ is supported. Denote $S_n\EX{w} = \frac{\mathbf{w}^t\SEC_n}{\mathbf{w}^t\mathbf{w}}$, and consider the $\leb^2$-norm of $\SEC_n-S_n\EX{w}\mathbf{w}$:
\begin{align*}
\esper\!\left[\left\|\SEC_n-S_n\EX{w}\mathbf{w}\right\|^2\right] &= \esper\!\left[\|\SEC_n\|^2 - \frac{(\mathbf{w}^t\SEC_n)^2}{\|\mathbf{w}\|^2}\right] \\
&= \sum_{i=1}^d \esper[(S_n\EX{i})^2] - \frac{1}{\|\mathbf{w}\|^2} \sum_{i=1}^d\sum_{j=1}^d w\EX{i}w\EX{j}\,\esper[ S_n\EX{i} S_n\EX{j}] = O(n^{2k-3})
\end{align*}
by using the estimate $\esper[S_n\EX{i} S_n\EX{j}] = 2(p^{-1}-1)\,w\EX{i}w\EX{j} \,n^{2k-2} + O(n^{2k-3})$. Therefore,
$$\esper\!\left[\left\|\frac{\SEC_n-S_n\EX{w}\mathbf{w}}{n^{k-1}}\right\|^2\right] = O\!\left(\frac{1}{n}\right),$$
so in particular $\frac{\SEC_n-S_n\EX{w}\mathbf{w}}{n^{k-1}}$ converges in probability to $0$. Then, the projection $S_n\EX{w}$ can be shown to converge in the one-dimensional mod-Gaussian sense:
\begin{align*}
&\log \esper\!\left[\E^{\frac{z\,S_n\EX{w}}{n^{k-1}}}\right] = \log \esper\!\left[\E^{\frac{\scal{z\mathbf{w}}{\SEC_n}}{n^{k-1}\,\scal{\mathbf{w}}{\mathbf{w}}}}\right] \\
&= \sum_{r\geq 2} \frac{z^r}{\|w\|^{2r}\,r!\,n^{(k-1)r}} \sum_{i_1,\ldots,i_r=1}^d w\EX{i_1}\cdots w\EX{i_r}\, \kappa(S_n\EX{i_1},\ldots,S_n\EX{i_r})\\
&= z^2(p^{-1}-1) + O\!\left(\frac{|z|^2}{n}\right) + \sum_{r\geq 3} \frac{z^r}{\|\mathbf{w}\|^{2r}\,r!\,n^{(k-1)r}} \sum_{i_1,\ldots,i_r=1}^d w\EX{i_1}\cdots w\EX{i_r}\, \kappa(S_n\EX{i_1},\ldots,S_n\EX{i_r}).
\end{align*}
Using the bounds on joint cumulants previously described, we can write the following upper bound on the remainder of the series $R(z)$:
\begin{align*}
|w\EX{i_1}\cdots w\EX{i_r}\, \kappa(S_n\EX{i_1},\ldots,S_n\EX{i_r})| &\leq |w\EX{i_1}\cdots w\EX{i_r}|\,r^{r-2}\,k^{4(r-1)}\,n^{k+(r-1)(k-2)}; \\
\sum_{i_1,\ldots,i_r=1}^d |w\EX{i_1}\cdots w\EX{i_r}\, \kappa(S_n\EX{i_1},\ldots,S_n\EX{i_r})| &\leq (\|\mathbf{w}\|_1)^r\, r^{r-2}\,k^{4(r-1)}\,n^{k+(r-1)(k-2)}; \\
|R(z)|&\leq \frac{n^2}{k^4}\sum_{r\geq 3}  \frac{r^{r-2}}{r!} \left(\frac{|z|\,\|\mathbf{w}\|_1\,k^4}{n\,\|\mathbf{w}\|^2}\right)^r,
\end{align*}
and the last term is a $O(|z|^3/n)$ on a zone $z \in [-Dn,Dn]$ with $D >0$ fixed (depending only on $k$ and on $\|w\|$). Therefore, if $\mu_n$ is the law of $S_n\EX{w}/(n^{k-1}\,\sqrt{2(p^{-1}-1)})$ and $\nu$ is the standard one-dimensional Gaussian distribution, then
$$\widehat{\mu_n}(\zeta)-\widehat{\nu}(\zeta) = \E^{-\frac{\zeta^2}{2}}\,\left(\exp\!\left(O\!\left(\frac{|\zeta|^2+|\zeta|^3}{n}\right)\right)-1\right)$$
for $|\zeta| \leq Dn$ and a certain constant $D>0$. The arguments of \cite[Section 2]{FMN19}, which give estimates of the speed of convergence in the special case $d=1$, ensure then that
$$\dkol\!\left(\frac{S_n\EX{w}}{n^{k-1}\,\sqrt{2(p^{-1}-1)}},\,\gauss{}{0}{1}\right) = O\!\left(\frac{1}{n}\right).$$
Thus, if $\YEC_n=(Y_n\EX{1},\ldots,Y_n\EX{d})$ with 
$Y_n\EX{i} = \frac{I(H\EX{i},G_n) - n^{\downarrow k} p^{h_i}}{\sqrt{2(p^{-1}-1)}\,n^{k-1}}$, 
and if $Y_n\EX{w} = \frac{\scal{\YEC_n}{\mathbf{w}}}{\scal{\mathbf{w}}{\mathbf{w}}}$
 then:
\begin{enumerate}
	\item As $n$ goes to infinity, the difference between $\YEC_n$ and its projection $Y_n\EX{w}\mathbf{w}$ converges in probability to $0$, and more precisely,
	$$ \esper[\|\YEC_n-Y_n\EX{w}\mathbf{w}\|^2]=O\!\left(\frac{1}{n}\right).$$
	\item As $n$ goes to infinity, the random variable $Y_n\EX{w}$ converges in law towards a standard Gaussian distribution, and more precisely,
	$$\dkol(Y_n\EX{w},\,\mathcal{N}_{\R}(0,1))=O\!\left(\frac{1}{n}\right).$$
\end{enumerate}

\noindent Thus, when the limiting covariance is not a positive-definite matrix, one can usually still apply the theory of mod-Gaussian convergence to some projection of the random vectors, and on the other hand control the difference between the projections and the original random vectors.\bigskip
\end{example}

\subsection{Empirical measures of Markov chains}\label{subsec:markov} 
The interest of Definition \ref{def:methodmulticum} is that one can use it with sums of random vectors $\SEC_n = \sum_{i=1}^{N_n}\mathbf{A}_{i,n}$ even when there is no sparse dependency graph of parameters $N_n$ and $D_n$ for the random vectors $\mathbf{A}_{i,n}$ of the sums. In \cite[Section 5]{FMN19}, it was proven that the bounds on cumulants corresponding to Hypothesis \ref{hyp:MC3} are also true for:
\begin{itemize}
 	\item certain observables of models from statisticial mechanics;
 	\item the linear functionals of finite ergodic Markov chains.
 \end{itemize} 
 In this section, we explain how to use the theory of \cite[Sections 5.4 and 5.5]{FMN19} to study the fluctuations of the empirical measure of a finite ergodic Markov chain (instead of a linear form of this empirical measure). We fix a state space $\mathfrak{X} = \lle 1,d\rre$, and we consider a transition matrix $P$ on $\mathfrak{X}$, which is assumed ergodic (irreducible and aperiodic), with invariant probability measure $\boldsymbol\pi=(\pi(1),\pi(2),\ldots,\pi(d))$. Thus, $\pi P = \pi$, and if $(X_n)_{n \in \N}$ is a Markov chain on $\mathfrak{X}$ with transition matrix $P$ and initial distribution $\proba[X_0=i]=\pi(i)$, then $\proba[X_n=i]=\pi(i)$ for any $n \in \N$. We are interested in the \emph{empirical measure} of $(X_n)_{n \in \N}$, which is the random vector in $\R^d$
 $$\boldsymbol\pi_n = \frac{1}{n} \left(\sum_{t=1}^n \delta_{X_t}\right).$$
 The ergodic theorem ensures that $\boldsymbol\pi_n \to \boldsymbol\pi$ almost surely. Set $\SEC_n = n(\boldsymbol\pi_n-\boldsymbol\pi)$. We define a constant
$$\theta_P = \sqrt{\max\{|z|\,|\,z\neq 1,\,\,z\text{ eigenvalue of }PD^{-1}P^tD\}},$$
where $D = \mathrm{diag}(\boldsymbol\pi)$. This constant $\theta_P$ is strictly smaller than $1$, and when $P$ is reversible it is just the module of the second largest eigenvalue of $P$. It has been proven \cite{FMN19} that
 $$\left|\kappa(1_{X_{t_1}=i_1},\ldots, 1_{X_{t_r}=i_r})\right| \leq 2^{r-1}\sum_{T \in \mathrm{ST}(\lle 1,r\rre)} w(T),$$
 where the weight of a spanning tree on $r$ vertices is the products of the weights
 $w(i,j) = (\theta_P)^{|t_j-t_i|}$
 of its edges $(i,j) \in E_T$. As a consequence, for any $r \geq 2$ and any choice of indices $i_1,\ldots,i_r$,
 \begin{align*}
 \left|\kappa(S_n\EX{i_1},\ldots,S_n\EX{i_r})\right| &\leq 2^{r-1} \sum_{T \in \mathrm{ST}(\lle 1,r\rre)} \sum_{t_1,\ldots,t_r = 1}^n \prod_{(i,j)\in E_T} (\theta_P)^{|t_j-t_i|} \leq n\,r^{r-2}\,\left(2\,\frac{1+\theta_P}{1-\theta_P}\right)^{r-1},
 \end{align*}
 see \cite[Theorem 55]{FMN19}. Therefore, Hypothesis \ref{hyp:MC3} is satisfied with $N_n=n$, $D_n = \frac{1+\theta_P}{1-\theta_P}$ and $A=1$. Let us then check the other hypotheses of Definition \ref{def:methodmulticum}:
 \begin{enumerate}
 	\item By construction, $\esper[\SEC_n]=0$ (here and in the sequel we assume that $(X_n)_{n \in \N}$ is stationary, hence has initial distribution $\boldsymbol\pi$).
 	\item We compute
 	\begin{align*}
 	\cov(S_n\EX{i},S_n\EX{j}) &= \sum_{a,b=1}^n \proba[X_a=i\text{ and }X_b=j] - \proba[X_a=i]\,\proba[X_b=j]\\
 	&= n\,\delta_{i,j}\,\pi(i)(1-\pi(i)) + 2\,\sum_{k=1}^{n-1} (n-k)\pi(i)(P^{k}(i,j)-\pi(j)).
 	\end{align*}
 	The limit of $\frac{\cov(S_n\EX{i},S_n\EX{j})}{n}$ is
 	$$K_{i,j} = \pi(i)(\delta_{i,j}-\pi(j)) + 2 \pi(i) \sum_{k=1}^\infty R^k(i,j),$$
 	where $R = P-P^\infty$; $P^\infty(i,j) = \pi(j)$; and $R$ is a matrix with all its eigenvalues strictly smaller than $1$. Notice that $K$ has at most rank $(d-1)$: indeed, $\sum_{i=1}^d K_{i,j}=0$, this being a consequence of $$\sum_{i=1}^d S_n\EX{i} = 0\quad\text{almost surely}.$$
 	Therefore, we shall consider the random vectors $\SEC_n$ as elements of the $(d-1)$-dim\-ens\-ional space
 	$$H = \left\{\xec \in \R^d \,\,\big|\,\,\sum_{i=1}^d x\EX{i}=0 \right\}, $$ 
 	and the covariance matrix as a symmetric operator on this space. By the discussion of \cite[Proposition 59]{FMN19}, if the Markov chain is reversible, then $K$ has indeed rank $(d-1)$ and we shall be able to describe precisely the fluctuations of $(\SEC_n)_{n \in \N}$ in $H$.\medskip
 	
 	\noindent Notice that the difference between $\frac{\cov(\SEC_n)}{n}$ and $K$ is a $O(\frac{1}{n})$; therefore, Condition \ref{hyp:MC2prime} is satisfied.
 	\item Finally, Hypotheses \ref{hyp:MC4} and \ref{hyp:MC5} are discussed at the very end of \cite{FMN19}, and they are satisfied with $v=3$.
 \end{enumerate}
 We conclude:
 \begin{theorem}[Empirical measures of reversible Markov chains]
 Let $(X_n)_{n \in \N}$ be an ergodic and reversible Markov chain on a finite state space $\mathfrak{X}=\lle 1,d\rre$. We denote $\boldsymbol\pi$ the stationary measure and $\boldsymbol\pi_n$ the empirical measure; their difference is a random vector in the hyperplane $H$. If
 $$\XEC_n = \frac{\boldsymbol\pi_n - \boldsymbol\pi}{n^{1/3}}\qquad;\qquad \YEC_n = \frac{\boldsymbol\pi_n - \boldsymbol\pi}{\sqrt{n}},$$
 and
 $$K_{i,j} = \pi(i)(\delta_{i,j}-\pi(j)) + 2 \pi(i) \sum_{k=1}^\infty (P^k(i,j)-\pi(j)),$$
 then:
 \begin{enumerate}
 	\item The sequence $(\XEC_n)_{n \in \N}$ is mod-Gaussian convergent on $H$, with parameters $n^{1/3}K$ and limit
 	$$\psi(\zec) = \exp\left(\frac{1}{6} \sum_{i,j,k=1}^d \left(\sum_{a,b\in \Z}\kappa(1_{X_0}=i,1_{X_a}=j,1_{X_b}=k)\right)z\EX{i}z\EX{j}z\EX{k}\right).$$

 	\item We have a convergence in law $\YEC_n \rightharpoonup \mathcal{N}_H(\mathbf{0},K)$, with a convex distance between the two distributions on $H$ that is a $O(\frac{1}{\sqrt{n}})$. The constant in the $O(\cdot)$ only depends on $d$ and on the constant $\theta_P$ previously introduced.
 \end{enumerate}
 \end{theorem}
\clearpage

\printbibliography

\end{document}